\documentclass[12pt]{report}

\let\ssection=\section
\renewcommand{\section}{\setcounter{equation}{0}\ssection}

\usepackage{graphicx}
\usepackage{hyperref}
\usepackage{graphicx}
\usepackage{amsfonts}
\usepackage{amssymb}
\usepackage{amsthm}
\usepackage{newlfont}
\usepackage{amsmath}
\usepackage{fancyhdr}
\usepackage{t1enc}
\usepackage[utf8]{inputenc}
\usepackage[french]{babel}
\usepackage[T1]{fontenc}
\usepackage{hyperref}
\newtheorem{thm}{Théorème}[section]
\newtheorem{pro}{Proposition}[section]
\newtheorem{lem}{Lemme}[section]
\newtheorem{4}{Propriétés}[section]

\newtheorem{exem}{Exemple}[section]
\newtheorem{exems}{Exemples}[section]
\newtheorem{rem}{Remarque}[section]
\newtheorem{rems}{Remarques}[section]
\newtheorem{dfn}{Définition}[section]
\newtheorem{dfns}{Définitions}[section]
\newtheorem{cor}{Corollaire}[section]

{}

\usepackage{amssymb}
\usepackage{relsize}
\usepackage{layout}
\usepackage{fancyhdr}
\newcommand{\ds}{\displaystyle}

\begin{document}

\title{ Espaces d'Orlicz, Orlicz-Sobolev et application \\
       aux E-D-P.}
\author{Sabri Bahrouni et Hichem Ounaies\\
Faculté des sciences de Monastir,  D\'epartement de Math\'ematiques}
\date{}
\maketitle

\newpage
\tableofcontents

\newpage
\chapter*{Introduction}
Les espaces de Lebesgue $L^{p}(\Omega)$ sont définis en utilisant la fonction puissance $|x|^{p}$, où $1\leq p<\infty$. Ces espaces sont des espaces de Banach, et leurs propriétés topologiques et métriques dépendent principalement de la fonction puissance $x\longmapsto |x|^{p}$ présente dans leurs définitions.

Les espaces d'Orlicz $L^{G}(\Omega)$ sont une extension naturelle des espaces de Lebesgue $L^{p}(\Omega)$ et sont définis à partir d'une fonction $G$ appelée $N$-fonction, qui généralise les fonctions puissances. Les propriétés topologiques et métriques de ces espaces sont étroitement liées à la croissance de la $N$-fonction qui les définit. Par conséquent, une étude approfondie des propriétés de croissance de ces fonctions, ainsi que de leurs conjuguées, est présentée dans le premier chapitre. Il convient de noter que le vocabulaire utilisé pour ces fonctions peut varier d'un auteur à l'autre, parfois en les appelant fonctions d'Orlicz et en désignant les fonctions d'Orlicz comme des $N$-fonctions.

La théorie des espaces d'Orlicz a été introduite par Orlicz en 1930 \cite{29} et connaît actuellement un développement important en raison de ses applications dans divers domaines de l'analyse mathématique, notamment l'analyse non linéaire.

Les espaces d'Orlicz et leurs propriétés ont initialement été étudiés dans le cadre d'une $N$-fonction et de la mesure de Lebesgue \cite{5}, puis dans le cas d'un espace mesuré quelconque \cite{26}, \cite{11}, \cite{12}.

Dans une tentative de généraliser les espaces de Lebesgue, on peut envisager l'espace $K^{G}(\Omega)$ défini comme suit:
$$K^{G}(\Omega)=\bigg{\{} u:\Omega\rightarrow \mathbb{R},\ \text{mesurable} \ /\ \rho(u;G)=\int_{\Omega} G(|u(x)|)dx < \infty\bigg{\}}.$$
Cependant, malheureusement, la classe $K^{G}(\Omega)$ n'est pas toujours un espace vectoriel. C'est pourquoi nous définissons l'espace d'Orlicz $L^{G}(\Omega)$ comme le plus petit sous-espace vectoriel contenant $K^{G}(\Omega)$.
\\

Les espaces d'Orlicz sont caractérisés par différentes normes. Dans les années trente, Orlicz a introduit une norme particulière, appelée norme d'Orlicz, définie par l'expression suivante :

$$\|u\|_{G}=\sup\bigg{\{}\int_{\Omega}|u(t)v(t)|dt:\ v\in K^{G^{*}}(\Omega),\ \rho(v\ ;G^{*})\leq 1\bigg{\}},$$

où $G^{*}$ représente la conjuguée de la fonction $G$ et est définie comme suit :
 $$G^{*}(t)=\ds\sup_{s\geq0}\{st-G(s)\}.$$

Par la suite, Nakano \cite{24} et Luxemburg \cite{22} ont introduit une autre norme, parfois appelée norme de Luxemburg-Nakano, mais plus communément connue sous le nom de norme de Luxemburg. Cette norme est définie comme la fonction de Minkowski (ou jauge) de la boule unité pour la modulaire d'Orlicz. En d'autres termes :

$$\|u\|_{(G)}=\inf\bigg{\{}\lambda>0\ :\ \rho\left(\frac{u}{\lambda}\ ;G\right)\leq1\bigg{\}}.$$

Approximativement à la même époque, I. Amemiya (\cite{30}, page 218) a proposé une autre norme, appelée norme d'Amemiya, définie comme suit :

$$\|u\|_{G}^{A}=\ds\inf\bigg{\{}\frac{1}{k}\left[1+\rho(ku;G)\right],\ k>0\bigg{\}}.$$

Dans l'article \cite{27}, les auteurs ont démontré l'égalité entre la norme d'Amemiya et celle d'Orlicz.

Pour l'étude de la réflexivité, de la séparabilité et du dual de l'espace d'Orlicz, un espace noté $E^{G}$ est introduit. Cet espace est défini comme suit :
$$E^{G}(\Omega):=\{u\in L^{G}(\Omega);\ \lambda u\in K^{G}(\Omega)\ \text{pour tout}\ \lambda>0\},$$

sous la condition appelée $\triangle_{2}$ à l'infini imposée sur $G$, c'est-à-dire qu'il existe $k>0$ et $T\geq0$ tels que :
 $$G(2t)\leq kG(t),\ \forall t\geq T.$$

On définit l'espace d'Orlicz-Sobolev $W^{m,G}(\Omega)$, en remplacent dans la définition de l'espace de Sobolev classique $W^{m,p}(\Omega)$; l'espace $L^{p}(\Omega)$ par l'espace d'Orlicz $L^{G}(\Omega)$. Dans l'article \cite{4}, Donaldson et Trudinger ont montré que les espaces d'Orlicz - Sobolev ont les m\^{e}mes propriétées que les espaces de Sobolev classiques. En particulier, ils ont établi des injections continues et des injections compactes.\\

Les espaces d'Orlicz-Sobolev jouent un r\^{o}le important dans de nombreux domaines des mathématiques, comme la théorie de l'approximation, les équations aux dérivées partielles, le calcul des variations, théorie du potentiel non linéaire, théorie des applications quasi-conformes,
géométrie différentielle et théorie des probabilités. L’étude des équations elliptiques non linéaires
à opérateurs de type homogène sont basés sur la théorie des espaces de Sobolev $W^{m,p}(\Omega)$.
Dans le cas d'opérateurs différentiels non homogènes, le cadre naturel de cette étude est les espaces d'Orlicz-Sobolev. On peut citer les travaux de  M. Mihăilescu, V. Rădulescu, D.Repovs \cite{18}, \cite{19}, \cite{20}, \cite{21}, \cite{31}.\\

Notre travail consiste en premier lieu à présenter les espaces d'Orlicz, d'Orlicz-Sobolev et donner leurs propriétés. Une attention particulière est accordée aux injections.\\

En second lieu, on s'intéresse à une application des espaces d'Orlicz-Sobolev sur des équations de dérivées partielles. On considère le problème suivant $$(P)\quad \begin{cases}
                                                            -div(a(|\nabla u|)\nabla u)=\lambda f(x,u) & \mbox{si}\ x\in\Omega  \\
                                                            u=0 & \mbox{si}\ x\in\partial\Omega \\
                                                            u\geq0 & \mbox{si}\ x\in\Omega.
                                                          \end{cases}$$
  où $\Omega$ est un domaine borné régulier de $\mathbb{R}^{N}$, $\lambda$ est un paramètre positif, $f$ est une fonction de Carathéodory et $a$ est une application définie sur $]0,+\infty[$.
  En appliquant le théorème de Mountain pass et la méthode du minimum global, on montre que le problème $(P)$ admet au moins deux solutions dans $W^{1,G}_{0}(\Omega)$.

\chapter{$N$-Fonction et classe d'Orlicz}

\section{$N$-Fonction}

\subsection{Représentation intégrale d'une fonction convexe}

Dans cette sous-section $I$ désigne un intervalle de $\mathbb{R}$.

\begin{dfn}
  On dit qu'une fonction $f:I\rightarrow\mathbb{R}$ est convexe si pour tout $a,c\in I$ et pour tout $t\in[0,1]$ on a
  $$f(ta+(1-t)c)\leq tf(a)+(1-t)f(c).$$
\end{dfn}

\begin{rem}
  Si la  fonction $f:I\rightarrow\mathbb{R}$ est convexe, alors pour tout $a,b,c\in I$ tels que $a<b<c$ on a
  \begin{equation}\label{65}
    \frac{f(b)-f(a)}{b-a}\leq\frac{f(c)-f(a)}{c-a}\leq\frac{f(c)-f(b)}{c-b}.
  \end{equation}
  Il suffit de considérer $t=\frac{b-c}{a-c}$ de telle sorte que $1-t=\frac{a-b}{a-c}$ et $b=ta+(1-t)c$.
\end{rem}

\begin{thm}\label{thm14}
  Soit $f:I\rightarrow\mathbb{R}$ une fonction convexe. Alors $f$ est lipschitzienne sur tout intervalle compact $[a,b]\subset int I$. En particulier $f$ est continue sur $int I$.
\end{thm}

\begin{proof}
  Comme $[a,b]\subset int I$, il existe $\varepsilon>0$ tel que $[a-\varepsilon,b+\varepsilon]\subset I$. Notons $m$ et $M$ les bornes inférieures et supérieures de $f$ sur $[a-\varepsilon,b+\varepsilon]$. Pour $x,y\in[a,b]$ avec $x\neq y$ on pose
  $$z=y+\varepsilon\frac{y-x}{|y-x|}\ \text{et}\ \lambda=\frac{|y-x|}{|y-x|+\varepsilon}\in]0,1[.$$
  On a $z\in[a-\varepsilon,b+\varepsilon]$ et $y=\lambda z+(1-\lambda)x$. Alors
  $$f(y)\leq\lambda f(z)+(1-\lambda)f(x)=\lambda(f(z)-f(x))+f(x)$$
  donc $$f(y)-f(x)\leq\frac{|y-x|}{|y-x|+\varepsilon}(M-m)\leq\frac{M-m}{\varepsilon}|y-x|.$$
  \'{E}changeant $x$ et $y$ on obtient $$|f(y)-f(x)|\leq\frac{M-m}{\varepsilon}|y-x|.$$
  La continuité de $f$ sur $int I$ découle du fait que tout $x\in int I$, $x\in[a,b]\subset I$.
\end{proof}

\begin{thm}\label{thm15}
   Soit $f:I\rightarrow\mathbb{R}$ une fonction convexe. Alors les dérivées à droite $f^{'}_{d}(x)$ et à gauche $f^{'}_{g}(x)$ de $f$ existent et sont finies, pour tout $x\in int I$. De plus les fonctions $f^{'}_{d}$ et $f^{'}_{g}$ sont croissantes sur $int I$.
\end{thm}

\begin{proof}
  Soit $x_{1}\in int I$ et soient $x_{0},x_{2},x_{3},x_{4}\in I$ avec $x_{0}<x_{1}<x_{2}<x_{3}<x_{4}$. Utilisant \eqref{65} on a
  \begin{equation}\label{66}
    \frac{f(x_{1})-f(x_{0})}{x_{1}-x_{0}}\leq\frac{f(x_{2})-f(x_{1})}{x_{2}-x_{1}}\leq\frac{f(x_{3})-f(x_{1})}{x_{3}-x_{1}}\leq\frac{f(x_{4})-f(x_{3})}{x_{4}-x_{3}}.
  \end{equation}
  On déduit que la fonction $$\phi(x)=\frac{f(x)-f(x_{1})}{x-x_{1}}$$ est croissante et minorée sur $]x_{1},+\infty[\cap I$. Il en résulte que $f^{'}_{d}(x_{1})=\ds\lim_{x\rightarrow x_{1}^{+}}\phi(x)$ existe et que
  $$f^{'}_{d}(x_{1})=\ds\lim_{x\rightarrow x_{1}^{+}}\phi(x)=\ds\inf_{x\in]x_{1},+\infty[\cap I}\phi(x).$$
  On a $\phi$ est décroissante et majorée sur $]-\infty,x_{1}[\cap I$, donc  $f^{'}_{g}(x_{1})=\ds\lim_{x\rightarrow x_{1}^{-}}\phi(x)$ existe et
  $$f^{'}_{g}(x_{1})=\ds\lim_{x\rightarrow x_{1}^{-}}\phi(x)=\ds\sup_{]-\infty,x_{1}[\cap I}\phi(x).$$
  Faisont tendre $x_{2}$ vers $x_{1}^{+}$ dans \eqref{66} on a $$\frac{f(x_{1})-f(x_{0})}{x_{1}-x_{0}}\leq f^{'}_{d}(x_{1})$$
  si $x_{0}$ tend vers $x_{1}^{-}$, alors $$f^{'}_{g}(x_{1})\leq f^{'}_{d}(x_{1}).$$
 Si $x_{1}\rightarrow x_{0}^{+}$ dans \eqref{66}, on obtient $$f^{'}_{d}(x_{0})\leq\frac{f(x_{2})-f(x_{1})}{x_{2}-x_{1}};$$
 si $x_{2}\rightarrow x_{1}^{+}$ dans \eqref{66}, on obtient $$f^{'}_{d}(x_{0})\leq f^{'}_{d}(x_{1}).$$
   $f^{'}_{g}$ est croissante se démontre de manière analogue.
\end{proof}

\begin{thm}\label{thm16}
  Soit $f:I\rightarrow\mathbb{R}$ une fonction convexe. Alors pour tout $c$, $x\in]a,b[$, $a,b\in\mathbb{R}$, on a $$f(x)-f(c)=\int_{c}^{x}f^{'}_{d}(t)dt=\int_{c}^{x}f^{'}_{g}(t)dt.$$
\end{thm}

\begin{proof}
  Notons que les fonctions croissantes $f^{'}_{d}$ et $f^{'}_{g}$ sont Riemann intégrables sur tout intervalle compact. Soit $[c,b]\subset I$ et $x\in [c,b] $. Considérons une subdivision $$c=x_{0}<x_{1}<...<x_{n}=x,\ \text{avec}\ x_{k+1}-x_{k}=\frac{1}{n}$$ Pour tout $k\in[1,n]$ on a
  $$f^{'}_{g}(x_{k-1})\leq f^{'}_{d}(x_{k-1})\leq\frac{f(x_{k})-f(x_{k-1})}{x_{k}-x_{k-1}}\leq f^{'}_{g}(x_{k})\leq f^{'}_{d}(x_{k})$$
  d'où $$\sum_{k=1}^{n}f^{'}_{d}(x_{k-1})(x_{k}-x_{k-1})\leq\sum_{k=1}^{n}f(x_{k})-f(x_{k-1})\leq\sum_{k=1}^{n}f^{'}_{d}(x_{k})(x_{k}-x_{k-1}).$$
   $$|x_{k+1}-x_{k}|=\frac{1}{n}\rightarrow0,\ n\rightarrow+\infty,$$
  on obtient
  $$\int_{c}^{x}f^{'}_{d}(t)dt\leq f(x)-f(c)\leq\int_{c}^{x}f^{'}_{d}(t)dt$$
  d'où le résultat (m\^{e}me démonstration pour $f^{'}_{g}$).
\end{proof}

\subsection{Définition et propriétés de $N$-Fonction}

Une généralisation de la fonction définie par $G(t)=|t|^{p}$, est donnée par la classe des fonctions suivantes.

\begin{dfn}
Soit $g$ une fonction réelle définie sur $[0,+\infty[$ vérifiant les propriétés suivantes:
 \begin{enumerate}
   \item[(a)] $g(0)=0$, $\ds\lim_{t\rightarrow\infty}g(t)=\infty$;
   \item[(b)] $g(t)>0$ si $t>0$;
   \item[(c)] $g$ est croissante;
   \item[(d)] $g$ est continue à droite.
 \end{enumerate}
 La fonction $G$ définie sur $[0,+\infty[$ par
   \begin{equation}\label{1}
   G(t)=\int_{0}^{t}g(\tau)d\tau
   \end{equation}
 est appelée une $N$-Fonction.
\end{dfn}

Les propriétés des $N$-Fonctions sont données dans le lemme suivant.

\begin{lem}
 Une $N$-Fonction $G$ est continue, strictement croissante, et convexe sur $[0,+\infty[$, de plus,
 \begin{equation}\label{2}
   G(0)=0,\ G(\infty):=\ds\lim_{t\rightarrow\infty}G(t)=\infty,
 \end{equation}
 \begin{equation}\label{3}
   \ds\lim_{t\rightarrow 0^{+}}\frac{G(t)}{t}=0,
 \end{equation}
 \begin{equation}\label{4}
   \ds\lim_{t\rightarrow\infty}\frac{G(t)}{t}=\infty,
 \end{equation}
 \begin{equation}\label{5}
   G(\alpha t)\leq\alpha G(t),\ \forall\alpha\in[0,1]\ \text{et}\ t\geq0,
 \end{equation}
 et
 \begin{equation}\label{6}
     G(\beta t)\geq\beta G(t),\ \forall\beta>1\ \text{et}\ t\geq0.
 \end{equation}
\end{lem}

\begin{proof}
  D’après la représentation \eqref{1}, chaque $N$-Fonction est continue et strictement croissante. Examinons la convexitée: soient $\lambda\in]0,1[$ et $0\leq s\leq t$,
  \begin{align*}
   G(\lambda s+(1-\lambda)t)&=\int_{0}^{\lambda s+(1-\lambda)t}g(r)dr\\
   &=\int_{0}^{s}g(r)dr+\int_{s}^{\lambda s+(1-\lambda)t}g(r)dr\\
   &=\lambda\int_{0}^{s}g(r)dr+(1-\lambda)\int_{0}^{s}g(r)dr+\int_{s}^{\lambda s+(1-\lambda)t}g(r)dr.
  \end{align*}
  Comme $g$ est croissante et continue à droite on a $$\int_{s}^{\lambda s+(1-\lambda)t}g(r)dr\leq(1-\lambda)(t-s)g(\lambda s+(1-\lambda)t)$$ et
  $$\int_{\lambda s+(1-\lambda)t}^{t}g(r)dr\geq \lambda(t-s)g(\lambda s+(1-\lambda)t),$$
  donc $$\lambda\int_{s}^{\lambda s+(1-\lambda)t}g(r)dr\leq(1-\lambda)\int_{\lambda s+(1-\lambda)t}^{t}g(r)dr,$$et
  \begin{align*}
  \int_{s}^{\lambda s+(1-\lambda)t}g(r)dr&=\lambda\int_{s}^{\lambda s+(1-\lambda)t}g(r)dr+(1-\lambda)\int_{s}^{\lambda s+(1-\lambda)t}g(r)dr\\
  &\leq(1-\lambda)\int_{s}^{t}g(r)dr.
  \end{align*}
  D'où \begin{align*}
          G(\lambda s+(1-\lambda)t)&\leq\lambda\int_{0}^{s}g(r)dr+(1-\lambda)\int_{s}^{t}g(r)dr+(1-\lambda)\int_{0}^{s}g(r)dr\\
          &\leq\lambda\int_{0}^{s}g(r)dr+(1-\lambda)\int_{0}^{t}g(r)dr\\
          &=\lambda G(s)+(1-\lambda)G(t).
       \end{align*}
  De plus, $$\ds\lim_{t\rightarrow 0^{+}}\frac{G(t)}{t}=\ds\lim_{t\rightarrow 0^{+}}\frac{1}{t}\int_{0}^{t}g(s)ds=g(0)=0.$$ d'où \eqref{3}.\\Pour $t>0$ on a $$\frac{G(t)}{t}=\frac{1}{t}\int_{0}^{t}g(s)ds\geq\frac{1}{t}\int_{\frac{t}{2}}^{t}g(s)ds\geq\frac{1}{t}\frac{t}{2}g(\frac{t}{2})=\frac{1}{2}g(\frac{t}{2})$$
  et puisque $\ds\lim_{t\rightarrow\infty}g(t)=\infty$ on déduit \eqref{4}.\\
  Pour $s,t\in\geq0$ et $\forall\alpha\in[0,1]$, $$G((1-\alpha)s+\alpha t)\leq (1-\alpha)G(s)+\alpha G(t),\ 0\leq\alpha\leq1,$$
  prenons $s=0$ on obtient \eqref{5}.\\
  Soit $\beta>1$, $$G\bigg{(}\frac{1}{\beta}\beta s\bigg{)}\leq \frac{1}{\beta}G(\beta s),$$ d'où \eqref{6}.
\end{proof}

Nous pouvons définir la fonction $G(t)$ d'une manière équivalente.

\begin{thm}
$G$ est une $N$-Fonction si et seulement si $G$ est une fonction positive, continue, convexe vérifiant  \eqref{2},\eqref{3} et \eqref{4}.
\end{thm}

\begin{proof}
Soit $G$ une fonction continue, convexe vérifiant  \eqref{2},\eqref{3} et \eqref{4}. D'après le théorème \ref{thm16}, $G$ admet la représentation \eqref{1} où $g=G_{d}^{'}$ est croissante et continue à droite. Comme $g(u)\geq\frac{G(u)}{u}$ pour $u>0$, alors $g(u)>0$ et en vertu de \eqref{4}, $\ds\lim_{u\rightarrow \infty}g(u)=\infty$.\\
D'autre part, si $u>0$, $$G(2u)=\int_{0}^{2u}g(t)dt>\int_{u}^{2u}g(t)dt>ug(u),$$ donc $g(u)<\frac{G(2u)}{u}$. Ainsi, en vertu de \eqref{3} on a
$g(0)=\ds\lim_{u\rightarrow0}g(u)=0$.
\end{proof}

\begin{exems}
\begin{enumerate}
  \item  $G_{1}(t)=t^{p}$, $1<p<\infty$.
  \item  $G_{2}(t)=e^{t}-t-1$.
  \item $G_{3}(t)=e^{t^{p}}-1$, $1<p<\infty$.
  \item $G_{4}(t)=(1+t)log(1+t)-t$, sont des $N$-Fonctions.
\end{enumerate}
\end{exems}

\begin{pro}
 \begin{enumerate}
   \item Si $G_{1}$ et $G_{2}$ sont deux $N$-fonctions, alors $G_{1}\circ G_{2}$ est une $N$-fonction.
   \item Si $G_{1},...,G_{m}$ sont des $N$-fonctions et $a_{1},...,a_{m}\geq 0$, alors $G=\sum_{i=1}^{m}a_{i}G_{i}$ est une $N$-fonction.
 \end{enumerate}
\end{pro}

\begin{proof}
  $\mathbf{1)}$ Comme $G_{2}$ est croissante et $G_{1}$,$G_{2}$ sont convexe alors $G_{1}\circ G_{2}$ est convexe, de plus
  $$\ds\lim_{t\rightarrow+\infty}\frac{G_{1}\circ G_{2}(t)}{t}=\ds\lim_{t\rightarrow+\infty}\frac{G_{2}(t)}{t}\frac{G_{1}\circ G_{2}(t)}{G_{2}(t)}=+\infty,$$ et  $$\ds\lim_{t\rightarrow0}\frac{G_{1}\circ G_{2}(t)}{t}=\ds\lim_{t\rightarrow0}\frac{G_{2}(t)}{t}\frac{G_{1}\circ G_{2}(t)}{G_{2}(t)}=0,$$ donc $G_{1}\circ G_{2}$ est une $N$-fonction.\\
  $\mathbf{2)}$ On sait que la somme de deux fonctions convexes est convexe donc $G=\sum_{i=1}^{m}a_{i}G_{i}$ est convexe, et
  $$\ds\lim_{t\rightarrow+\infty}\frac{G(t)}{t}=\ds\lim_{t\rightarrow+\infty}\sum_{i=1}^{m}a_{i}\frac{G_{i}(t)}{t}=+\infty,$$
   $$\ds\lim_{t\rightarrow0}\frac{G(t)}{t}=\ds\lim_{t\rightarrow0}\sum_{i=1}^{m}a_{i}\frac{G_{i}(t)}{t}=0$$
   donc $G$ est une $N$-Fonction.\\
\end{proof}

\subsection{Fonction complémentaire ou conjugué de $N$-Fonction}

 Soit $G$ une $N$-Fonction. On associe à $G$ une autre fonction convexe $G^{*}:\mathbb{R}^{+}\rightarrow\mathbb{R}^{+}$ qui admet les m\^{e}mes propriétés que $G$.
\begin{dfn}
Soient $G^{*}:\mathbb{R}^{+}\rightarrow\mathbb{R}^{+}$ et $g^{*}:\mathbb{R}^{+}\rightarrow\mathbb{R}^{+}$ définies par
\begin{equation}\label{7}g^{*}(s)=\ds\sup_{g(t)\leq s}t\end{equation}et\begin{equation}\label{8} G^{*}(s)=\int_{0}^{s}g^{*}(\sigma)d\sigma,\ s\in[0,+\infty[.\end{equation}La fonction $G^{*}$ est appelée fonction compl\'{e}mentaire  ou conjugué de $G$ et on dit que $(G,G^{*})$ est un couple de $N$-Fonctions.
\end{dfn}
\begin{pro}\label{pro11}
\begin{enumerate}
\item La fonction $g^{*}$ possède les m\^{e}mes propriétés que la fonction $g$; elle est positive, continue à droite, croissante sur $]0,+\infty[$ et vérifiant \begin{equation}\label{62}
                                 g^{*}(0)=0,\ \ds\lim_{s\rightarrow+\infty}g^{*}(s)=+\infty.
                               \end{equation}
\item Si $g$ est continue et strictement croissante sur $[0,+\infty[$ alors $g^{*}$ est l'inverse de $g$.
\end{enumerate}
\end{pro}
\begin{proof}
$\mathbf{(1)}$ Soit $0\leq u\leq v,$ $$g^{*}(u)=\ds\sup_{g(t)\leq u}t\leq \ds\sup_{g(t)\leq v}t=g^{*}(v),$$ donc $g^{*}$ est croissante.\\
 $g^{*}(0)=\sup\{t,\ g(t)\leq0\}$, puisque $g(s)>0,\ s>0$ donc $g^{*}(0)=0$.\\
 $\ds\lim_{u\rightarrow+\infty}g^{*}(u)=\ds\lim_{u\rightarrow+\infty}\sup_{g(t)\leq u}t=\ds\lim_{u\rightarrow+\infty}\sup\{t,\ g^{-1}([0,u[)\}=+\infty$.\\
 Soient $y_{0}\geq0$ et $(y_{n})$ une suite telle que $y_{n}\downarrow y_{0}$,  $$x_{0}:= g^{*}(y_{0})\leq g^{*}(y_{n}):=x_{n},\ \text{et}\ x_{n}\downarrow x\geq x_{0},\ \text{quand}\ n\rightarrow+\infty.$$ Par définition de $g^{*}$, pour tout $\epsilon>0$ et $n\in\mathbb{N}$,
 $$g(x_{n}-\epsilon)<y_{n}\leq g(x_{n}+\epsilon).$$ Supposons que $x>x_{0}$ et soit $\epsilon=\frac{x-x_{0}}{2}$ donc
 $$y_{0}\leq g(x_{0}+\epsilon)\leq g(x_{n}-\epsilon)<y_{n},$$ alors $$y_{0}\leq g(x_{0}+\epsilon) \leq g(x_{0}-\epsilon)<y_{0},$$ contradiction. D'où $x=x_{0}$ et $g^{*}$ continue à droite.\\ \\
 $\mathbf{(2)}$ On a; \begin{align*}
         g^{*}(u)&=\sup\{t;\ g(t)\leq u\}\\
         &= \sup\{t;\ t\in g^{-1}([0,u])\}\\
         &=\sup\{t;\ t\in[0,g^{-1}(u)]\}=g^{-1}(u),
       \end{align*}
       donc
       \begin{align*}
        g^{*}(g(u))&=\sup\{t;\ g(t)\leq g(u)\}\\
        &=\sup\{t;\ t\in[0,u]\}=u,
       \end{align*}

\end{proof}
\begin{exems}
  Soient $G_{1}$ et $G_{2}$ deux $N$-Fonctions définies par
  \begin{enumerate}
    \item $G_{1}(u)=\frac{|u|^{\alpha}}{\alpha},\ g_{1}(\tau)=G^{'}_{1}(\tau)=\tau^{\alpha-1},\ \tau\geq0,\ donc\ g^{*}_{1}(s)=s^{\beta-1},\ s\geq0$, avec $\frac{1}{\alpha}+\frac{1}{\beta}=1,\ \alpha\geq1$, alors la fonction complémentaire ou conjuguée associée est $$G_{1}^{*}(v)=\int_{0}^{v} g^{*}_{1}(s)ds=\frac{|v|^{\beta}}{\beta}$$.
    \item $G_{2}(u)=e^{u}-u-1$, où $g_{2}(t)=G_{2}^{'}(t)=e^{t}-1$, $t\geq0$ donc $g_{2}^{*}(s)=ln(s+1)$, $s\geq0$. La fonction complémentaire ou conjuguée associée est
    $$G_{2}^{*}(v)=\int_{0}^{v}g_{2}^{*}(s)ds=(1+v)ln(1+v)-v.$$
  \end{enumerate}
\end{exems}
Nous allons maintenant présenter l'inégalité de Young généralisée.
\begin{thm}
 Soit $(G,G^{*})$ un couple de $N$-Fonctions. Alors, $\forall a,b\in [0,+\infty[$ \\on a \begin{equation}\label{9} ab\leq G(a)+G^{*}(b). \end{equation}
 On a égalité dans \eqref{9} si et seulement si $b=g(a)$ ou $a=g^{*}(b)$.
\end{thm}
\begin{proof}
Soit $b_{0}=g(a)$, sans pert de généraliter on peut supposer que $b\geq b_{0}$.
Donc $$G^{*}(b)=\int_{0}^{b}g^{*}(s)ds=\int_{0}^{b_{0}}g^{*}(s)ds+\int_{b_{0}}^{b}g^{*}(s)ds.$$
Si $s\geq b_{0}=g(a)$, $g^{*}(s)\geq a$,
 $$G^{*}(b)\geq\int_{0}^{b_{0}}g^{*}(s)ds+\int_{b_{0}}^{b}a ds=\int_{0}^{b_{0}}g^{*}(s)ds+a(b-b_{0}).$$
Donc $$G(a)+G^{*}(b)\geq\int_{0}^{a}g(s)ds+\int_{0}^{b_{0}}g^{*}(s)ds+ab-ab_{0}.$$
Pour achever la démonstration il suffit de montrer que $$G(a)+\int_{0}^{b_{0}}g^{*}(s)ds=ab_{0}.$$
On a $$ab_{0}=G(a)+\displaystyle\int\int_{E}dudv,$$
où $$E=\{(u,v)\ \text{tel que}\ 0\leq u\leq a\ \text{et}\ g(u)\leq v\leq b_{0}\}.$$
En effet \begin{align*}
       \displaystyle\int\int_{E}dudv&=\int_{0}^{a}\int_{g(u)}^{b_{0}}dvdu=\int_{0}^{a}(b_{0}-g(u))du\\
       &=\int_{0}^{a}b_{0}du-\int_{0}^{a}g(u)du=ab_{0}-G(a).
     \end{align*}
 On peut exprimer $E$ comme suit $$E=\{(u,v)\ \text{tel que}\ 0\leq v\leq b_{0}\ \text{et}\ 0\leq u\leq g^{*}(v)\}.$$
 Alors, $$\int\int_{E}dudv=\int_{0}^{b_{0}}\int_{0}^{g^{*}(v)}dudv=\int_{0}^{b_{0}}g^{*}(v)dv=G^{*}(b_{0}).$$
on déduit que $$ab_{0}-G(a)=G^{*}(b_{0})$$ d'où le résultat.

\end{proof}

\begin{rem}
  D'après \eqref{9} on a $G^{*}(t)\geq st-G(s)$, donc $\ds\sup_{s\geq0}\{st-G(t)\}\leq G^{*}(t)$,
 l'égalité dans \eqref{9} aura lieu si et seulement si $g^{*}(t)=s$ c.à.d, $$sg^{*}(t)=G(g^{*}(t))+G^{*}(t)$$
 Par suite, $$G^{*}(t)=\ds\sup_{s\geq0}\{st-G(s)\}.$$
\end{rem}
\begin{pro}
  Soit $(G_{i},G_{i}^{*})_{1,2}$ deux couples de $N$-Fonctions, on suppose q'il existe $x_{0}\geq 0$ telles que $G_{1}(x)\leq G_{2}(x),\ \forall x\geq x_{0}\geq 0$.
Alors $$G_{2}^{*}(y)\leq G_{1}^{*}(y),\ \forall y\geq y_{0}\geq 0$$
où $y_{0}=g_{2}(x_{0})$, $g_{2}$ est la dérivée à droite de $G_{2}$.
\end{pro}
\begin{proof}
Par l'égalité dans \eqref{9} on a \begin{equation}\label{33}g_{2}(y)y=G_{2}(g_{2}(y))+G_{2}^{*}(y),\end{equation}
ce qui implique $$G_{2}^{*}(y)=g_{2}(y)y-G_{2}(g_{2}(y)).$$
L'inégalité de Young \eqref{9} pour $(G_{1},G_{1}^{*})$ donne, \begin{equation}\label{34}g_{2}(y)y\leq G_{1}(g_{2}(y))+G_{1}^{*}(y).\end{equation}
Les deux inégalités \eqref{33} et \eqref{34} donnent, \begin{align*}G_{2}^{*}(y)&\leq G_{1}(g_{2}(y))+G_{1}^{*}(y)-G_{2}(g_{2}(y))\\&= G_{1}^{*}(y)+G_{1}(g_{2}(y))-G_{2}(g_{2}(y)).\end{align*}
Comme $G_{1}(x)\leq G_{2}(x)$ alors $G_{1}(g_{2}(y))-G_{2}(g_{2}(y))<0\ \forall y\geq y_{0}.$
Par conséquent, $$G_{2}^{*}(y)\leq G_{1}^{*}(y).$$
\end{proof}
\begin{pro}
  Soit $(G,G^{*})$ un couple de $N$-Fonctions et $G^{-1}$ l'inverse de $G$. Alors
  \begin{enumerate}
    \item $G(a)+G(b)\leq G(a+b),\ \forall a,b\in\mathbb{R}^{+}$.
    \item $G^{-1}(a+b)\leq G^{-1}(a)+G^{-1}(b),\ \forall a,b\in\mathbb{R}^{+}$.
    \item $t<G^{-1}(t)(G^{*})^{-1}(t)\leq 2t,\ t>0$.
  \end{enumerate}
\end{pro}
\begin{proof}
  $(1)$ Soit $0\leq a\leq b$ alors \begin{align*}
                                     G(a)+G(b)&=\int_{0}^{a}g(t)dt+\int_{0}^{b}g(t)dt\\
                                     &\leq\int_{0}^{a}g(t)dt+\int_{a}^{a+b}g(t)dt=\int_{0}^{a+b}g(t)dt=G(a+b).
                                   \end{align*}
 $(2)$ Prenons dans $(1)$, $a=G^{-1}(t)$ et $b=G^{-1}(s)$ on aura $$t+s\leq G(G^{-1}(t)+G^{-1}(s)),$$ on applique $G^{-1}$ on obtient $(2)$.\\
 $(3)$ Comme $G$ et $G^{*}$ sont strictement croissantes, alors elles sont inversibles et l'inégalité \eqref{9} nous donne
$$G^{-1}(t){G^{*}}^{-1}(t)\leq G\bigg{(}G^{-1}(t)\bigg{)}+G^{*}\bigg{(}{G^{*}}^{-1}(t)\bigg{)}=2t$$ et d'après le théorème de la moyenne $\exists t^{*}\in]0,t[$ tel que $$\frac{G(t)}{t}=\frac{1}{t}\int_{0}^{t}g(s)ds=g(t^{*}).$$ Ainsi \begin{equation}\label{91}
                                                                                          \begin{aligned}
                                                                                          G^{*}\bigg{(}\frac{G(t)}{t}\bigg{)}&=\int_{0}^{\frac{G(t)}{t}}g(s)ds=
                                                                                          \frac{G(t)}{t}g^{*}(\tilde{t}),\                                 0<\tilde{t}<\frac{G(t)}{t}=g(t^{*})\\
                                                                                          &< \frac{G(t)}{t}.g^{*}(g(t))\leq \frac{G(t)}{t}.t=G(t).
                                                                                        \end{aligned}
                                                                                        \end{equation}
                                                                                        On pose $G(t)=\alpha$ et appliquant $(G^{*})^{-1}$ à \eqref{91} on obtient $$\frac{\alpha}{G^{-1}(\alpha)}<(G^{*})^{-1}(\alpha)\ \text{ou}\ \alpha<G^{-1}(\alpha)(G^{*})^{-1}(\alpha).$$
                                                                                        Ainsi $$t<G^{-1}(t)(G^{*})^{-1}(t)\leq 2t.$$
\end{proof}
\subsection{La condition $\triangle_{2}$}
Dans cette section, nous allons étudier une classe spéciale des $N$-fonctions. 
\begin{dfn}\label{dfn2}
 Une $N$-Fonction $G$ satisfait la condition-$\triangle_{2}$ (resp: $\triangle_{2,0}$) s'il existe $k>0$ et $T\geq0$ (resp: $T=0$) tels que,\begin{equation}\label{68} G(2t)\leq kG(t),\ \forall t\geq T\ (\text{resp:}\ \forall t\geq0).\end{equation}
\end{dfn}
\begin{exem}
  La fonction $G(t)=ct^{p}$, $c>0$, $p>1$, satisfait la condition $\triangle_{2}$: On peut prendre $T=0$ et $k=2^{p}$.
\end{exem}
\begin{rem}
  Si $G$ satisfait $\triangle_{2}$, alors $k>2$. En effet,  $$G(t)=G\bigg{(}\frac{1}{2}2t\bigg{)}\leq\frac{1}{2}G(2t).$$
\end{rem}

L'assertion suivante est un critère utile pour qu'une fonction satisfaite la condition $\triangle_{2}$.
\begin{thm}\label{thm7}
 Soit $G$ une $N$-Fonction. On a,
 \begin{enumerate}
  \item $G$ satisfait $\triangle_{2}$ si et seulement s'il existe $p>0$ et $t_{0}>0$ tels que \begin{equation}\label{10}\frac{g(a)}{G(a)}\leq\frac{p}{a},\ \forall a>t_{0}.\end{equation}
   \item $G$ satisfait $\triangle_{2}$ si et seulement pour tout $l>1$ il existe une constante $C(l)$ tel que \begin{equation}\label{11}G(lt)\leq C(l) G(t)\ \forall t\geq T.
           \end{equation}
   \item   $G^{*}$ satisfait $\triangle_{2}$ si et seulement si $g$ est inversible et $$\ds\inf_{t>0}\frac{tg(t)}{G(t)}>1.$$
 \end{enumerate}
\end{thm}
\begin{proof}
 1) On suppose que $\frac{g(a)}{G(a)}\leq\frac{p}{a}$. En intégrant sur $[a,2a]$, on obtient $$\int_{a}^{2a}\frac{g(t)}{G(t)}dt\leq p\int_{a}^{2a}\frac{dt}{t}=p\ln(2)$$
  et par suite $$G(2a)\leq 2^{p }G(a).$$
  Inversement: on suppose que $G(2a)\leq CG(a)$, alors $$CG(a)\geq G(2a)=\int_{0}^{2a}g(t)dt>\int_{a}^{2a}g(t)dt>ag(a).$$
 2) On suppose que $G$ satisfait la condition $\triangle_{2}$. Soit $l>1$, donc $\exists n\in\mathbb{N}$ tel que $2^{n}\geq l$. Alors $$G(lu)\leq G(2^{n}u)\leq k^{n}G(u)=C(l)G(u).$$
 Inversement: Il suffit de prendre $l=2$.\\
 3) Supposons que $g$ est inversible, alors $\ds\inf_{t>0}\frac{tg(t)}{G(t)}>1$ si et seulement si $\displaystyle\frac{tg^{-1}(t)}{G^{*}(t)}<+\infty$ si et seulement si $G^{*}$ satisfait $\triangle_{2}$.
\end{proof}
\begin{exem}
  Soit la fonction $$G(x)=x^{\alpha}(\ln (x)+1),\ \alpha>1,\ x\geq0,$$ alors $G$ satisfait $\triangle_{2}$.\\
  \text{En effet,} $$g(x)=\alpha x^{\alpha-1}(\ln (x)+1)+x^{\alpha-1},$$
  alors $$\frac{xg(x)}{G(x)}=\alpha+\frac{1}{\ln (x)+1}\leq \alpha+1,\ \forall x\geq1.$$
  Donc il existe $p=\alpha+1>0$ et il existe $t_{0}=1$ tels que $\frac{tg(t)}{G(t)}\leq p,\ \forall t\geq t_{0}$.
\end{exem}

Nous allons maintenant nous intéresser à l'étude des relations de comparaison entre les $N$-Fonctions.
\begin{dfn}\label{dfn1}
Soient $G_{1}$ et $G_{2}$ deux $N$-Fonctions. On dit que:
\begin{enumerate}\label{ac}
  \item $G_{1}$ domine $G_{2}$ à l'infini et on note $G_{2}\prec G_{1}$ ou $G_{1}\succ G_{2}$ si et seulement s'il existe $c>0$ et $T>0$ tel que$$G_{2}(x)\leq G_{1}(cx),\ \forall x\geq T.$$
  \item $G_{1}$ est équivalente à $G_{2}$ à l'infini et on note $G_{1}\sim G_{2}$ si et seulement si $G_{1}\prec G_{2}$ à l'infini et $G_{1}\succ G_{2}$ à l'infini. Autrement dit,\\$G_{1}$ et $G_{2}$ sont équivalentes si et seulement s'il existe $a,b\in R_{+}$ tels que $0<a\leq b<\infty$ et il existe $x_{0}>0$ tels que $$G_{1}(ax)\leq G_{2}(x)\leq G_{1}(bx),\ \forall x\geq x_{0}.$$
  \item On dit que $G_{2}$ est à croissance plus lente que $G_{1}$, et on note $G_{2}\prec\prec G_{1}$, si  $\forall\lambda>0,$
  $$\ds\lim_{t\rightarrow\infty}\frac{G_{2}(t)}{G_{1}(\lambda t)}=0.$$
\end{enumerate}
\end{dfn}
\begin{exems}
  \begin{enumerate}
    \item Soient $1<p<\infty$, $\epsilon>0$, $G_{1}(t)=t^{p}(|Log(t)|+1)$ et $G_{2}(t)=t^{p+\epsilon}$. Alors $G_{1}\prec G_{2}$.
    \item Pour toute $N$-Fonction $G$, les deux fonctions $G_{1}(t)=G(kt),\ k>0$ et $G$ sont équivalentes.
    \item Soient $1<p<\infty$ et $\epsilon>0$. Alors $$t^{p}\prec\prec t^{p}(|Log(t)|+1)\prec\prec t^{p+\epsilon}.$$
  \end{enumerate}
\end{exems}
\section{Classe d'Orlicz  $K^{G}(\Omega)$ }
Soient $\Omega$ un ouvert  de $\mathbb{R}^{N}$ et $G$ une $N$-Fonction. On appelle classe d'Orlicz l'ensemble des fonctions mesurables $u$ de $\Omega$ à valeurs réelles telles que $$\rho{(u\ ;G)}=\int_{\Omega}G(|u(x)|)dx<\infty.$$
On note $K^{G}(\Omega)$ la classe d'Orlicz. C'est à dire,
$$K^{G}(\Omega)=\bigg{\{} u:\Omega\rightarrow \mathbb{R},\ \text{mesurable} \  /\  \int_{\Omega} G(|u(x)|)dx < \infty\bigg{\}}.$$
\begin{thm}
 Soit $G$ une $N$-Fonction. Si $mes(\Omega)<\infty$, alors $$K^{G}(\Omega)\subset L^{1}(\Omega).$$
\end{thm}
\begin{proof}
  Soit $u\in K^{G}(\Omega)$. On sait que $\ds\lim_{t\rightarrow\infty}\frac{G(t)}{t}=+\infty$, donc il existe $k>0$ telle que pour $|t|>k$, $\frac{G(t)}{t}>1$. Soit  $$\Omega_{k}=\{x\in\Omega;|u(x)|>k\},$$ alors
  $$\frac{G(|u(x)|)}{|u(x)|}>1,\ \text{c.à.d}\ |u(x)|<G(|u(x)|),\ \forall\ x\in\Omega_{k}.$$
  Ainsi
  \begin{align*}
  \int_{\Omega}|u(x)|dx&=\int_{\Omega_{k}}|u(x)|dx+ \int_{\Omega\setminus\Omega_{k}}|u(x)|dx\\
  &\leq \int_{\Omega_{k}}G(|u(x)|)dx+k\ mes(\Omega\setminus\Omega_{k})\\
  &\leq\rho(u,G)+k\ mes(\Omega)<\infty.
  \end{align*}
  D'où $u\in L^{1}(\Omega)$.
\end{proof}

Dans le théorème suivant, on peut considérer l'espace de Lebesgue $L^{1}(\Omega)$ comme la réunion de toutes les classes d'Orlicz $ K^{G}(\Omega)$ où $G$ varie dans l'ensemble des $N$-Fonctions.
\begin{thm}
 Soit $u\in L^{1}(\Omega)$. On suppose que $mes(\Omega)<\infty$. Alors il existe une $N$-Fonction $G$ telle que $u\in K^{G}(\Omega)$.
\end{thm}
\begin{proof}
  Pour tout $ n\in\mathbb{N}^{*}$, on considère $$\Omega_{n}=\{x\in\Omega;n-1\leq|u(x)|<n,\}.$$
  Alors \begin{align*}
          \int_{\Omega}|u(x)|dx & =\sum_{n=1}^{\infty}\int_{\Omega_{n}}|u(x)|dx\\
          &\geq\sum_{n=1}^{\infty}(n-1)mes(\Omega_{n})\\
          &=\sum_{n=1}^{\infty}n\ mes(\Omega_{n})-mes(\Omega).
        \end{align*}
   La série $\sum_{n=1}^{\infty}n\ mes(\Omega_{n})$ est convergente, car $u\in L^{1}(\Omega)$ et $mes(\Omega)<\infty$. \\
   De plus, il existe une suite croissante $(\alpha_{n})$ telle que $$\alpha_{n}>1,\ \lim\alpha_{n}=\infty\ \text{et}\ \sum_{n=1}^{\infty}\alpha_{n}n\ mes(\Omega_{n})<\infty.$$
 On définit $$g(t)=\begin{cases}
                     t & \mbox{si }\ t\in[0,1[  \\
                     \alpha_{n} & \mbox{si}\ t\in[n,n+1[,\ n\in\mathbb{N}.
                   \end{cases}$$
  La fonction $G(t)=\int_{0}^{t}g(s)ds$ est une $N$-Fonction et $G(n)\leq n\alpha_{n},\ n\in\mathbb{N}.$ Alors
        \begin{align*}
          \int_{\Omega}G(|u(x)|)dx  & =\sum_{n=1}^{\infty}\int_{\Omega_{n}}G(|u(x)|)dx\leq\sum_{n=1}^{\infty}G(n)\ mes(\Omega_{n})\\
          &\leq\sum_{n=1}^{\infty}n\alpha_{n}\ mes(\Omega_{n})<\infty.
        \end{align*}
      D'où $u\in K^{G}(\Omega)$.
\end{proof}
\begin{rem}
La classe $K^{G}(\Omega)$ n'est pas toujours un espace vectoriel.
\end{rem}

Soit $\Omega=]0,1[$ et
$$G(x)=e^{|x|}-1,\ f(x)=\begin{cases}
                                    \frac{n}{2} & \mbox{si } \frac{1}{2^{n}}\leq x\leq\frac{1}{2^{n-1}} \\
                                    0 & \mbox{sinon}.
                                  \end{cases}$$
On a \begin{align*}
       \rho(f; G)&=\int_{\Omega}G(f(x))dx=\int_{\Omega}(e^{\frac{n}{2}}-1)\chi_{[\frac{1}{2^{n}},\frac{1}{2^{n-1}}]}dx \\
       &=\sum_{n\geq1}(e^{\frac{n}{2}}-1)mes([\frac{1}{2^{n}},\frac{1}{2^{n-1}}])\\
       &=\sum_{n\geq1}\frac{1}{2^{n}}(e^{\frac{n}{2}}-1)=\sum_{n\geq1}\bigg{(}\frac{\sqrt{e}}{2}\bigg{)}^{n}-\sum_{n\geq1}\frac{1}{2^{n}}.
     \end{align*}
On a la série $\sum_{n\geq1}\frac{1}{2^{n}}$ converge et $\frac{\sqrt{e}}{2}<1$ donc $\sum_{n\geq1}\bigg{(}\frac{\sqrt{e}}{2}\bigg{)}^{n}<\infty$,  $\rho(f;\ G)<\infty$, $f\in K^{G}(\Omega)$,\\
\begin{align*}
 \rho(2f; G)&=\int_{\Omega}G(2f(x))dx=\int_{\Omega}(e^{n}-1)\chi_{[\frac{1}{2^{n}},\frac{1}{2^{n-1}}]}dx \\
  &=\sum_{n\geq1}(e^{n}-1)mes([\frac{1}{2^{n}},\frac{1}{2^{n-1}}])=\sum_{n\geq1}\frac{1}{2^{n}}(e^{n}-1)
  =\sum_{n\geq1}\bigg{(}\frac{e}{2}\bigg{)}^{n}-\sum_{n\geq1}\frac{1}{2^{n}}.
\end{align*}
On a $\frac{e}{2}>1$ donc la série $\ds\sum_{n\geq1}\bigg{(}\frac{e}{2}\bigg{)}^{n}$ diverge. D'où $2f\notin K^{G}(\Omega)$.\\

Le théorème suivant donne un résultat analogue de l'inégalité de Hölder dans les classes d'Orlicz.
\begin{thm}\label{thm26}
Soient $(G,G^{*})$ un couple de $N$-Fonctions, $u\in K^{G}(\Omega)$ et $v\in K^{G^{*}}(\Omega)$. Alors $$uv\in L^{1}(\Omega)$$ et
\begin{equation}\label{12}\int_{\Omega}|u(x)v(x)|dx\leq\rho(u;G)+\rho(v;G^{*}).\end{equation}
\end{thm}
\begin{proof}
Dans l'inégalité de Young \eqref{9} on remplace $a$ par $|u(x)|$ et $b$ par $|v(x)|$, on obtient \eqref{12}.
\end{proof}
\begin{rem}\label{rem1}
  Si $|v(x)|=g(|u(x)|)$ ou $|u(x)|=g^{*}(|v(x)|)$, on a égalité dans \eqref{12}.
\end{rem}
La proposition suivante nous donne une comparaison entre les classes d'Orlicz.
\begin{pro}\label{pro3}
 Soient $G_{1}$ et $G_{2}$ deux $N$-fonctions et $mes(\Omega)<\infty$. Alors $K^{G_{1}}(\Omega)\subset K^{G_{2}}(\Omega)$\\ si et seulement si, il existe $T>0$ et $a>0$
 telles que \begin{equation}\label{13} G_{2}(u)\leq aG_{1}(u)\ \forall u\geq T.\end{equation}
\end{pro}
\begin{proof}
  \text{Condition suffisante.}\\ \\
  Soit $u\in K^{G_{1}}(\Omega)$, on a $$\rho(u;\ G_{2})=\int_{\Omega}G_{2}(|u(x)|)dx\leq G_{2}(T)mes(\Omega)+a\int_{\Omega}G_{1}(|u(x)|)dx<+\infty$$
  \text{Condition nécessaire.}\\ \\
  On suppose que \eqref{13} n'est pas satisfaite, alors il existe une suite $(u_{n})_{n\in\mathbb{N}}$ strictement croissante telle que
  $$G_{2}(u_{n})>2^{n}G_{1}(u_{n})\ \forall n\in\mathbb{N}.$$
  Soit $(\Omega_{n})_{n\in\mathbb{N}}$ une suite de sous-ensembles de $\Omega$ disjoints telle que $$mes(\Omega_{n})=\frac{G_{1}(u_{1})mes(\Omega)}{2^{n}G_{1}(u_{n})},\ n=1,2,...$$
  On considère la fonction suivante
   $u(x)=\begin{cases}
     u_{n} & \mbox{si }\ x\in\Omega_{n} \\
     0 & \mbox{si}\ x\notin\ds\cup_{n=1}^{\infty} \Omega_{n}.
   \end{cases}$
  On a
   \begin{align*}
     \int_{\Omega}G_{1}(u(x))dx & =\sum_{n=1}^{\infty}\int_{\Omega_{n}}G_{1}(u(x))dx=\sum_{n=1}^{\infty}G_{1}(u_{n})mes(\Omega_{n})\\
     &=\sum_{n=1}^{\infty}\frac{G_{1}(u_{1})mes(\Omega)}{2^{n}}<\infty,
   \end{align*}
  d'où $u\in K^{G_{1}}$.\\ Montrons que $u\notin K^{G_{2}}$, en effet
   \begin{align*}
      \int_{\Omega}G_{2}(u(x))dx & =\sum_{n=1}^{\infty}\int_{\Omega_{n}}G_{2}(u(x))dx=\sum_{n=1}^{\infty}G_{2}(u_{n})mes(\Omega_{n})\\&\geq
      \sum_{n=1}^{\infty}G_{1}(u_{1})mes(\Omega)=\infty
   \end{align*}
   donc $$K^{G_{1}}(\Omega)\nsubseteq K^{G_{2}}(\Omega).$$
\end{proof}
\begin{thm}\label{thm15}
 \begin{enumerate}
   \item [(i)] Si $mes(\Omega)<\infty$, alors $K^{G}(\Omega)$ est un espace vectoriel si et seulement si $G$ satisfait $\triangle_{2}$.
   \item [(ii)] Si $mes(\Omega)=+\infty$ et $G$ vérifiant $\triangle_{2,0}$, alors $K^{G}(\Omega)$ est un espace vectoriel.
 \end{enumerate}
\end{thm}
\begin{proof}
 Soit $G$ une $N$-fonction vérifiant $\triangle_{2}$,\\
 $\bullet$ D'après la convexité de $G$, on a $\forall u,v\in K^{G}$,
  \begin{align*}
     \int_{\Omega}G(|u(x)+v(x)|)dx&=\int_{\Omega}G(\frac{1}{2}|2u(x)+2v(x)|)dx\\&\leq
     \frac{1}{2}\int_{\Omega}G(2|u(x)|)dx+\frac{1}{2}\int_{\Omega}G(|2v(x)|)dx <\infty.
  \end{align*} Donc $u+v\in K^{G}.$\\
 $\bullet$ Supposons $mes(\Omega)<\infty$ et $G$ vérifiant $\triangle_{2}$.
    Soient $u\in K^{G}(\Omega)$ et $\gamma\geq0$, il existe $n\in\mathbb{N}$ tel que $\gamma\leq2^{n}$,  on a
   $$G(|\gamma| t)\leq G(2^{n}t)\leq k^{n}G(t),\ \forall t\geq T,$$ donc
   \begin{align*}
      \int_{\Omega}G(|\gamma| |u(x)|)dx&=\int_{\{x,\ |u(x)|\leq T\}}G(|\gamma| |u(x)|)dx+\int_{\{x,\ |u(x)|>T\}}G(|\gamma| |u(x)|)dx\\
      &\leq k^{n}\int_{\{x,\ |u(x)|\leq T\}}G(|u(x)|)dx+k^{n}\int_{\{x,\ |u(x)|>T\}}G(|u(x)|)dx\\
      &\leq k^{n}G(T)mes(\Omega)+k^{n}\int_{\Omega}G(|u(x)|)dx<\infty,
   \end{align*}
 $\gamma u\in  K^{G}(\Omega)$. \\
 $\bullet$  Supposons $mes(\Omega)=\infty$ et $G$ satisfait  $\triangle_{2,0}$, on a
  $$G(|\gamma| t)\leq G(2^{n}t)\leq k^{n}G(t),\ \forall t\geq 0,$$ donc $$  \int_{\Omega}G(|\gamma| u(x))dx\leq k^{n}\int_{\Omega}G(|u(x)|)dx<\infty,$$ et  $\gamma u\in  K^{G}(\Omega)$.\\
$\bullet$ On suppose que $K^{G}(\Omega)$ est un espace vectoriel et $mes(\Omega)<+\infty$. Soit $u\in K^{G}(\Omega)$, $2u\in K^{G}(\Omega)$. Alors $$K^{G}(\Omega)\subset K^{G_{1}}(\Omega),\ \text{avec}\ G_{1}(t)=G(2t).$$ D'après la proposition \ref{pro3} il existe $T>0$, $a>0$ tels que
$$G_{1}(t)=G(2t)\leq aG(t),\ \forall t\geq T,$$ c.à.d, $G$ satisfait la condition $\triangle_{2}$.
\end{proof}

\chapter{Espaces d'Orlicz }
\section{Espaces $L^{G}$ d'Orlicz}
\subsection{Définition et norme d'Orlicz}
\begin{dfn}
  Soient $(G,G^{*})$ un couple de $N$-Fonctions  et $\Omega$ un ouvert de $\mathbb{R}^{N}$. Soit $u$ une fonction mesurable définie presque partout sur $\Omega$. On note par \begin{equation}\label{14}
                                          \|u\|_{G}=\sup\bigg{\{}\int_{\Omega}|u(t)v(t)|dt:\ v\in K^{G^{*}}(\Omega),\ \rho(v\ ;G^{*})\leq 1\bigg{\}}.
                                        \end{equation}
   est appelé  la norme d'Orlicz. L'ensemble $L^{G}(\Omega)$ de toutes les fonctions mesurables $u$ telles que  $\|u\|_{G}<\infty$ est appelé \textbf{l'espace d'Orlicz}. $$L^{G}(\Omega)=\bigg{\{}u:\Omega\rightarrow\mathbb{R},\ \text{mesurable}/\ \|u\|_{G}<+\infty\bigg{\}}.$$
\end{dfn}
Nous allons maintenant établir une relation simple entre la classe d'Orlicz et l'espace d'Orlicz.
\begin{pro}\label{pro6}
  Soit $G$ une $N$-Fonction. Alors \begin{equation}\label{15}
                                     K^{G}(\Omega)\subset L^{G}(\Omega),
                                   \end{equation}
        et \begin{equation}\label{16}
             \|u\|_{G}\leq \rho(u;G)+1.
           \end{equation}
\end{pro}
\begin{proof}
 Soit $u\in K^{G}(\Omega)$ et $v\in K^{G^{*}}(\Omega)$ tel que $\rho(v;G^{*})\leq1$, d'après l'inégalité de Young on a $$\int_{\Omega}|u(x)v(x)|dx\leq\rho(u;G)+\rho(v;G^{*})\leq \rho(u;G)+1.$$
 Ceci donne $\|u\|_{G}\leq \rho(u;G)+1$ et par suite $u\in L^{G}(\Omega)$.
\end{proof}
\begin{exem}
 Dans le cas $N=1$, $\Omega=]1,+\infty[$, $G(t)=e^{t}$, $u(x)=-2Log(x)$, on a, $u\notin L^{p}(]1,+\infty[)$, $\forall$ $1\leq p \leq +\infty,$ mais $u\in L^{G}(]1,+\infty[)$. On déduit que les espaces d'Orlicz sont plus grands que les espaces de Lebesgue.
\end{exem}
\begin{thm}
  L'ensemble $(L^{G}(\Omega),\|.\|_{G})$ est un espace vectoriel normé.
\end{thm}
\begin{proof}
  Soit $u,\omega\in L^{G}(\Omega)$ (c.à.d $\|u\|_{G}<\infty$, $\|\omega\|_{G}<\infty$).\\
  $(i)$ soit $\lambda\in\mathbb{R}$, on a
  $$\|\lambda u\|_{G}=\ds\sup_{\rho(v;G^{*})\leq1,\ v\in K^{*}(\Omega)}\int_{\Omega}|\lambda u(x)v(x)|dx=|\lambda| \ds\sup_{\rho(v;G^{*})\leq1,\ v\in K^{*}(\Omega)}\int_{\Omega}| u(x)v(x)|dx= |\lambda|\|u\|_{G}<\infty,$$ c.à.d $\lambda u\in L^{G}(\Omega)$.\\
 $(ii)$ \begin{align*}
    \|u+\omega\|_{G} & =\ds\sup_{\rho(v;G^{*})\leq1}\int_{\Omega}|u(x)+\omega(x)||v(x)|dx\\&\leq\ds\sup_{\rho(v;G^{*})\leq1}\int_{\Omega}| u(x)v(x)|dx+\ds\sup_{\rho(v;G^{*})\leq1}\int_{\Omega}|\omega(x)v(x)|dx\\&=\|u\|_{G}+\|\omega\|_{G}<\infty,
  \end{align*}
   donc $u+\omega\in L^{G}(\Omega)$. Ainsi $L^{G}(\Omega)$ est un espace vectoriel.\\
    $(iii)$ Soit $u\in L^{G}(\Omega)$, $\|u\|_{G}=0$. Soit $$A=\{x\in\Omega ,|u(x)|>0\},$$ supposons que $mes(A)>0$. Comme la mesure de Lebesgue possède la propriété du sous-ensemble de mesure fini il existe alors $\Omega_{1}\subset A$ tel que $0<mes(\Omega_{1})<\infty$. On sait que  $\ds\lim_{t\rightarrow0}G^{*}(t)=0$, alors, il existe $k>0$ telle que $G^{*}(k)<\frac{1}{mes(\Omega_{1})}$.\\ On considère la fonction $v_{0}$
  $$v_{0}(x):=\begin{cases}
                k & \mbox{si }\ x\in\Omega_{1} \\
                0 & \mbox{sinon},
              \end{cases}$$
   alors  $$\rho(v_{0},G^{*})=\int_{\Omega}G^{*}(|v_{0}(x)|)dx=\int_{\Omega_{1}}G^{*}(k)dx<1,$$ donc $$\|u\|_{G}\geq\int_{\Omega}|u(x)|v_{0}(x)dx=k\int_{\Omega_{1}}|u(x)|dx,$$ ceci donne $u=0\ p.p$ sur $\Omega_{1}$. Contradiction. Donc $mes(A)=0$ et par suite $u\equiv0$ $p.p$ sur $\Omega$.
\end{proof}
\begin{rem}\label{rem2}
  Si $$u_{1},\ u_{2}\in L^{G}(\Omega)\ \text{et}\ |u_{1}(x)|\leq|u_{2}(x)|$$
  presque partout sur $\Omega$, alors $\|u_{1}\|_{G}\leq\|u_{2}\|_{G}$.
\end{rem}
Peut être il n'est pas facile de calculer précisément la norme  d'une fonction donnée, cependant, il est possible de calculer la norme d'une fonction caractéristique d'un ensemble mesurable donné. Il s'avère que cette information est très utile.
\begin{pro}\label{pro5}
 Soient $G$ une $N$-Fonction et $E$ un sous-ensemble mesurable de $\Omega$ tel que $0<mes(E)<\infty$. Alors \begin{equation}\label{17}
                               \|\chi_{E}\|_{G}=mes(E)(G^{*})^{-1}\bigg{(}\frac{1}{mes(E)}\bigg{)}
                             \end{equation}
\end{pro}
\begin{proof}
  On a $$\|\chi_{E}\|_{G}=\ds\sup_{\rho(v;G^{*})\leq1}\int_{E}|v(x)|dx$$ Soit $v\in K^{G^{*}}(\Omega)$ telle que $\rho(v;G^{*})\leq1$, d'après l'inégalité de Jensen on a $$G^{*}\bigg{(}\frac{1}{mes(E)}\int_{E}|v(x)|dx\bigg{)}\leq \frac{1}{mes(E)}\int_{E}G^{*}(|v(x)|)dx,$$
  donc $$\int_{E}|v(x)|dx\leq mes(E)( G^{*})^{-1}(\frac{1}{mes(E)}).$$ Ainsi $$ \|\chi_{E}\|_{G}\leq mes(E) (G^{*})^{-1}(\frac{1}{mes(E)}).$$
  D'autre part, soit $v_{0}(x)=\chi_{E}(x)(G^{*})^{-1}(\frac{1}{mes(E)}),\ x\in\Omega$; on a\\
  $\rho(v_{0};G^{*})=1$, donc $$\|\chi_{E}\|_{G}\geq\displaystyle\int_{E}v_{0}(x)dx=mes(E) (G^{*})^{-1}(\frac{1}{mes(E)}).$$ Alors
  $$\|\chi_{E}\|_{G}=mes(E)(G^{*})^{-1}\bigg{(}\frac{1}{mes(E)}\bigg{)}.$$
\end{proof}
\begin{rem}\label{rem5}
  Il est facile de vérifier que, pour toute $N$-Fonction $G$ et toute
fonction mesurable finie $u$ sur $\Omega$, on a \begin{equation}\label{80}
                                                  \|u\|_{G}=\ds\sup_{\rho(v;G^{*})\leq1}\bigg{|}\int_{\Omega}u(x)v(x)dx\bigg{|}.
                                                \end{equation}
  Il suffit de remarquer que $$\bigg{|}\int_{\Omega}u(x)v(x)dx\bigg{|}\leq\int_{\Omega}|u(x)||v(x)|dx=\int_{\Omega}u(x)\tilde{v}(x)dx,$$
  où $\tilde{v}=|v|sign(u)$ et $\rho(\tilde{v};G^{*})\leq1$ si $\rho(v;G^{*})\leq1$.
\end{rem}

\subsection{Inégalité de H\"{o}lder}

\begin{thm}\label{lem1}
  Soient $G$ une $N$-Fonction et $u\in L^{G}(\Omega)$ tel que $\|u\|_{G}\neq0$. Alors
  \begin{equation}\label{18}
  \int_{\Omega}G\bigg{(}\frac{|u(x)|}{\|u\|_{G}}\bigg{)}dx\leq1.
\end{equation}
\end{thm}
Pour la démonstration du théorème \ref{lem1} on aura besoin les lemmes suivants.
\begin{lem}\label{lem6}
  Soit $u\in L^{G}(\Omega)$ et $v\in K^{G^{*}}(\Omega)$. Alors \begin{equation}\label{64}
                                     \int_{\Omega}|u(x)v(x)|dx\leq\begin{cases}
                                                                    \|u\|_{G} & \mbox{si }\ \rho(v;G^{*})\leq1, \\
                                                                    \|u\|_{G}\rho(v;G^{*})& \mbox{si}\ \rho(v;G^{*})>1.
                                                                  \end{cases}
                                   \end{equation}
\end{lem}
\begin{proof}
  La première partie de l'inégalité \eqref{64} découle de la définition de la norme d'Orlicz.
  Pour la deuxième partie nous utilisons la définition de $N$-Fonction, on sait que $G^{*}(\alpha t)\leq \alpha G^{*}(t)$, pour $t\geq0$ et $\alpha\in]0,1[$.
  Posons, $t=|v(x)|$, $\alpha=\frac{1}{\rho(v;G^{*})}$, on intègre sur $\Omega$, on obtient
  $$\int_{\Omega}G^{*}\bigg{(}\frac{1}{\rho(v;G^{*})}|v(x)|\bigg{)}dx\leq\frac{1}{\rho(v;G^{*})}\rho(v;G^{*})=1$$
  d'après la définition de la norme d'Orlicz, on a
  $$\int_{\Omega}|u(x)|\frac{|v(x)|}{\rho(v;G^{*})}dx\leq\|u\|_{G}.$$
\end{proof}
\begin{lem}\label{lem7}
  Soit $u\in L^{G}(\Omega)$ une fonction bornée telle que $u(x)=0$ si $x\in\Omega\setminus\Omega_{0}$ où $\Omega_{0}\subset\Omega$ et $mes(\Omega_{0})<\infty$. Alors $$\int_{\Omega}G\bigg{(}\frac{|u(x)|}{\|u\|_{G}}\bigg{)}dx\leq1.$$
\end{lem}
\begin{proof}
  On pose $$v(x)=g\bigg{(}\frac{1}{\|u\|_{G}}|u(x)|\bigg{)},$$ où $g$ est la dérivée à droite de $G$. Alors les deux fonctions $G\bigg{(}\frac{1}{\|u\|_{G}}|u(x)|\bigg{)}$ et $G^{*}(|v(x)|)$ sont bornées et intégrable sur $\Omega_{0}$, de plus elles appartiennent à $L^{1}(\Omega)$.
 D'après le théorème \ref{thm26} on a
  $$\int_{\Omega}\frac{1}{\|u\|_{G}}|u(x)v(x)|dx=\rho\bigg{(}\frac{|u|}{\|u\|_{G}};G\bigg{)}+\rho(v;G^{*}).$$
  On déduit $$\max\{\rho(v;G^{*}),1\}\geq\int_{\Omega}G\bigg{(}\frac{1}{\|u\|_{G}}|u(x)|\bigg{)}dx+\rho(v;G^{*}).$$
  Si $\rho(v;G^{*})>1$, alors $$\int_{\Omega}G\bigg{(}\frac{1}{\|u\|_{G}}|u(x)|\bigg{)}dx=0.$$
  Si $\rho(v;G^{*})\leq1$, alors $$\int_{\Omega}G\bigg{(}\frac{1}{\|u\|_{G}}|u(x)|\bigg{)}dx+\rho(v;G^{*})\leq1$$
  D'où le résultat.
\end{proof}
\begin{proof}[\textbf{Démonstration du théorème \ref{lem1}}]
  Soit $u\in L^{G}(\Omega)$. Soit $(\Omega_{n})_{n\in\mathbb{N}^{*}}$ une suite de sous-ensembles de $\Omega$ telle que $\Omega_{n}\subset\Omega_{n+1}$, $mes(\Omega_{n})<\infty$ et $\Omega=\cup_{n=1}^{\infty}\Omega_{n}$.\\
On considère la suite de fonctions suivante $$u_{n}(x)=\begin{cases}
                                               u(x) & \mbox{si }\ x\in\Omega_{n}\ \text{et}\ |u(x)|\leq n, \\
                                               n & \mbox{si }\ x\in\Omega_{n}\ \text{et}\ |u(x)|>n, \\
                                               0 & \mbox{si}\ x\in\Omega\setminus\Omega_{n}.
                                             \end{cases} $$
      On a $ |u_{n}(x)|\leq n,\ \forall x\in\Omega$, d'après le lemme \ref{lem7} on a
      $$\int_{\Omega}G\bigg{(}\frac{1}{\|u_{n}\|_{G}}|u_{n}(x)|\bigg{)}dx\leq1.$$
      De plus, $|u_{n}(x)|\leq|u(x)|$ pour tout $x\in\Omega$, d'après la remarque \ref{rem2} on a $\|u_{n}\|_{G}\leq\|u\|_{G}$, $\forall n\in\mathbb{N}$. Ainsi $$\frac{|u_{n}(x)|}{\|u\|_{G}}\leq\frac{|u_{n}(x)|}{\|u_{n}\|_{G}}\ \text{et}\ G\bigg{(}\frac{|u_{n}(x)|}{\|u\|_{G}}\bigg{)}\leq  G\bigg{(}\frac{|u_{n}(x)|}{\|u_{n}\|_{G}}\bigg{)}$$
      par conséquent $$\int_{\Omega}G\bigg{(}\frac{1}{\|u\|_{G}}|u_{n}(x)|\bigg{)}dx\leq1.$$
       Soit $x\in\Omega$, les deux suites $(u_{n}(x))_{n}$ et $G\bigg{(}\frac{1}{\|u\|_{G}}|u_{n}(x)|\bigg{)}$ sont croissantes, $u_{n}\rightarrow u\ p.p$ sur $\Omega$. D'après le théorème du convergence monotone on a
      $$\int_{\Omega}G\bigg{(}\frac{1}{\|u\|_{G}}|u(x)|\bigg{)}dx=\ds\lim_{n\rightarrow\infty}\int_{\Omega}G\bigg{(}\frac{1}{\|u\|_{G}}|u_{n}(x)|\bigg{)}dx\leq1.$$
\end{proof}
\begin{rem}
 D'après la proposition \ref{pro6}, on a $K^{G}(\Omega)\subset L^{G}(\Omega)$. Cette inclusion peut être généralisée dans le sens suivant: S'il existe $\lambda>0$ tel que $\lambda u\in K^{G}(\Omega)$, alors $u\in L^{G}(\Omega)$, c-à-d,  $$L^{G}(\Omega)\supset\bigg{\{}u: \Omega\rightarrow \bar{\mathbb{R}},\ \text{mesurable}/\ \exists \lambda>0, \int_{\Omega}G(|\lambda u(x)|)dx< +\infty\bigg{\}}.$$
En effet, si $\lambda u\in K^{G}(\Omega)$ on a $\lambda u\in L^{G}(\Omega)$ et comme  $L^{G}(\Omega)$ est un espace vectoriel, alors $\lambda^{-1}(\lambda u)=u\in L^{G}(\Omega)$.
De plus, une déclaration dans un sens inverse est également vraie: Si $u\in L^{G}(\Omega)$, alors il existe $\lambda>0$ tel que $\lambda u\in K^{G}(\Omega)$ (la forme explicite de la constante $\lambda$ est donné dans le théorème \ref{lem1}). Ainsi, on peut redéfinir l'espace d'Orlicz comme suit $$L^{G}(\Omega)=\bigg{\{}u: \Omega\rightarrow \mathbb{R},\ \text{mesurable}/\ \exists \lambda>0, \int_{\Omega}G(\lambda |u(x)|)dx< +\infty\bigg{\}}.$$
\end{rem}
En vertu du théorème \ref{lem1} , nous pourrons caractériser quand est-ce-que  une classe d'Orlicz coïncide
avec l'espace d'Orlicz.
\begin{thm}\label{thm13}
  Si $G$ satisfait la condition $\triangle_{2}$ (avec $G\in\triangle_{2,0}$ si $mes(\Omega)=+\infty$). Alors $$K^{G}(\Omega)=L^{G}(\Omega).$$
\end{thm}
\begin{proof}
  On sait que $K^{G}(\Omega)\subset L^{G}(\Omega)$. Soit $u\in L^{G}(\Omega),\ \|u\|_{G}\neq0$. D'après le théorème \ref{lem1}, $$\omega=\frac{u}{ \|u\|_{G}}\in K^{G}(\Omega).$$
  Comme $K^{G}(\Omega)$ est un espace vectoriel alors $\omega\|u\|_{G}=u\in K^{G}(\Omega)$. Ainsi $K^{G}(\Omega)\supset L^{G}(\Omega).$
\end{proof}
\begin{rem}
  Il découle du théorème \ref{lem1} que l'espace d'Orlicz $L^{G}(\Omega)$ est engendré par la classe $K^{G}(\Omega)$. En effet, si  $u\in L^{G}(\Omega)$ tel que $\|u\|_{G}\neq0$, alors $u=cv$ avec $v\in K^{G}(\Omega)$, $v=\frac{u}{\|u\|_{G}}$ et $c=\|u\|_{G}$.
\end{rem}
Dans l'environnement des espaces de Lebesgue, l'inégalité de H\"{o}lder classique est sous cette forme$$\int_{\Omega}|u(x)v(x)|dx\leq\|u\|_{p}\|v\|_{q}$$ avec $p,q\in[1,+\infty[$ et $\frac{1}{p}+\frac{1}{q}=1$. Le théorème suivant fournit une inégalité analogue pour les espaces d'Orlicz.
\begin{thm}\label{thm1}
  Soit $(G,G^{*})$ un couple de $N$-Fonction, si $u\in L^{G}(\Omega)$ et $v\in L^{G^{*}}(\Omega)$, alors
  \begin{equation}\label{19}
  \int_{\Omega}|u(t)v(t)|dt\leq \|u\|_{G}\|v\|_{G^{*}}.
  \end{equation}
\end{thm}
\begin{proof}
Si $\|v\|_{G^{*}}=0$ l'inégalité \eqref{19} est triviale. Si $\|v\|_{G^{*}}\neq0$, on applique l'inégalité \eqref{18} à la fonction $G^{*}$ on obtient $\rho(\frac{v}{\|v\|_{G^{*}}};G^{*})\leq1$. Ainsi, l'inégalité \eqref{19} découle  de la définition de la norme d'Orlicz de $u$,
$$\int_{\Omega}|u(t)v(t)|dt=\|v\|_{G^{*}}\int_{\Omega}\bigg{|}u(t)\frac{v(t)}{\|v\|_{G^{*}}}\bigg{|}dt\leq\|u\|_{G}\|v\|_{G^{*}}.$$
\end{proof}
\subsection{La norme de Luxemberg}
 La définition de la norme d'Orlicz $\|u\|_{G}$ nécessite la connaissance de l'expression de la fonction conjuguée $G^{*}$ qui est, en général, difficile à calculer. Une autre norme équivalente exprimée uniquement en termes de $G$ a été introduite par Luxemburg dans \cite{22} .
 \begin{dfn}
   Soient $G$ une $N$-Fonction et $u$ une fonction mesurable définie sur $\Omega$. On définit
   $$\|u\|_{(G)}=\inf\bigg{\{}\lambda>0\ :\ \rho(\frac{u}{\lambda}\ ;G)\leq1\bigg{\}}$$ appelée la norme de Luxemberg de $u$.
 \end{dfn}
  \begin{rem}\label{rem7}
   D'après le lemme \ref{lem1}, si $u\in L^{G}(\Omega)$, alors \begin{equation}\label{20}
                                                                                 \|u\|_{(G)}\leq\|u\|_{G}.
                                                                               \end{equation}
   En faisant tendre $\lambda$ vers $\|u\|_{(G)}$  et appliquant le théorème de Fatou dans $$\rho(\frac{u}{\lambda};G)=\int_{\Omega}G\bigg{(}\frac{|u(x)|}{\lambda}\bigg{)}dx\leq1,$$
   on obtient \begin{equation}\label{21}
                \rho(\frac{u}{\|u\|_{(G)}};G)=\int_{\Omega}G\bigg{(}\frac{|u(x)|}{\|u\|_{(G)}}\bigg{)}dx\leq1.
              \end{equation}
    \end{rem}
 \begin{pro}
 Soit $u\in L^{G}(\Omega)$. Alors l'application $$u\mapsto\|u\|_{(G)}$$ est une norme sur l'espace $ L^{G}(\Omega)$.
 \end{pro}
 \begin{proof}
 \textbf{1)}
   Il est évident que si $u$ est une fonction nulle presque partout, alors $\|u\|_{(G)}=0$.\\
   Montrons que $\|u\|_{(G)}=0$ implique $u=0$.
   Soit $u\in L^{G}(\Omega)$ telle que  $\|u\|_{(G)}=0$. Il existe alors une suite décroissante de nombres réels strictement positifs $(\lambda_{n})_{n\in\mathbb{N}}$ vérifiant à la fois \begin{equation}\label{35}\ds\lim_{n\rightarrow\infty}\lambda_{n}=0\ \text{et}\ \int_{\Omega}G\bigg{(}\frac{|u(x)|}{\lambda_{n}}\bigg{)}dx\leq 1,\ \forall n\in\mathbb{N}.\end{equation}
   La suite $(v_{n})$ définie par $v_{n}(x):=G\bigg{(}\frac{|u(x)|}{\lambda_{n}}\bigg{)}$ satisfait
   $$\ds\lim_{n\rightarrow\infty}v_{n}(x)=\begin{cases}
                                            +\infty & \mbox{si }\ u(x)\neq0  \\
                                            0 & \mbox{si}\ u(x)=0.
                                          \end{cases}$$
   Supposons maintenant que la mesure de Lebesgue de l’ensemble $\{x\in\mathbb{R}^{N}; |u(x)|>0\}$ est non nulle. En appliquant le théorème de convergence monotone, on trouve que $$\ds\lim_{n\rightarrow\infty}\int_{\{x\in\mathbb{R}^{N}; |u(x)|>0\}}v_{n}(x)dx=+\infty,$$
   ce qui contredit \eqref{35}. Donc $u$ est nulle presque pour tout $x\in\mathbb{R}^{N}$.

 \textbf{2)} Soient $\alpha\ne0$ dans $\mathbb{R}$ et  $f\in L^{G}(\Omega)$; on a
\begin{align*}
\|\alpha f\|_{(G)}&=\inf\{k>0\,;\int_{\Omega}G(\frac{|\alpha f(x)|}{k})dx\leq1\}\\&
=|\alpha|\inf\{\frac{k}{|\alpha|}>0;\ \int_{\Omega}G(\frac{|f(x)|}{\frac{k}{|\alpha|}})dx\leq1\}\\
&=|\alpha|\inf\{\beta>0;\ \int_{\Omega} G(\frac{|f(x)|}{\beta})dx\leq1\}\\&
=|\alpha|\|f\|_{(G)}.\\
\end{align*}

\textbf{3)} Pour achever la démonstration du lemme, on établit l’inégalité triangulaire. Pour cela, il suffit de montrer que
\begin{equation}\label{36}
  \int_{\Omega}G\bigg{(}\frac{u(x)+v(x)}{\|u\|_{(G)}+\|v\|_{(G)}}\bigg{)}dx\leq1.
\end{equation}
En utilisant la convexité de la fonction $G$, on a
\begin{align*}
  \int_{\Omega}G\bigg{(}\frac{u(x)+v(x)}{\|u\|_{(G)}+\|v\|_{(G)}}\bigg{)}dx & \leq\int_{\Omega}G\bigg{(}\frac{\|u\|_{(G)}}{\|u\|_{(G)}+
  \|v\|_{(G)}}\frac{|u(x)|}{\|u\|_{(G)}}+\frac{\|v\|_{(G)}}{\|u\|_{(G)}+\|v\|_{(G)}}\frac{|v(x)|}{\|v\|_{(G)}}\bigg{)}dx\\
  &\leq \frac{\|u\|_{(G)}}{\|u\|_{(G)}+\|v\|_{(G)}}\int_{\Omega}G\bigg{(}\frac{|u(x)|}{\|u\|_{(G)}}\bigg{)}dx \\&+ \frac{\|v\|_{(G)}}{\|u\|_{(G)}+\|v\|_{(G)}}\int_{\Omega}G\bigg{(}\frac{v(x)}{\|v\|_{(G)}}\bigg{)}dx\leq1
\end{align*}
d'où \eqref{36}.
 \end{proof}

 \begin{lem}\label{lem2}
   Soit $G$ une $N$-Fonction et $u\in L^{G}(\Omega)$. Alors
   \begin{enumerate}
     \item [(i)] $\rho(u;G)\leq\|u\|_{(G)}$ si $\|u\|_{(G)}\leq1$,
     \item [(ii)]$\rho(u;G)\geq\|u\|_{(G)}$ si $\|u\|_{(G)}>1$.
   \end{enumerate}
 \end{lem}

 \begin{proof}
   \textbf{(i)} Soit $\alpha=\|u\|_{(G)}$ et $t=\frac{u(x)}{\|u\|_{(G)}}$ dans la relation
   $$G(\alpha t)\leq\alpha G(t),\ \alpha\in[0,1],\ t\geq0,$$
   et on intègre sur $\Omega$, sachant que
   $$\rho\bigg{(}\frac{1}{\|u\|_{(G)}}u;G\bigg{)}=\int_{\Omega}G\bigg{(}\frac{|u(x)|}{\|u\|_{(G)}}\bigg{)}dx\leq1,$$
   on obtient $$\int_{\Omega}G(|u(x)|)dx\leq\|u\|_{(G)}\int_{\Omega}G\bigg{(}\frac{|u(x)|}{\|u\|_{(G)}}\bigg{)}dx\leq\|u\|_{(G)},$$
   d'où l'assertion $(i)$.\\
   \textbf{(ii)} On sait que
   $$G(\beta t)\geq\beta G(t),\ \beta>1,\ t\geq0,$$
    si $\|u\|_{(G)}>1$, on pose $\beta=\|u\|_{(G)}-\varepsilon$ et $t=\frac{u(x)}{\|u\|_{(G)}-\varepsilon}$ pour $\varepsilon$ assez petit, on obtient \begin{align*}\int_{\Omega}G(|u(x)|)dx&\geq(\|u\|_{(G)}-\varepsilon)\int_{\Omega}G\bigg{(}\frac{u(x)}{\|u\|_{(G)}-\varepsilon}\bigg{)}dx\\&>\|u\|_{(G)}-\varepsilon.\end{align*}
 \end{proof}

 \begin{thm}\label{thm2}
   La norme d'Orlicz est équivalente à la norme de Luxemburg. Plus précisément,
   \begin{equation}\label{22}
   \|u\|_{(G)}\leq\|u\|_{G}\leq2\|u\|_{(G)}.
   \end{equation}
 \end{thm}

 \begin{proof}
   On sait que $\|u\|_{G}\leq\rho(u;G)+1$. Appliquant cette inégalité à $\omega=\frac{u}{\|u\|_{(G)}}$, on obtient$$\|\omega\|_{G}\leq\rho(\omega;G)+1\ \text{et}\ \rho(\omega;G)\leq1,$$ ainsi $\|u\|_{G}\leq2\|u\|_{(G)}$.
 \end{proof}

 \begin{rem}\label{rem8}
   D'après le lemme \ref{lem2} on a $$\rho(u;G^{*})\leq1\ \text{si et seulement si}\ \|u\|_{(G^{*})}\leq1.$$
    On peut alors redéfinir la norme d'Orlicz sous la forme suivante
   \begin{equation}\label{23}
     \|u\|_{G}=\ds\sup_{\|v\|_{(G^{*})}\leq1}\int_{\Omega}|u(x)v(x)|dx.
   \end{equation}
 \end{rem}
 Par le même argument introduit dans le théorème \ref{thm1}  et en vertu de \eqref{23}, nous pourrons obtenir des versions modifiées de l'inégalité de Hölder.
 \begin{thm}
   Soit $(G,G^{*})$ un couple de $N$-Fonction. Alors \begin{equation}\label{30}
                                                       \int_{\Omega}|u(t)v(t)|dt\leq 2\|u\|_{(G)}\|v\|_{(G^{*})}
                                                     \end{equation}
                                                     \begin{equation}\label{24}
                                                        \int_{\Omega}|u(t)v(t)|dt\leq \|u\|_{G}\|v\|_{(G^{*})}
                                                     \end{equation}
                                                     \begin{equation}\label{25}
                                                        \int_{\Omega}|u(t)v(t)|dt\leq \|u\|_{(G)}\|v\|_{G^{*}}
                                                     \end{equation}
                                                     $\forall$ $u\in L^{G}(\Omega)$ et $v\in L^{G^{*}}(\Omega)$.
 \end{thm}
\begin{proof}
Soient $u\in L^{G}(\Omega)$, $v\in L^{G^{*}}(\Omega)$, d'après l'inégalité de Young et la remarque \ref{rem7} on a
$$\int_{\Omega}\frac{|u(t)v(t)|}{\|u\|_{(G)}\|v\|_{(G^{*})}}dt\leq\int_{\Omega} G\bigg{(}\frac{|u(t)|}{\|u\|_{(G)}}\bigg{)}dt+ \int_{\Omega} G^{*}\bigg{(}\frac{|v(t)|}{\|v\|_{(G^{*})}}\bigg{)}dt\leq 2,$$ d'où \eqref{30}.\\
D'après la remarque \eqref{rem8} on a \begin{align*}
                                          \int_{\Omega}|u(t)v(t)|dt&=\|v\|_{(G^{*})}\int_{\Omega}\frac{|v(t)|}{\|v\|_{(G^{*})}}|u(t)|dt\\
                                          &\leq \|v\|_{(G^{*})}\|u\|_{G}
                                      \end{align*} d'où \eqref{24}. De m\^{e}me pour \eqref{25}.

\end{proof}

 \subsection{Convergence dans $L^{G}(\Omega)$}
 La convergence usuelle d'une suite $(u_{n})$  par rapport à la norme d'Orlicz dans l'espace  $L^{G}(\Omega)$ peut être introduite comme suit:
$$u_{n}\rightarrow u\ \text{dans}\ L^{G}(\Omega)\ \text{si et seulement si}\ \ds\lim_{n\rightarrow\infty}\|u_{n}-u\|_{G}=0.$$
Nous allons maintenant introduire un autre type de convergence dans les espaces d'Orlicz.
\begin{dfn}
  Soient $G$ une $N$-Fonction et $(u_{n})_{n\in\mathbb{N}}$ une suite de $L^{G}(\Omega)$. On dit que $(u_{n})$ converge en module vers $u$ si
  \begin{equation}\label{26}
   \ds\lim_{n\rightarrow\infty}\rho(u_{n}-u;G)= \ds\lim_{n\rightarrow\infty}\int_{G}G(|u_{n}(x)-u(x)|)dx=0.
  \end{equation}
\end{dfn}
\begin{rem}
  D'après le lemme \ref{lem2} et le théorème \ref{thm2}, $\forall\omega\in L^{G}(\Omega)$ tel que $\|\omega\|_{G}\leq1$,
  $$\rho(\omega;G)\leq \|\omega\|_{G}.$$ Ainsi, pour toute suite $(u_{n})$ qui converge vers $u$ dans $ L^{G}(\Omega)$, on a
  $$  \rho(u_{n}-u;G)\rightarrow 0.$$
  En d'autres termes, la convergence en norme est plus forte que celle en module.
\end{rem}
 On déduit la proposition suivante.

 \begin{pro}\label{pro1}
  Soit $G$ une $N$-Fonction. Alors, la convergence dans  $L^{G}(\Omega)$ implique toujours la convergence en module.
 \end{pro}

 Le sens inverse de la proposition \ref{pro1} est assuré par le théorème suivant.
  \begin{thm}\label{thm4}
 Soit $G$ une $N$-Fonction. Si $G$ satisfait la condition $\triangle_{2}$ ($G$ satisfait $\triangle_{2,0}$ si $mes(\Omega)=+\infty$), alors la convergence en norme est équivalente à la convergence en module.
 \end{thm}
 Pour la démonstration nous avons besoin du lemme suivant.
 \begin{lem}\label{pro4}
   Soit $G$ une $N$-Fonction satisfaisant la condition $\triangle_{2}$ ($G$ satisfait $\triangle_{2,0}$ si $mes(\Omega)=+\infty$). Soit $\omega\in L^{G}(\Omega)$; s'il existe $m\in\mathbb{N}$ tel que \begin{equation}\label{67}\rho(\omega;G)\leq\frac{1}{k^{m}},\ k\ \text{fixée par la définition \ref{dfn2} },\end{equation} alors, \begin{equation}\label{69}
            \|\omega\|_{G}\leq\frac{c}{2^{m}},
          \end{equation}
          où $c=2$ si $mes(\Omega)=+\infty$ et $c=G(T)mes(\Omega)+2$ si $mes(\Omega)<\infty$.
 \end{lem}
 \begin{proof}
   Soit $m\in\mathbb{N}$ fixé. Si $mes(\Omega)<\infty$, on note $$\Omega_{1}=\{x\in\Omega;\ 2^{m}|\omega(x)|\leq T\}.$$
   Alors
   \begin{equation}\label{70}
     G(2^{m}|\omega(x)|)\leq G(T)\ \forall\ x\in\Omega_{1},
   \end{equation}
   et
   \begin{equation}\label{71}
      G(2^{m}|\omega(x)|)\leq k^{m}G(|\omega(x)|)\ \forall\ x\in\Omega\setminus\Omega_{1}.
   \end{equation}
   Ainsi \begin{align*}
           \int_{\Omega}G(2^{m}|\omega(x)|)dx&= \int_{\Omega_{1}}G(2^{m}|\omega(x)|)dx+ \int_{\Omega\backslash\Omega_{1}}G(2^{m}|\omega(x)|)dx\\
           &\leq\int_{\Omega_{1}}G(T)dx+k^{m}\int_{\Omega\backslash\Omega_{1}}G(|\omega(x)|)dx\\
           &\leq G(T)mes(\Omega_{1})+k^{m}\int_{\Omega}G(|\omega(x)|)dx,
         \end{align*}
         c.à.d \begin{equation}\label{72}
                 \rho(2^{m}\omega;G)\leq G(T)mes(\Omega)+k^{m}\rho(\omega;G).
               \end{equation}
   Si $T=0$ on obtient $$ \rho(2^{m}\omega;G)\leq  k^{m}\rho(\omega;G).$$
   Supposons que $u$ satisfait \eqref{67} on obtient $$\rho(2^{m}\omega;G)\leq c-1,$$ où
    \begin{equation}\label{133}
      c=2\ \text{si}\ mes(\Omega)=+\infty\ \text{et}\ c=G(T)mes(\Omega)+2\ \text{si}\ mes(\Omega)<\infty
    \end{equation}

   D'autre part, pour $v\in K^{G}(\Omega)$ vérifiant $\rho(v;G^{*})\leq1$ on a
   $$  \int_{\Omega}2^{m}|\omega(x)||v(x)|dx\leq\rho(2^{m}\omega;G)+1\leq c.$$
   Ceci donne $$\|2^{m}\omega\|_{G}\leq c$$ et $$\|\omega\|_{G}\leq\frac{c}{2^{m}}.$$
    \end{proof}
\begin{proof}[\textbf{Démonstration du théorème \ref{thm4}}]
En vu de la proposition \ref{pro1} il nous reste à montrer que la convergence en module implique la convergence dans $L^{G}(\Omega)$. Soient $(u_{n})_{n\in\mathbb{N}}\subset L^{G}(\Omega)$ et $u\in L^{G}(\Omega)$ telle que $$\ds\lim_{n\rightarrow+\infty}\rho(u_{n}-u;G)=0.$$
Soit $\varepsilon>0$, $\exists m\in\mathbb{N},$ $\exists n_{0}\in\mathbb{N}$
 tel que $$\rho(u_{n}-u;G)\leq \frac{1}{k^{m}},\ \forall n\geq n_{0}\ \text{et}\ \varepsilon>\frac{c}{2^{m}},$$ où $c$ est définie dans \eqref{133}.
Soit $\omega=u_{n}-u$, d'après \eqref{69}, on a $$\|u_{n}-u\|_{G}\leq\frac{c}{2^{m}}\leq\varepsilon,\ \forall n\geq n_{0}.$$

\end{proof}

 \begin{rem}\label{rem1}
   Soit $u\in K^{G}(\Omega)$.
   On considère la suite de fonction $(u_{n})_{n\in\mathbb{N}}$ définie par
                                 \begin{equation}\label{26}
                                    u_{n}(x)=\begin{cases}
                                                 u(x) & \mbox{si }\ |u(x)|\leq n,\ \text{et}\ |x|\leq n,\ x\in\Omega \\
                                                 0 & \mbox{sinon}.
                                             \end{cases}
                                  \end{equation}
      $(u_{n})$ converge en module vers $u$. \text{En effet}, comme $$G(|u(x)-u_{n}(x)|)\leq G(|u(x)|),$$
     alors d'après le théorème de convergence dominée
           \begin{equation}\label{27}
             \ds\lim_{n\rightarrow\infty}\rho(u_{n}-u;G)=\ds\lim_{n\rightarrow\infty}\int_{\Omega}G(|u_{n}(x)-u(x)|)dx=0.
           \end{equation}
 \end{rem}
 Nous annonçons l'assertion suivante.
 \begin{cor}\label{cor1}
   Soit $G$ une $N$-Fonction. Alors
   \begin{enumerate}
     \item Toute fonction $u$ de la classe $K^{G}(\Omega)$ peut \^{e}tre approchée au sens de la convergence en module par une suite de fonctions bornées et à supports bornés dans $\overline{\Omega}$.
     \item Si $G$ satisfait la condition $\triangle_{2}$, alors l'ensemble des fonctions mesurables bornées sur $\Omega$ et à supports bornés dans $\overline{\Omega}$ est dense dans $L^{G}(\Omega)$.
   \end{enumerate}
 \end{cor}
 \begin{proof}
   L'assertion $(1)$ découle de la Remarque \ref{rem1}. Si $G$ satisfait la condition $\triangle_{2}$, on a $L^{G}(\Omega)=K^{G}(\Omega)$ et la convergence en module implique la convergence en norme.
    D'après l'assertion $(1)$ l'ensemble des fonctions bornées sur $\Omega$ et à support bornés dans $\overline{\Omega}$ est dense dans $L^{G}(\Omega)$.
 \end{proof}
 Nous allons maintenant introduire la notion d'un ensemble \textit{borné en module}.
 \begin{dfn}
   Soit $G$ une $N$-Fonction. Un sous-ensemble $\Re$ de $L^{G}(\Omega)$ est dit borné en module, s'il existe $c>0$ telle que
   $$\int_{\Omega}G(|u(x)|)dx\leq c,\ \forall u\in\Re.$$
 \end{dfn}
 \begin{rem}\label{rem6}
    Soit $G$ une $N$-Fonction. Alors, il découle de la proposition \ref{pro6}, qu'un sous-ensemble $\Re$ de $L^{G}(\Omega)$ borné en module est automatiquement borné en norme et on a $ \|u\|_{G}\leq \rho(u;G)+1$ donc $$\|u\|_{G}\leq c+1,\ \forall u\in\Re.$$
 \end{rem}
 \begin{pro}
    Si $G$ satisfait la condition $\triangle_{2}$ ($G$ satisfait $\triangle_{2,0}$ si $mes(\Omega)=+\infty$). Alors un ensemble $\Re$ est borné dans $L^{G}(\Omega)$ si et seulement si est borné en module.
 \end{pro}
\begin{proof}On suppose que  $\|u\|_{G}\leq a$, $\forall u\in\Re$. On choisit $m\in\mathbb{N}$ tel que $a\leq2^{m}$, $\omega=\frac{u}{2^{m}}$ on a

$$\rho\bigg{(}\frac{u}{2^{m}};G)\bigg{)}\leq\rho\bigg{(}\frac{u}{\|u\|_{G}};G)\bigg{)}\leq 1,$$
par le m\^{e}me argument que dans la démonstration du lemme \ref{pro4} on obtient $$\rho(u;G)\leq\begin{cases}
                       G(T)mes(\Omega)+k^{m} & \mbox{si }\ mes(\Omega)<\infty \\
                       k^{m} & \mbox{si}\ mes(\Omega)=+\infty,
                     \end{cases}
                     ,\ \forall u\in\Re$$
\end{proof}
\subsection{Comparaison des espaces d'Orlicz}
On sait que si $mes(\Omega)<\infty$ et $1<p<q$ on a $$L^{q}(\Omega)\hookrightarrow L^{p}(\Omega).$$
Nous pouvons obtenir un résultat similaire pour les espaces d'Orlicz en utilisant la relation $\prec$ introduite dans la définition \ref{dfn1}.
\begin{thm}\label{thm22}
  Soient $G_{1}$, $G_{2}$ deux $N$-Fonctions et $mes(\Omega)<+\infty$. Alors $$L^{G_{1}}(\Omega)\subset L^{G_{2}}(\Omega)$$ si et seulement si \begin{equation}\label{87}
                                                                                                                        G_{2}\prec G_{1}.
                                                                                                                      \end{equation}
\end{thm}
\begin{proof}
  On suppose qu'il existe $c>0$ et $T\geq0$ tels que \begin{equation}\label{86}
                                                       G_{2}(x)\leq G_{1}(cx),\ \forall x\geq T.
                                                     \end{equation}
Soit $u\in L^{G_{1}}(\Omega)$. Il existe $\gamma>0$ tel que $$\gamma u\in  K^{G_{1}}(\Omega),$$ c.à.d, $\rho(\gamma u;G_{1})<\infty$. Soit
$$\Omega_{1}=\{x\in\Omega;\ |u(x)|\leq c\frac{T}{\gamma}\}.$$  Pour tout $ x\in\Omega\backslash\Omega_{1}$, on a
$$G_{2}(\frac{\gamma}{c}|u(x)|)\leq G_{1}(\gamma|u(x)|)$$ ceci donne, \begin{align*}
                                                                            \int_{\Omega}G_{2}(\frac{\gamma}{c}|u(x)|)dx & = \int_{\Omega_{1}}G_{2}(\frac{\gamma}{c}|u(x)|)dx+ \int_{\Omega\backslash\Omega_{1}}G_{2}(\frac{\gamma}{c}|u(x)|)dx\\
                                                                            &\leq G_{2}(T)mes(\Omega)+\int_{\Omega\backslash\Omega_{1}}G_{1}(\gamma|u(x)|)dx\leq
                                                                            G_{2}(T)mes(\Omega)+\rho(\gamma u;G)<\infty,
                                                                          \end{align*}
                                       c.à.d, $$ \frac{\gamma}{c}u\in K^{G_{2}}(\Omega)$$ d'où $u\in L^{G_{2}}(\Omega)$.\\ \\
  \text{Inversement}\\
  Supposons que \eqref{87} n'est pas vérifiée, alors il existe une suite $(t_{n})_{n\in\mathbb{N}}$ telle que $$0<t_{1}<t_{2}<...,\ \ds\lim_{n\rightarrow+\infty}t_{n}=+\infty,$$ et \begin{equation}\label{88}
                                                     G_{2}(t_{n})>G_{1}(2^{n}t_{n}),\ n\in\mathbb{N}.
                                                   \end{equation}
     Appliquant la formule \eqref{5} pour  $G_{1}$ et $t=2^{n}nt_{n}$, $\alpha=2^{-n}$ on obtient $$G_{1}(nt_{n})\leq 2^{-n}G_{1}(2^{n}nt_{n})$$
     et
                                  \begin{equation}\label{89}
                                     G_{2}(t_{n})>2^{n}G_{1}(nt_{n}).
                                  \end{equation}
       D'après la propriété des sous-ensembles mesurables de mesure finis, il existe une
      suite de sous-ensembles disjoints $\Omega_{n}$ de $\Omega$, $n\in\mathbb{N}$ telle que
     $$mes(\Omega_{n})=\frac{G_{1}(t_{1})mes(\Omega)}{2^{n}G_{1}(t_{n})}.$$
     On remarque que $$\sum_{n=1}^{\infty}mes(\Omega_{n})<\sum_{n=1}^{\infty}2^{-n}mes(\Omega)=mes(\Omega).$$
     Soit la fonction $u$ définie par $$u(x)=\begin{cases}
                                     nt_{n} & \mbox{si }\ x\in\Omega_{n},\ n\in\mathbb{N}  \\
                                     0 & \mbox{sinon},
                                   \end{cases} $$
          $u\in L^{G_{1}}(\Omega)$, en effet
          \begin{align*}
            \int_{\Omega}G_{1}(|u(x)|)dx&=\sum_{n=1}^{\infty}\int_{\Omega_{n}}G_{1}(nt_{n})dx=\sum_{n=1}^{\infty}G_{1}(nt_{n})mes(\Omega_{n})\\
            &=\sum_{n=1}^{\infty}\frac{1}{2^{n}}G_{1}(t_{1})mes(\Omega)<\infty.
          \end{align*}
          D'autre part, $u\in L^{G_{2}}(\Omega)$. En effet, pour un choix de $\gamma$ arbitraire avec $0<\gamma\leq1$ on a $$\gamma u\notin L^{G_{2}}(\Omega).$$ Pour cela il suffit de choisir $m\in\mathbb{N}$ tel que $\gamma>\frac{1}{m}$, et en vertu de \eqref{89}, on a
          \begin{align*}
            \int_{\Omega}G_{2}(\gamma|u(x)|)dx & =\sum_{n=1}^{\infty}G_{2}(\gamma n t_{n})mes(\Omega_{n})\\
            &\geq\sum_{n=m}^{\infty}G_{2}( t_{n})mes(\Omega_{n})>\sum_{n=m}^{\infty}2^{n}G_{1}(n t_{n})mes(\Omega_{n})\\
            &=\sum_{n=m}^{\infty}G_{1}(t_{1})mes(\Omega)=\infty.
          \end{align*}
          D'où, le résultat.
\end{proof}
Le théorème suivant est une conséquence du théorème \ref{thm22}.
\begin{thm}
  Soient $G_{1}$, $G_{2}$ deux $N$-Fonctions et $mes(\Omega)<\infty$. Alors \begin{equation}\label{90}
                                                      L^{G_{1}}(\Omega)= L^{G_{2}}(\Omega)
                                                    \end{equation} si et seulement si $$G_{2}\sim G_{1}.$$
\end{thm}
\begin{rem}
  De l'égalitée \eqref{90}, on peut montrer que les normes des deux espaces considérés sont équivalentes. Ceci découlera du théorème suivant qui fournit que l'inclusion  $L^{G_{1}}(\Omega)\subset L^{G_{2}}(\Omega)$ implique l'injection $$ L^{G_{1}}(\Omega)\hookrightarrow L^{G_{2}}(\Omega).$$
\end{rem}
\begin{thm}
  Si $L^{G_{1}}(\Omega)\subset L^{G_{2}}(\Omega)$, alors il existe $k>0$ tel que
  \begin{equation}\label{121}\|u\|_{G_{2}}\leq k\|u\|_{G_{1}},\ \forall u\in L^{G_{1}}(\Omega).\end{equation}
\end{thm}
\begin{proof}
Supposons que \eqref{121} n'est pas vérifiée, il existe une suite de fonction $(u_{n})$ de $L^{G_{1}}(\Omega)$ telle que $$u_{n}\geq0,\ \|u_{n}\|_{G_{1}}\leq 1\ \text{et}\ \|u_{n}\|_{G_{2}}>n^{3},\ \forall n\in\mathbb{N}.$$
D'aprés la propriétée de Riesz-Fischer on a $\sum\frac{u_{n}}{n^{2}}$ converge dans $L^{G_{1}}(\Omega)$ vers une fonction $u\in L^{G_{1}}(\Omega)$, comme  $L^{G_{1}}(\Omega)\subset L^{G_{2}}(\Omega)$ alors $u\in L^{G_{2}}(\Omega)$. Pour $n$ assez grand on a $0\leq n^{-2}u_{n}\leq u$ donc
$\|u\|_{G_{2}}\geq n^{-2}\|u_{n}\|_{G_{2}}>n$, contradiction.

\end{proof}

\section{L'espace $E^{G}(\Omega)$}
Dans cette section on va introduire un nouveau espace de fonctions qui joue un r\^{o}le très important dans l'étude de quelques propriétés des espaces d'Orlicz.
\begin{dfn}
    Soit $G$ une $N$-Fonction.
    $$E^{G}(\Omega):=\{u\in L^{G}(\Omega);\ \lambda u\in K^{G}(\Omega)\ \text{pour tout}\ \lambda>0\}.$$
\end{dfn}
\begin{rem}
  Il est clair que $E^{G}(\Omega)$ est un espace vectoriel et on a les inclusions suivantes
  \begin{equation}\label{28}
    E^{G}(\Omega)\subset K^{G}(\Omega)\subset L^{G}(\Omega).
  \end{equation}
\end{rem}
 Les inclusions réciproques dans \eqref{28} ne sont pas validées en général.
 \begin{pro}\label{pro2}
  Soit $G$ une $N$-Fonction vérifiant la condition $\triangle_{2}$ ($G$ satisfait $\triangle_{2,0}$ si $mes(\Omega)=+\infty$). Alors
  \begin{equation}\label{29}
     E^{G}(\Omega)= K^{G}(\Omega)= L^{G}(\Omega).
  \end{equation}
 \end{pro}
 \begin{proof}
   Soit $u\in L^{G}(\Omega)$, $\exists\lambda>0$ tel que $\lambda u\in K^{G}(\Omega)$. Soit $K>0$, d'après la condition $\triangle_{2}$, il existe $C>0$ tel que
   $$G(K|u(x)|)\leq C G(\lambda|u(x)|),$$
    donc $$\int_{\Omega}G(K|u(x)|)dx\leq C\int_{\Omega}G(|\lambda u(x)|)dx<\infty.$$

     Ainsi $u\in E^{G}(\Omega)$.
 \end{proof}
 \begin{rems}
   $\bullet$ Si $G$ est une $N$-Fonction telle que $G$ ne vérifie pas $\triangle_{2}$, alors  $$E^{G}(\Omega)\varsubsetneq K^{G}(\Omega)\varsubsetneq L^{G}(\Omega),$$
   car $E^{G}$ et $L^{G}$ sont des espaces vectoriels mais $K^{G}$ ne l'est pas en général.\\
   $\bullet$ $E^{G}$ est le plus grand sous-espace de $L^{G}$ contenu dans $K^{G}$.
 \end{rems}
 \textbf{Notations:} On note par $S(\Omega)$ l'ensemble des fonctions simples (voir Appendice) définies sur $\Omega$.\\

 Voici une autre caractérisation de l'espace $E^{G}$.
 \begin{thm}\label{thm21}
   Soit $G$ une $N$-Fonction.
     Alors $E^{G}$ est la fermeture dans $L^{G}(\Omega)$  de $S(\Omega)$, c.à.d, $E^{G}=\overline{S(\Omega)}$.
 \end{thm}
 \begin{proof}
   Soit $u\in E^{G}(\Omega)$. Sans perte de généralité, on peut supposer que $u\geq0$. Par conséquent, il existe une suite de fonctions simples $(u_{n})_{n\in\mathbb{N}}$ telle que $$0\leq u_{n}\uparrow u,\ p.p\ \text{sur}\ \Omega.$$
    Pour tout $\lambda>0$ et $n\in\mathbb{N}$ on a $$\lambda u_{n}\in E^{G}(\Omega)$$ et $$G(\lambda u)\geq G(\lambda(u-u_{n})).$$
   D'après le théorème du convergence dominée, on a, $$\int_{\Omega}G(\lambda(u-u_{n}))dx\rightarrow0,\ n\rightarrow+\infty.$$
   Fixons $\epsilon>0$ et prenons $\lambda>\frac{2}{\epsilon}$. Il existe $n_{0}\in\mathbb{N}$ telle que $\forall n>n_{0}$, on a
   $$\int_{\Omega}G(\lambda(u-u_{n}))dx\leq1.$$ D'après \eqref{16} on a
   \begin{align*}
     \|u-u_{n}\|_{G} & =\frac{1}{\lambda}\|\lambda(u-u_{n})\|_{G}\\
     &\leq\frac{1}{\lambda}\bigg{(}\int_{\Omega}G(\lambda|u-u_{n}|)dx+1\bigg{)}\\
     &\leq\frac{2}{\lambda}<\epsilon.
   \end{align*}
   D'où, $u\in \overline{S(\Omega)}$.\\
   \text{Inversement}:\\
   Soit $u\in\overline{S(\Omega)}$ et $\lambda>0$. D'après le théorème \ref{thm2}, la norme d'Orlicz est équivalente à la norme de Luxemberg, donc, il existe une fonction étagée $u_{0}$ telle que $$\|2\lambda(u-u_{0})\|_{(G)}\leq1.$$
   D'après le lemme \ref{lem2} on obtient $$\rho(2\lambda(u-u_{0});G)\leq1.$$ En particulier, on a $$2\lambda(u-u_{0})\in K^{G}(\Omega).$$ D'autre part, $S(\Omega)\subset K^{G}(\Omega)$, donc $$2\lambda u_{0}\in K^{G}(\Omega).$$
   Finalement, par la convexité de la classe d'Orlicz, on a $$\lambda u=\frac{2\lambda(u-u_{0})+2\lambda u_{0}}{2}\in K^{G}(\Omega).$$
   Alors $u\in E^{G}(\Omega)$. D'où le résultat.
 \end{proof}
 \section{Propriétés de l'espace d'Orlicz}
 \subsection{Complétude de $L^{G}(\Omega)$}
 \begin{thm}

   L'espace d'Orlicz $L^{G}(\Omega)$ est un espace de Banach.
 \end{thm}
 \begin{proof}
Soit $(u_{n})_{n\in\mathbb{N}}$ une suite de Cauchy de $L^{G}(\Omega)$, c.à.d, $\forall\epsilon>0$, $\exists n_{0}=n_{0}(\epsilon)\in\mathbb{N}$
tel que pour $v\in L^{G^{*}}(\Omega)$, $\rho(v;G^{*})\leq1$ et pour tout $n>n_{0}$, $m>n_{0}$,
\begin{equation}\label{76}
      \int_{\Omega}|u_{m}(x)-u_{n}(x)||v(x)|dx<\epsilon.
    \end{equation}
    Nous décomposons $\Omega$ en des sous-ensembles $\Omega_{n}$ tels que $\Omega=\cup_{n=1}^{\infty}\Omega_{n}$, où $\Omega_{i}\neq\Omega_{j}$ si $i\neq j$.\\ On choisit $K>0$ tel que $$G^{*}(K)\leq \frac{1}{mes(\Omega_{1})}$$ On considère la fonction $$v_{1}(x)=\begin{cases}
                                                                                                                       K & \mbox{si }\ x\in\Omega_{1}\\
                                                                                                                       0 & \mbox{si}\ x\in\Omega\backslash\Omega_{1}.
                                                                                                                     \end{cases}$$
  Alors $$\rho(v_{1};G^{*})=\int_{\Omega}G^{*}(|v_{1}(x)|)dx=\int_{\Omega_{1}}G^{*}(K)dx=G^{*}(K)mes(\Omega_{1})\leq1,$$ remplaçons $v_{1}$ dans \eqref{76} on obtient $$\int_{\Omega_{1}}|u_{m}(x)-u_{n}(x)|dx<\frac{\epsilon}{K}\ \text{pour}\ n>n_{0},\ m>n_{0}.$$
  Cela signifie que $(u_{n})_{n\in\mathbb{N}}$ est de Cauchy dans $L^{1}(\Omega_{1})$, et comme cet espace est complet, il existe une sous-suite
   $(u_{n,1})$ de $(u_{n})$ qui converge vers une fonction $u_{1}$ dans $ L^{1}(\Omega_{1})$ et presque partout sur $\Omega_{1}$.
   Nous répétons la m\^{e}me procédure avec $\Omega_{2}$ et $(u_{n,1})$, il existe une sous-suite $(u_{n,2})$ de $(u_{n,1})$  qui converge vers une fonction $u_{2}$ dans $ L^{1}(\Omega_{2})$ et presque partout sur $\Omega_{2}$. \\
   Répétant cette procédure, nous obtenons une suite de sous suites de $(u_{n})$,
    toute sous-suite $(u_{n,k})$ converge vers $u_{k}$  dans $ L^{1}(\Omega_{k})$ presque partout sur $\Omega_{k}$.\\
    Soit $$u(x)=u_{k}(x)\chi_{\Omega_{k}}(x)$$ on a $$u_{n,k}(x)\rightarrow u(x)\ p.p\ \text{sur}\ \Omega,\ \text{quand}\ n\rightarrow+\infty$$ On remplace $u_{n}$ par $u_{n,k}$ dans \eqref{76}, d'après le lemme de Fatou on a $$\|u-u_{m}\|_{G}\leq\epsilon,\ \ \forall\ m\geq n_{0}.$$ Cela signifie que
   \begin{enumerate}
     \item [(i)] $u_{m}-u\in L^{G}(\Omega)$ et par conséquent $$u=u_{m}-(u_{m}-u)\in L^{G}(\Omega);$$
     \item [(ii)] $\ds\lim_{n\rightarrow+\infty}\|u_{n}-u\|_{G}=0$.
   \end{enumerate}
   \end{proof}
 \subsection{Séparabilitée de $L^{G}(\Omega)$}
 Dans ce paragraphe, nous caractérisons quand est ce que l'espace d'Orlicz est séparable.
 \begin{thm}
   Soit $G$ une fonction qui satisfait la condition $\triangle_{2}$. Alors l'espace d'Orlicz  $L^{G}(\Omega)$ est séparable.
 \end{thm}

 \begin{proof}
   Soit $S$ l'ensemble de tous les cubes ouverts dans $\mathbb{R}^{N}$ d'arrêtes parallèles aux axes de coordonnées et des sommets de coordonnées rationnels.\\L'ensemble $S$ est évidemment dénombrable. On note par $\mathfrak{F}$ l'ensemble des fonctions de la forme
   $$f(x)=\sum_{i=1}^{\sigma}r_{i}\chi_{C_{i}}(x),\ x\in\mathbb{R}^{N},$$
   où $\sigma=\sigma(f)\in\mathbb{N}$, $r_{i}$ est un nombre rationnel, $C_{1},C_{2},...,C_{\sigma}$ sont des cubes de $S$ deux à deux disjoints. $\mathfrak{F}$ est aussi dénombrable. On montre que $\mathfrak{F}$ dense dans  $L^{G}(\Omega)$.\\ \\
   \textbf{\'{E}tape 1}\\ \\
   Soient $\varepsilon>0$, $u\in L^{G}(\Omega)$.
     D'après le corollaire \ref{cor1} il existe une fonction $u_{0}$ définie sur $\Omega$ telle que
   \begin{enumerate}
     \item $u_{0}(x)=u(x)$ si $x\in\Omega_{0}$, $\Omega_{0}\subset\Omega$, $mes(\Omega_{0})<+\infty$,
     \item $u_{0}(x)=0$ si  $x\in\Omega\backslash\Omega_{0}$,

   \end{enumerate}
et

    \begin{equation}\label{73}
          \|u_{0}-u\|_{G}\leq\frac{\varepsilon}{3}.
    \end{equation}
    \textbf{\'{E}tape 2}\\ \\
    On écrit $u_{0}$ comme suit $$u_{0}=u_{1}-u_{2},$$ où $u_{1},u_{2}$ sont deux fonctions positives, nulles sur $\Omega\backslash\Omega_{0}$. Il existe une suite de fonctions  $(f_{n})_{n\in\mathbb{N}}$ étagées à valeurs rationnels, croissantes et positives telle que $f_{n}$ est nulle en dehors de $\Omega_{0}$, $\forall n\in\mathbb{N}$ et $$u_{1}(x)=\ds\lim_{n\rightarrow+\infty}f_{n}(x),\ p.p\ \text{sur}\ \Omega_{0}.$$

    Comme $$G(|u_{1}(x)-f_{n}(x)|)\leq G(|u_{1}(x)|)$$
    et $$\ds\lim_{n\rightarrow+\infty} G(|u_{1}(x)-f_{n}(x)|)=0,$$
    d'après le théorème de convergence dominée on a$$\ds\lim_{n\rightarrow+\infty}\int_{\Omega} G(|u_{1}(x)-f_{n}(x)|)dx=0.$$ Utilisant la proposition \ref{pro1}, il existe $n_{1}\in\mathbb{N}$ tel que $\forall\ n\geq n_{1}$, $$\|u_{1}-f_{n}\|_{G}\leq\frac{\varepsilon}{6}.$$
    On procéde de la m\^{e}me manière avec $u_{2}$, il existe une suite de fonctions  $(g_{n})_{n\in\mathbb{N}}$ étagées à valeurs rationnels, croissantes et positives telle que $g_{n}$ est nulle en dehors de $\Omega_{0}$, $\forall n\in\mathbb{N}$ et il existe $n_{2}\in\mathbb{N}$ tels que $\forall\ n\geq n_{2}$,  $$\|u_{2}-g_{n}\|_{G}\leq\frac{\varepsilon}{6}.$$
    Soit $n_{3}=\max(n_{1},n_{2})$, on a $\forall n\geq n_{3}$, $$\|u_{0}-(f_{n}-g_{n})\|_{G}\leq\frac{\varepsilon}{3}$$
    prenons $v_{0}=f_{n_{3}+1}-g_{n_{3}+1}$ on obtient \begin{equation}\label{74}
                                              \|v_{0}-u_{0}\|_{G}\leq\frac{\varepsilon}{3}.
                                            \end{equation}
     La fonction $v_{0}$ a la forme $$v_{0}(x)=\sum_{i=1}^{p}r_{i}\chi_{\Omega_{i}}(x),$$ où $\Omega_{i}\subset\Omega$ $(i=1,2,...,p)$ tels que $\Omega_{i}\cap\Omega_{j}=\varnothing$ si $i\neq j$ et $\cup_{i=1}^{p}\Omega_{i}=\Omega_{0}$; $r_{i}\in\mathbb{Q}$, $r_{i}\neq0$.\\ \\
     \textbf{\'{E}tape 3}\\ \\
    Soit $M\subset\Omega$ un ensemble. Il est clair que $$\rho(\chi_{M};G)=\int_{M}G(1)dx=G(1)mes(M).$$
     Si $mes(M)\rightarrow0$ alors $\rho(\chi_{M};G)\rightarrow0$ et par suite d'après le théorème \ref{thm4} on a $\|\chi_{M}\|_{G}\rightarrow0$. Pour chaque ensemble $\Omega_{i}$ de l'étape 2, il existe un ensemble $S_{i}$ réunion d'un nombre fini de cubes de $S$ tels que $$\|\chi_{\Omega_{i}}-\chi_{S_{i}}\|_{G}\leq\frac{\varepsilon}{3p|r_{i}|},\ i=1,2,...,p.$$
     Soit $$\omega_{0}(x)=\sum_{i=1}^{p}r_{i}\chi_{S_{i}}$$ alors $\omega_{0}\in\mathfrak{F}$ et
     \begin{equation}\label{75}
       \|\omega_{0}-v_{0}\|_{G}=\sum_{i=1}^{p}|r_{i}|\|\chi_{\Omega_{i}}-\chi_{S_{i}}\|_{G}\leq\frac{\varepsilon}{3}.
     \end{equation}
     \textbf{\'{E}tape 4}\\ \\
     De \eqref{73}, \eqref{74}, \eqref{75} on déduit que $$\|u-\omega_{0}\|_{G}\leq\|u-u_{0}\|_{G}+\|u_{0}-v_{0}\|_{G}+\|v_{0}-\omega_{0}\|_{G}<\varepsilon,$$
     d'où le résultat.

 \end{proof}

 \subsection{Dualité et réflexivité du $L^{G}(\Omega)$}

\begin{thm}\label{thm18}
   Soient $(G,G^{*})$ un couple de $N$-Fonctions et $v\in L^{G^{*}}(\Omega)$ fixé. Alors la fonction

  \begin{equation}\label{77}
   F(u)=\int_{\Omega}u(t)v(t)dt,\ \forall u\in L^{G}(\Omega),
  \end{equation}

 est linéaire continue, de plus, \begin{equation}\label{78}\frac{1}{2}\|v\|_{G^{*}}\leq\|F\|\leq\|v\|_{G^{*}}.\end{equation}
\end{thm}

\begin{proof}
  D'après l'inégalité de H\"{o}lder, on a \begin{equation}\label{79}
                                           |F(u)|\leq\|u\|_{G}\|v\|_{G^{*}}\ \ \text{alors}\ \ \|F\|\leq\|v\|_{G^{*}}.
                                         \end{equation}

    Soit $u\in L^{G}(\Omega)$ tel que $\rho(u;G)\leq1$, d'après \eqref{16}, $\|u\|_{G}\leq\rho(u;G)+1\leq2$ et
   $$\|v\|_{G^{*}}=\ds\sup_{\rho(u;G)\leq1}\bigg{|}\int_{\Omega}u(x)v(x)dx\bigg{|}\leq\ds\sup_{\|u\|_{G}\leq2}|F(u)|=\ds\sup_{\|u\|_{G}\leq1}|F(2u)|\leq
   2\|F\|_{G}.$$
\end{proof}

On peut considérer l'espace $L^{G^{*}}(\Omega)$ comme un sous-ensemble de $\bigg{[}L^{G}(\Omega)\bigg{]}^{*}$. Mais en général $$L^{G^{*}}(\Omega) \varsubsetneqq\bigg{[}L^{G}(\Omega)\bigg{]}^{*}.$$

\begin{thm}\label{thm19}
  Supposons que $G$ ne satisfait pas la condition $\triangle_{2}$. Alors il existe une forme linéaire continue sur $L^{G}(\Omega)$ qui ne s'exprime pas sous la forme \eqref{77}.
\end{thm}

\begin{proof}
  Sous les hypothèses du théorème ci-dessus, $E^{G}(\Omega)$ est un sous-espace propre de $L^{G}(\Omega)$, donc il existe une fonction $u_{0}\in L^{G}(\Omega)$ et $u_{0}\notin E^{G}(\Omega)$. Soit $F$ une forme linéaire continue sur $L^{G}(\Omega)$ tel que
  \begin{equation}\label{81}
    F(u_{0})=1,\ F(u)=0\ \text{si}\ u\in E^{G}(\Omega),\ \text{(voir le théorème de Hahn Banach)}.
  \end{equation}
  Supposons que $F$ s'exprime sous la forme \eqref{77} avec $v\in L^{G^{*}}(\Omega)$. Considérons la suite de fonctions $(v_{n})_{n\in\mathbb{N}}$ par $$v_{n}(x)=\begin{cases}
                                              v(x) & \mbox{si }\ |v(x)|\leq n, \\
                                              0 & \mbox{si}\ |v(x)|>n.
                                            \end{cases}$$
   $v_{n}\in E^{G}(\Omega)$, et $$F(v_{n})=\int_{\Omega}v_{n}(x)v(x)dx=0,\ \forall n\in\mathbb{N}.$$
   Ceci donne $v(x)=0$ presque partout $x\in\Omega$ et donc $F(u_{0})=0$
   ce qui contredit \eqref{81}.
\end{proof}

Le théorème suivant est un analogue au théorème de représentation de Riesz dans l'espace de Lebesgue classique.

\begin{thm}\label{thm20}
  Soit $F$ une forme linéaire continue sur $E^{G}(\Omega)$. Alors il existe une unique fonction $v\in L^{G^{*}}(\Omega)$ telle que
  \begin{equation}\label{82}
    F(u)=\int_{\Omega}u(x)v(x)dx,\ \forall u\in E^{G}(\Omega),
  \end{equation}
  c.à.d,  \begin{equation}\label{84}
    L^{G^{*}}(\Omega)=\bigg{[}E^{G}(\Omega)\bigg{]}^{*}.
  \end{equation}
\end{thm}

\begin{proof}
  L'unicité de $v$ est évidente, il nous reste à montrer l'existence. Soit $M$ un sous-ensemble mesurable de $\Omega$ de mesure finie et
  $$\nu(M)=F(\chi_{M}).$$ D'après la proposition \ref{pro5} on a
  $$|\nu(M)|=|F(\chi_{M})|\leq\|F\|\|\chi_{M}\|_{G}=\|F\|mes(M)(G^{*})^{-1}\bigg{(}\frac{1}{mes(M)}\bigg{)},$$ donc $$\ds\lim_{mes(M)\rightarrow0}\nu(M)=0.$$ Alors $\nu$ est une mesure absolument continue par rapport à la mesure de Lebesgue. D'après le théorème de Radon–Nikod\'{y}m, il existe une fonction $v\in L^{1}(\Omega)$ telle que \begin{equation}\label{83}
                                                                                        \nu(M)=\int_{M}v(x)dx.
                                                                                      \end{equation}
Soit $u$ est une fonction étagée, c.à.d, $u(x)=\sum_{i=1}^{m}\alpha_{i}\chi_{M_{i}}$ avec $$M_{i}\subset\Omega,\ M_{i}\cap M_{j}=\varnothing\ \text{si}\ i\neq j,$$ d'après \eqref{83}, on a \begin{equation}\label{85}
                                        \begin{aligned}
                                          F(u)&= \sum_{i=1}^{m}\alpha_{i}F(\chi_{M_{i}})=\sum_{i=1}^{m}\alpha_{i}\nu(\chi_{M_{i}})=\sum_{i=1}^{m}\alpha_{i}
                                          \int_{M_{i}}
                                          v(x)dx \\
                                          & =\sum_{i=1}^{m}\alpha_{i}\int_{\Omega}v(x)\chi_{M_{i}}(x)dx=\int_{\Omega}u(x)v(x)dx.
                                        \end{aligned}
                                        \end{equation}
   Soit $u\in E^{G}(\Omega)$. Il existe une suite de fonctions étagées $(u_{n})_{n\in\mathbb{N}}$ telle que
   $$u(x)=\ds\lim_{n\rightarrow+\infty}u_{n}(x),\ |u_{n}(x)|\leq|u(x)|,\ p.p\ \text{sur}\ \Omega.$$  Ainsi $$\|u_{n}\|_{G}\leq\|u\|_{G},\ \text{et}\ \ds\lim_{n\rightarrow+\infty}|u_{n}(x)v(x)|=|u(x)v(x)|,\ p.p\ \text{sur}\ \Omega.$$  D'après le théorème de Fatou on a
   \begin{align*}
     \bigg{|}\int_{\Omega}u(x)v(x)dx\bigg{|} & \leq\ds\sup_{n\in\mathbb{N}}\int_{\Omega}|u_{n}(x)v(x)|dx\\
     &=\ds\sup_{n\in\mathbb{N}}F(|u_{n}sign(v)|)\leq\|F\|\ds\sup_{n\in\mathbb{N}}\|u_{n}\|_{G}\leq\|F\|\|u\|_{G}.
   \end{align*}
On déduit $v\in L^{G^{*}}(\Omega)$ et par suite \eqref{82} est vérifiée sur l'ensemble des fonctions étagées.
D'après le théorème \ref{thm21} on sait que l'ensemble des fonctions étagées est dense dans $E^{G}(\Omega)$, par conséquent,
 $$F(u):=\int_{\Omega}u(x)v(x)dx,\ u\in E^{G}(\Omega).$$
\end{proof}

\begin{thm}
    Soit $(G,G^{*})$ un couple de $N$-Fonction. Alors l'espace d'Orlicz $L^{G}(\Omega)$ est réflexif si et seulement si $G$ et $G^{*}$ satisfont la condition $\triangle_{2}$.
\end{thm}

\begin{proof}
    D'après la proposition \ref{pro2} si $G$ satisfait $\triangle_{2}$ on a $$L^{G^{*}}(\Omega)=\bigg{[}L^{G}(\Omega)\bigg{]}^{*}$$
  et si $G^{*}$ satisfait $\triangle_{2}$ alors $$L^{G}(\Omega)=\bigg{[}L^{G^{*}}(\Omega)\bigg{]}^{*}.$$ D'où $$L^{G}(\Omega)=\bigg{[}L^{G}(\Omega)\bigg{]}^{**}.$$
\end{proof}

\section{Critère de compacité}

Dans cette section, nous donnons des conditions nécessaires et suffisantes pour qu'un sous-ensemble $\Re$ de $L^{G}(\Omega)$ soit relativement
compact. Nous introduisons d'abord quelques notions.

\begin{dfn}
  Un sous-ensemble $\Re$ est dit équicontinu en moyenne, si $\forall\varepsilon>0$, $\exists\gamma=\gamma(\varepsilon)$ telle que $\forall h\in\mathbb{R}^{N}$ avec $|h|<\gamma$ on a \begin{equation}\label{44}
                           \|u-u_{h}\|_{G}<\varepsilon,\ \forall u\in\Re,
                         \end{equation}
          où $u_{h}$ est la fonction définie sur $\mathbb{R}^{N}$ par
                                                 \begin{equation}\label{45}
                                                   u_{h}(x):=\begin{cases}
                                                               u(x+h) & \mbox{si } x\in\Omega\ \text{et}\ x+h\in\Omega \\
                                                                0 & \mbox{sinon}.
                                                             \end{cases}
                                                 \end{equation}
\end{dfn}

\textbf{Notation:} Soit $r>0$ et $x\in\mathbb{R}^{N}$. On note $B_{r}(x)$ la boule de centre $x$ et de rayon $r$ et $m_{r}$ sa mesure.

 \begin{dfn}
 Soit $u\in L^{G}(\Omega)$.
 On définit la fonction $S_{r}(u)$ dite fonction de Steklov comme suit
 \begin{equation}\label{46}
   S_{r}(u)(x):=\frac{1}{m_{r}}\int_{B_{r}(x)}u(y)dy=\frac{1}{m_{r}}\int_{|y|<r}u(x+y)dy.
 \end{equation}
 \end{dfn}

\textbf{Notation:} Soit $\Re$ un sous-ensemble borné de $L^{G}(\Omega)$,

pour $r>0$ fixé on note $\Re_{r}$ l'ensemble des fonctions de Steclov $S_{r}(u)$ correspond à $u\in\Re$.
\begin{pro}\label{pro8}
  Soit $u\in L^{G}(\Omega)$. Alors \begin{equation}\label{48}
                                     \|S_{r}(u)\|_{G}\leq\|u\|_{G}.
                                   \end{equation}
\end{pro}

\begin{proof}
 On a
  \begin{align*}
      \|S_{r}(u)\|_{G}&\leq \frac{1}{m_{r}}\ds\sup_{\rho(v;G^{*})\leq1,\ v\in K^{G^{*}}}\int_{\Omega}\int_{B_{r}(x)}|u(y)v(x)|dydx\\
      &=\frac{1}{m_{r}}\ds\sup_{\rho(v;G^{*})\leq1,\ v\in K^{G^{*}}}\int_{\Omega}\int_{|y|<r}|u(x+y)v(x)|dydx.
  \end{align*}
  D'après le théorème de Fubini-Tonelli, on a
  \begin{align*}
    \|S_{r}(u)\|_{G}&\leq\frac{1}{m_{r}}\int_{|y|<r}\bigg{(}\ds\sup_{\rho(v;G^{*})\leq1}\int_{\Omega}|u(x+y)v(x)|dx\bigg{)}dy\\
    &=\ds\sup_{\rho(v;G^{*})\leq1}\int_{\Omega}|u(x+y)v(x)|dx\\&\leq\ds\sup_{\rho(v;G^{*})\leq1}\int_{\Omega}|u(x)v(x-y)|dx\leq\|u\|_{G}.
  \end{align*}
\end{proof}

\begin{thm}[de Vallée-Poussin]\label{thm31}
  Soient $G$ une $N$-Fonction et $\Re$ un sous-ensemble de $L^{G}(\Omega)$ borné en module. Alors pour tout $\epsilon>0$,
  $\exists\delta>0$ tels que $\forall E\subset\Omega$, $mes(E)<\delta$ on a $$\int_{E}|u(x)|dx\leq\epsilon,\ \forall u\in\Re.$$
\end{thm}

\begin{proof}
 Soit $\Re$ un sous-ensemble de $L^{G}(\Omega)$ borné en module donc il existe $C>0$ tel que $\rho(u;G)\leq C$, $\forall u\in\Re$, d'après la remarque
 \ref{rem6} on a $$\|u\|_{G}\leq C+1,\ \forall u\in\Re.$$
 Soit $E\subset\Omega$, d'après la proposition \ref{pro5} on a $$\|\chi_{E}\|_{G^{*}}=mes(E)(G)^{-1}(\frac{1}{mes(E)}).$$
  Comme$\ds\lim_{t\rightarrow0^{+}}tG^{-1}(\frac{1}{t})=0$, alors $\forall\epsilon>0$, $\exists\delta>0$ tels que $\forall E\subset\Omega$, $mes(E)<\delta$, $$\|\chi_{E}\|_{G^{*}}<\epsilon.$$ D'après l'inégalité de H\"{o}lder on a
 $$\int_{E}|u(x)|dx\leq \|u\|_{G}\|\chi_{E}\|_{G^{*}}\leq(C+1)\epsilon,$$ d'où le résultat.
\end{proof}

\begin{pro}\label{pro7}
  Soient $G$ une $N$-Fonction et $r>0$. On suppose que $\Re$ un sous-ensemble borné de $L^{G}(\Omega)$ et $mes(\Omega)<\infty$. Alors $\Re_{r}$ est relativment compact dans $C(\bar{\Omega})$.
\end{pro}

\begin{proof}
Soit $C>0$ tel que $\forall u\in\Re$, $\|u\|_{G}\leq C$, on a
                        \begin{equation}\label{115}
                                  \int_{\Omega}G\bigg{(}\frac{|u(x)|}{C}\bigg{)}dx\leq\int_{\Omega}G\bigg{(}\frac{|u(x)|}{\|u\|_{G}}\bigg{)}dx \leq1
                                 \end{equation}
                                 d'après l'inégalité de Young on a
   \begin{align*}
      |S_{r}(u)(x)|&\leq\frac{1}{m_{r}}\int_{B_{r}(x)}|u(y)|dy=\frac{C}{m_{r}}\int_{B_{r}(x)}\frac{|u(y)|}{C}dy\\
      &\leq\frac{C}{m_{r}}\int_{\Omega}\frac{|u(y)|}{C}dy\leq\frac{C}{m_{r}}\bigg{(}\rho\bigg{(}\frac{u}{C};G\bigg{)}+\rho(1;G^{*})\bigg{)}\\
      &\leq\frac{C}{m_{r}}(1+G^{*}(1)mes(\Omega))<+\infty
   \end{align*}
   donc $\Re_{r}$ est borné dans $C(\bar{\Omega})$ et donc $\Re_{r}$ est borné en module. D'après le théorème de la Vallée-Poussin on a, pour tout $\epsilon>0$, $\exists\delta>0$ telle que pour tout $E\subset\Omega$, $mes(E)<\delta$, $$\int_{E}|u(x)|dx<\epsilon,\ \forall u\in\Re_{r}.$$
   Soit $$B_{x,y}=(B(x,r)\cup B(y,r))\setminus(B(x,r)\cap B(y,r))$$ et choisissons $\eta>0$ tel que $mes(B_{x,y})<\delta$ pour $|x-y|<\eta$, donc
   $$|S_{r}(u)(x)-S_{r}(u)(y)|\leq\frac{1}{m_{r}}\int_{B_{x,y}}|u(z)|dz\leq\frac{\epsilon}{m_{r}}$$  $\Re_{r}$ est équicontinue et d'après le théorème d'Arzel\`{a}-Ascoli $\Re_{r}$ est relativement compact dans $C(\bar{\Omega})$.
\end{proof}

Nous sommes maintenant en mesure de donner un critère de compacité dans $E^{G}(\Omega)$.

\begin{thm}\label{thm10}
  Soit $\Omega$ un ouvert borné de $\mathbb{R}^{N}$. Un sous-ensemble $\Re$ de $E^{G}(\Omega)$ est relativement compact si et seulement si les conditions suivantes sont satisfaites:
  \begin{enumerate}
    \item $\Re$ est borné dans $E^{G}(\Omega)$,
    \item $\forall\varepsilon>0$ il existe $\gamma=\gamma(\varepsilon)>0$ telle que
    \begin{equation}\label{49}
      \|u-S_{r}(u)\|_{G}<\varepsilon,
    \end{equation}
    $\forall0<r<\gamma$ et $u\in\Re$.
  \end{enumerate}
\end{thm}

\begin{proof}
\text{Condition suffisante:}\\ \\
On a $\Re$ est borné dans  $E^{G}(\Omega)$, donc d'après la proposition \ref{pro7}, $\Re_{r}$  est relativment compact dans $C(\bar{\Omega})$, donc dans
  $E^{G}(\Omega)$. La condition $(2)$ montre que $\Re$ est approché par un ensemble relativement compact dans $E^{G}(\Omega)$. D'après le théorème de Fréchet $\Re$ est relativement compact dans $E^{G}(\Omega)$.\\ \\
  \text{Condition nécessaire:}\\ \\
    Soit $\Re$ un sous-ensemble de $E^{G}(\Omega)$ relativement compact, donc $\Re$ est précompact, c.à.d, pour tout $\epsilon>0$ il existe des fonctions continues
    $\omega^{1},\omega^{2},...,\omega^{s}\in\Re$ tels que $\Re\subset\cup_{i=1}^{s}B(\omega^{i},\frac{\epsilon}{3})$.
   Soit $c=\|\chi_{\Omega}\|_{G}$, comme $\omega^{i}$ est continue sur $\overline{\Omega}$ alors $\exists r>0$ tels que $\forall x,y\in \Omega$ ,$|x-y|<r,$ on a
                              \begin{equation}\label{116}
                                |\omega^{i}(x)-\omega^{i}(y)|<\frac{\epsilon}{3c},\ (i=1,...,s)
                              \end{equation}donc
                              \begin{align*}
                                 |S_{r}(\omega)^{i}(x)-\omega^{i}(x)|&\leq\bigg{|}\frac{1}{m_{r}}\int_{B_{r}(x)}\omega^{i}(y)dy-\frac{1}{m_{r}}
                                 \int_{B_{r}(x)}\omega^{i}(x)dy\bigg{|}\\
                                 &=\frac{1}{m_{r}}\bigg{|}\int_{B_{r}(x)}(\omega^{i}(y)-\omega^{i}(x))dy\bigg{|}<\frac{\epsilon}{3c}
                              \end{align*}
                              d'après la remarque \ref{rem2} on a
                              \begin{equation}\label{117}
                                \|S_{r}(\omega)^{i}-\omega^{i}\|_{G}< \|\frac{\epsilon}{3c}\chi_{\Omega}\|_{G}=\frac{\epsilon}{3}.
                              \end{equation}
                              Soit $u\in\Re$, il existe $\sigma$, $1\leq\sigma\leq s$ tel que
                              \begin{equation}\label{118}
                                \|u-\omega^{\sigma}\|_{G}<\frac{\epsilon}{3}.
                              \end{equation}
                              d'après la proposition \ref{pro8} on a
                              \begin{equation}\label{119}
                                 \|S_{r}(u-\omega^{\sigma})\|_{G}=\|S_{r}(u)-S_{r}(\omega^{\sigma})\|_{G}<\frac{\epsilon}{3}.
                              \end{equation}
                              En vertu de \eqref{117},\eqref{118} et \eqref{119} on obtient
       $$\|u-S_{r}u\|_{G}\leq\|u-\omega^{\sigma}\|_{G}+\|\omega^{\sigma}-S_{r}(\omega^{\sigma})\|_{G}+\|S_{r}(\omega^{\sigma})-S_{r}(u)\|_{G}<\epsilon$$
       d'où le résultat.
\end{proof}

\begin{rem}
  On sait que si $G$ vérifie la condition $\triangle_{2}$, $E^{G}(\Omega)$ co\"{\i}ncide avec $L^{G}(\Omega)$
  et qu'un ensemble borné en norme est borné en module, donc on peut reformuler le théorème \ref{thm10} comme suit.
\end{rem}

\begin{thm}
  On suppose que $G$ satisfait la condition $\triangle_{2}$. Alors un sous-ensemble $\Re$ de $E^{G}(\Omega)$ est relativement compact si et seulement si les conditions suivantes sont satisfaites:
  \begin{enumerate}
    \item Il existe $C>0$ telle que $$\int_{\Omega}G(|u(x)|)dx\leq C,\ \forall u\in\Re,$$
    \item $\forall\varepsilon>0$, $\exists\gamma=\gamma(\varepsilon)>0$ telle que $$\int_{\Omega}G(|u(x)-S_{r}(u)(x)|)dx<\varepsilon$$
    $\forall0<r<\gamma$ et $u\in\Re$.
  \end{enumerate}
\end{thm}

Nous terminons cette section par donner un autre critère de compacité en utilisant l'équicontinuité.

\begin{thm}
 Soit $G$ une $N$-fonction et soit $\Re$ un sous-ensemble de\\ $E^{G}(\Omega)$. Alors,\\
 $\Re$ est relativement compact dans $E^{G}(\Omega)$ si et seulement si
 \begin{enumerate}
   \item[(a)] $\Re$ est borné dans $E^{G}(\Omega)$,
   \item[(b)] $\Re$ est équicontinu en moyenne.
 \end{enumerate}
\end{thm}

\begin{proof}
Soit $u\in\Re$, on a \begin{align*}
        |u(x)-S_{r}(u)(x)|&=\bigg{|}\frac{1}{m_{r}}\int_{|y|<r}(u(x)-u(x+y))dy\bigg{|}\\
        &\leq\frac{1}{m_{r}}\int_{|y|<r}|u(x)-u(x+y)|dy.
     \end{align*}
     Soit $v\in K^{G^{*}}(\Omega)$ telle que $\rho(v;G^{*})\leq1$ donc \begin{align*}
                                                                          \int_{\Omega}|u(x)-S_{r}(u)(x)||v(x)|dx&\leq
                                                                          \frac{1}{m_{r}}\int_{|y|<r}\bigg{[}\int_{\Omega}|u(x+y)-u(x)||v(x)|dx\bigg{]}dy\\
                                                                          &\leq \frac{1}{m_{r}}\int_{|y|<r}\|u_{y}-u\|_{G}dy,
                                                                       \end{align*}
alors \begin{equation}\label{120}
                  \|u-S_{r}(u)\|_{G}\leq \frac{1}{m_{r}}\int_{|y|<r}\|u_{y}-u\|_{G}dy.
                \end{equation}
                Puisque $\Re$ est équicontinu, la deuxième condition du théorème \ref{thm10} est vérifiée, alors $\Re$ est relativement compact.\\ \\
 \textbf{Inversement}\\ \\Par le m\^{e}me argument utilisé dans la démonstration du théorème \ref{thm10}, pour tout $\epsilon>0$ il existe des fonctions continues
    $\omega^{1},\omega^{2},...,\omega^{s}\in\Re$ tels que $\Re\subset\cup_{i=1}^{s}B(\omega^{i},\frac{\epsilon}{3})$, $\exists r>0$ tels que
 $$|\omega^{i}(x+h)-\omega^{i}(x)|<\frac{\epsilon}{3c}\ \text{pour}\ h\in\mathbb{R}^{N},\ |h|<\delta,\ i=1,...,s$$ avec $c=\|\chi_{\Omega}\|_{G}$. Alors $\|(\omega^{i})_{h}-\omega^{i}\|_{G}<\frac{\epsilon}{3}$ et $\|u_{h}\|_{G}\leq\|u\|_{G},$ donnent
 $$\|u-u_{h}\|_{G}\leq \|u-\omega^{\sigma}\|_{G}+\|\omega^{\sigma}-(\omega^{\sigma})_{h}\|_{G} +\|(\omega^{\sigma}-u)_{h}\|_{G}<\epsilon.$$
\end{proof}







\chapter{Espace d'Orlicz-Sobolev}

\section{Définition et propriétés élémentaires des espaces d'Orlicz-Sobolev}

Soit $\Omega$ un ouvert de $\mathbb{R}^{N}$ et soient $G$ une $N$-Fonction et $m\geq1$ un entier.

\begin{dfn}
 On appelle espace d'Orlicz-Sobolev noté $W^{m,G}(\Omega)$ défini par
 $$W^{m,G}(\Omega)=\Biggm{\{} u\in L^{G}(\Omega)
  \left |
   \begin{array}{r c l}
       \forall\alpha\in\mathbb{N}^{N}\, |\alpha|\leq m,\exists g_{\alpha}\in  L^{G}(\Omega)\ \text{tel que}\\
        \displaystyle\int_{\Omega}u D^{\alpha}\varphi=(-1)^{|\alpha|}\displaystyle\int_{\Omega}g_{\alpha}\varphi,\emph{}\ \forall\varphi\in C_{c}^{\infty}(\Omega)\\
   \end{array}
   \right \}.$$
 On définit de la m\^{e}me manière l'espace  $W^{m,G}\big{(}E(\Omega)\big{)}$ en remplaçant $L^{G}(\Omega)$ par $E^{G}(\Omega)$.
\end{dfn}

On note $g_{\alpha}=D^{\alpha}u$.\\
On munit $W^{m,G}(\Omega)$ de la norme $$\|u\|_{ W^{m,G}(\Omega)}=\ds\max_{0\leq|\alpha|\leq m}\|D^{\alpha}u\|_{L^{G}(\Omega)}$$
ou de la norme $\sum_{0\leq|\alpha|\leq m}\|D^{\alpha}u\|_{L^{G}(\Omega)}$, ces deux normes sont équivalentes.

\begin{pro}
 \begin{enumerate}
   \item $(W^{m,G}(\Omega),\ \|.\|_{W^{m,G}(\Omega)})$ est un espace de Banach.
   \item $W^{m,G}\big{(}E(\Omega)\big{)}$ co\"{\i}ncide avec  $W^{m,G}(\Omega)$ si et seulement si $G$ satisfait la condition $\triangle_{2}$ ($G$ satisfait $\triangle_{2,0}$ si $mes(\Omega)=+\infty$).
   \item L'espace $W^{m,G}\big{(}E(\Omega)\big{)}$ est séparable.
   \item Si $G$ et $G^{*}$ satisfaient $\triangle_{2}$, alors $W^{m,G}(\Omega)$ est réflexif.
 \end{enumerate}
\end{pro}

\begin{exem}
 Si $1<p<\infty$ et $G_{p}(t)=t^{p}$, alors  $W^{m,G_{p}}(\Omega)= W^{m,G_{p}}\big{(}E(\Omega)\big{)}= W^{m,p}(\Omega)$.
\end{exem}

\begin{dfn}
  $W^{m,G}_{0}(\Omega)$ désigne la fermeture de $C^{\infty}_{c}(\Omega)$ dans  $W^{m,G}(\Omega)$.
\end{dfn}

\begin{thm}\label{thm23}
    $C^{\infty}_{c}(\mathbb{R}^{N})$ dense dans $W^{m,G}\big{(}E(\mathbb{R}^{N})\big{)}$.
\end{thm}

  Pour la démonstration du théorème \ref{thm23} nous avons besoin du lemme suivant.

\begin{lem}\label{lem9}
 Soit $\rho\in C^{\infty}_{c}(\mathbb{R}^{N})$ telle que $\rho\geq0$ et $\displaystyle\int_{\mathbb{R}^{N}}\rho(t)dt=1$. On considère la suite $(\rho_{k})_{k\in\mathbb{N}}$ de
 $C^{\infty}_{c}(\mathbb{R}^{N})$, définie par $\rho_{k}(t)=k\rho(kt)$. Soit $f\in E^{G}$, alors  $$\rho_{k}\ast f\in E^{G}\ \text{et}\ \|\rho_{k}\ast f-f\|_{G}\rightarrow0\ \text{quand}\ k\rightarrow\infty.$$
\end{lem}

\begin{proof}
  Soit $G^{*}$ la conjuguée de $G$ et $v\in L^{G^{*}}(\mathbb{R}^{N})$ tel que $\|v\|_{G^{*}}=1$,
  \begin{align*}
  \int_{\mathbb{R}^{N}}|\rho_{k}\ast f(x)-f(x)||v(x)|dx & \leq\int_{\mathbb{R}^{N}}\bigg{(}\int_{\mathbb{R}^{N}}|f(x-t)-f(x)||v(x)|dx\bigg{)}\rho_{k}(t)dt \\
    &\leq 2\int_{\mathbb{R}^{N}}\|f_{t}-f\|_{G}\|v\|_{G^{*}}\rho_{k}(t)dt\\
    &=2\int_{\mathbb{R}^{N}}\|f_{t}-f\|_{G}\rho_{k}(t)dt,
  \end{align*}
  où $f_{t}(x)=f(x-t)$. Alors
  \begin{align*}
    \|\rho_{k}\ast f-f\|_{G} & \leq  2\int_{\mathbb{R}^{N}}\|f_{t}-f\|_{G}\rho_{k}(t)dt\\
     & = 2\int_{\mathbb{R}^{N}}\|f_{\frac{t}{k}}-f\|_{G}\rho(t)dt.
  \end{align*}
  Comme $f\in E^{G}(\Omega)$ et $\rho$ est à support compact, alors,  pour tout $\varepsilon>0$, il existe $k$ assez grand tel que $$\|f_{\frac{t}{k}}-f\|_{G}\leq\frac{\varepsilon}{2}$$ donc
  $$2\int_{\mathbb{R}^{N}}\|f_{\frac{t}{k}}-f\|_{G}\rho(t)dt\leq\varepsilon\int_{\mathbb{R}^{N}}\rho(t)dt=\varepsilon$$
  et par suite $$\|\rho_{k}\ast f-f\|_{G}\rightarrow0,\ \ k\rightarrow\infty.$$
  D'où le résultat.
\end{proof}

\begin{proof}[\textbf{Démonstration du théorème \ref{thm23}}]
 Soit $u\in W^{m,G}E(\Omega)$.\\ \\ \textbf{\'{E}tape 1}\\ \\ On suppose que $u$ est à support compact. Soient $(\rho_{k})$ la suite définie dans le lemme \ref{lem9} et $u_{k}(x)=u\ast\rho_{k}(x)=\displaystyle\int_{\mathbb{R}^{N}} u(t)\rho_{k}(x-t)dt$.

  Comme $\forall\alpha, |\alpha|\leq m,\ D^{\alpha}u\in E^{G}$ et $\rho_{k}\in C^{\infty}_{c}$, alors
  \begin{align*}
    D^{\alpha}u_{k}(x) & =D^{\alpha}\int_{\mathbb{R}^{N}} u(t)\rho_{k}(x-t)dt \\
     & =\int_{\mathbb{R}^{N}} u(t) D^{\alpha}\rho_{k}(x-t)dt\\&=\int_{\mathbb{R}^{N}} D^{\alpha}u(t)\rho_{k}(x-t)dt.
  \end{align*}
   et pour tout $\alpha$ tel que $|\alpha|\leq m$ $$\|D^{\alpha}u_{k}-D^{\alpha}u\|_{G}=\|D^{\alpha}u\ast\rho_{k}-D^{\alpha}u\|_{G}\rightarrow0,\ k\rightarrow\infty,\ \text{(d'après le lemme \ref{lem9}}).$$
  D'où $$\ds\lim_{k\rightarrow\infty}\|u_{k}-u\|_{W^{m,G}(\Omega)}=0.$$
  \textbf{\'{E}tape 2}\\ \\
  Soit $u\in W^{m,G}\big{(}E(\mathbb{R}^{N})\big{)}$, on considère la fonction $$\varphi_{r}(t)=\begin{cases}
                                                                                               1 & \mbox{si }\ 0\leq t\leq r  \\
                                                                                               0 & \mbox{si}\ t\geq 2r.
                                                                                             \end{cases}$$
         La fonction $\nu_{r}(x)=u(x).\varphi_{r}(|x|)$ est à support compact et  $$\ds\lim_{r\rightarrow\infty}\|\nu_{r}-u\|_{W^{m,G}(\Omega)}=0.$$
         On applique l'étape 1 à la fonction $\nu_{r}$, $$\nu_{k,r}=\nu_{r}\ast\rho_{k}\in C_{c}^{\infty}(\mathbb{R}^{N}),\ \
         \|\nu_{k,r}-u\|_{W^{m,G}(\Omega)}\rightarrow0,\ \ k,r\rightarrow+\infty.$$

\end{proof}

\begin{thm}[Inégalité de Poincaré \cite{7}]\label{thm37}
  Soit $\Omega$ un ouvert borné de $ \mathbb{R}^{N}$. Alors il existe deux constantes $c_{m}$ et $c_{m,\Omega}$ telles que
  $$\int_{\Omega}\sum_{|\alpha|<m}G(D^{\alpha}u)dx\leq c_{m}\int_{\Omega}\sum_{|\alpha|=m}G(c_{m,\Omega}D^{\alpha}u)dx,\ \forall u\in W_{0}^{m,G}(\Omega),$$ et
 $$ \int_{\Omega}G(|u|)dx\leq  \int_{\Omega}G(d|\nabla u|)dx, \forall u\in   W_{0}^{1,G}(\Omega),$$
 où $d=2diam(\Omega)$.
\end{thm}

\section{Injections des espace d'Orlicz-Sobolev dans les espaces $L^{G}(\Omega)$}

Commençons par le cas borné:

\subsection{Cas où $\Omega$ est un ouvert borné de $\mathbb{R}^{N}$}

Soit $G$ une $N$-Fonction. On suppose que \begin{equation}\label{41}
                                            \int_{0}^{1}\frac{G^{-1}(t)}{t^{\frac{N+1}{N}}}dt<\infty,
                                          \end{equation}
et \begin{equation}\label{42}\int_{1}^{\infty}\frac{G^{-1}(t)}{t^{\frac{N+1}{N}}}dt=\infty.\end{equation}
On considère la fonction $G_{*}$, \textbf{conjugué de Sobolev} de $G$, par:
   \begin{equation}\label{43}
      G_{*}^{-1}(t)= \int_{0}^{t}\frac{G^{-1}(\tau)}{\tau^{\frac{N+1}{N}}}d\tau,\ t\geq0.
   \end{equation}

\begin{pro}
  Soit $G$ une $N$-Fonction telle que  $G_{*}$ donnée par \eqref{43}. Alors  $G_{*}$ est une  $N$-Fonction.
\end{pro}

\begin{proof} on a $$\frac{G^{-1}(t)}{t^{\frac{N+1}{N}}}=\bigg{(}\frac{G^{-1}(t)}{t}\bigg{)}\bigg{(}\frac{1}{t^{\frac{1}{N}}}\bigg{)}$$
 puisque $\frac{G^{-1}(t)}{t}$ et $t^{-\frac{1}{N}}$ sont strictement décroissantes donc l'application $t\rightarrow\frac{G^{-1}(t)}{t^{\frac{N+1}{N}}}$ est strictement décroissante.\\
Soit $0\leq v_{1}<v_{2}$. D'après \eqref{43} on a
   \begin{align*}
       G_{*}^{-1}\bigg{(}\frac{v_{1}+v_{2}}{2}\bigg{)}&=\bigg{(}\int_{0}^{v_{1}}+\frac{1}{2}\int_{v_{1}}^{\frac{v_{1}+v_{2}}{2}}
       +\frac{1}{2}\int_{v_{1}}^{\frac{v_{1}+v_{2}}{2}}\bigg{)}\frac{G^{-1}(\tau)}{\tau^{\frac{N+1}{N}}}d\tau\\&>
       \bigg{(}\int_{0}^{v_{1}}+\frac{1}{2}\int_{v_{1}}^{\frac{v_{1}+v_{2}}{2}}
       +\frac{1}{2}\int_{\frac{v_{1}+v_{2}}{2}}^{v_{2}}\bigg{)}\frac{G^{-1}(\tau)}{\tau^{\frac{N+1}{N}}}d\tau\\&=
       \frac{1}{2}\bigg{(}\int_{0}^{v_{1}}+\int_{0}^{v_{2}}\bigg{)}\frac{G^{-1}(\tau)}{\tau^{\frac{N+1}{N}}}d\tau\\&=
       \frac{1}{2}[G_{*}^{-1}(v_{1})+G_{*}^{-1}(v_{2})]
   \end{align*}
   donc $  G_{*}^{-1}$ est strictement concave, ainsi $G_{*}$ est convexe.\\
   On a $$\ds\lim_{v\rightarrow0}\frac{G^{-1}(v)}{v}=\ds\lim_{v\rightarrow0}\frac{v}{G(v)}=+\infty,$$ et
        $$\ds\lim_{v\rightarrow+\infty}\frac{G^{-1}(v)}{v}=\ds\lim_{v\rightarrow0}\frac{v}{G(v)}=0.$$
         D'après la règle de l'H\^{o}pital on a
        \begin{align*}
           \ds\lim_{v\rightarrow0}\frac{G_{*}^{-1}(v)}{v}&=\ds\lim_{v\rightarrow0}\frac{1}{v}\int_{0}^{v}\frac{G^{-1}(t)}{t^{\frac{N+1}{N}}}dt\\&=
                                                          \ds\lim_{v\rightarrow0}\bigg{(}\frac{G^{-1}(v)}{v}\bigg{)}\bigg{(}\frac{1}{v^{\frac{1}{N}}}\bigg{)}
                                                          =+\infty
        \end{align*}
         \begin{align*}
           \ds\lim_{v\rightarrow+\infty}\frac{G_{*}^{-1}(v)}{v}&=\ds\lim_{v\rightarrow+\infty}\frac{1}{v}\int_{0}^{v}\frac{G^{-1}(t)}{t^{\frac{N+1}{N}}}dt\\&=
                                                          \ds\lim_{v\rightarrow+\infty}\bigg{(}\frac{G^{-1}(v)}{v}\bigg{)}\bigg{(}\frac{1}{v^{\frac{1}{N}}}\bigg{)}
                                                          =0
        \end{align*}
        d'où $G_{*}$ est une $N$-Fonction.
\end{proof}

\begin{dfns}[la propriété du c\^{o}ne]

\begin{enumerate}
  \item On appelle c\^{o}ne fini de sommet $x$ tout sous-ensemble ouvert de $\mathbb{R}^{N}$ de la forme $B_{1}\cap\{x+\lambda(y-x),\lambda>0,y\in B_{2}\}$ où $B_{1}$ est une boule ouverte de $\mathbb{R}^{N}$ centrée en $x$ et $B_{2}$ une boule ouverte de $\mathbb{R}^{N}$ ne contenant pas $x$.
  \item Deux c\^{o}nes finis seront dits $congruents$ s'ils se déduisent l'un de l'autre par un déplacement.
   \item On dit que $\Omega$ vérifie la propriété du c\^{o}ne s'il existe un c\^{o}ne fini $C$ de sommet $0$ tel que tout point $x$ de $\Omega$ est le sommet d'un c\^{o}ne fini $C_{x}$ $congruent$ à $C$ et contenu dans $\Omega$.
\end{enumerate}
\end{dfns}

\begin{thm}\label{thm9}
  Soient $G$ une $N$-Fonction et $\Omega$ un ouvert borné ayant la propriété du c\^{o}ne. On suppose que \eqref{41} et \eqref{42} sont satisfaites. Alors
  $$W^{1,G}(\Omega)\hookrightarrow L^{G_{*}}(\Omega),$$
  où $G_{*}$ est donnée par \eqref{43}.
\end{thm}

Commençons par démontrer le lemme suivant.

\begin{lem}\label{lem8}
 Soient $u\in W^{1,1}_{loc}(\Omega)$ et $f$ une fonction satisfait la condition de Lipschitz sur $\mathbb{R}$. Soit $\theta(x)=f(|u(x)|)$, alors $\theta\in W^{1,1}_{loc}(\Omega)$ et $$\frac{\partial\theta}{\partial x_{j}}(x)=f^{'}(|u(x)|)sign(u(x)).\frac{\partial u}{\partial x_{j}}(x).$$
 En particulier, si $u\in W^{1,G}(\Omega)$ alors $f\circ u\in W^{1,G}(\Omega)$.
\end{lem}

\begin{proof}
   Comme $|u|\in W^{1,1}_{loc}(\Omega)$ et $\frac{\partial}{\partial x_{j}}|u(x)|=sgn(u(x)).\frac{\partial u}{x_{j}}(x)$ on peut supposer que $u\geq0$ et $\theta(x)=f(u(x))$.\\
   Soit $\phi\in \mathcal{D}(\Omega)$ et $(e_{j})_{1\leq j\leq N}$ une base canonique de $\mathbb{R}^{N}$. Alors
   \begin{align*}
      -\int_{\Omega}f(u(x))\frac{\partial\phi}{\partial x_{j}}(x)dx &=-\ds\lim_{h\rightarrow0}\int_{\Omega}f(u(x))\frac{\phi(x)-\phi(x-he_{j})}{h}dx\\
      &=\ds\lim_{h\rightarrow0}\int_{\Omega}\frac{f(u(x+he_{j}))-f(u(x))}{h}\phi(x)dx\\
      &=\ds\lim_{h\rightarrow0}\int_{\Omega}Q(x,h)\frac{u(x+he_{j})-u(x)}{h}\phi(x)dx,
   \end{align*}
   où $$Q(x,h)=\frac{f(u(x+he_{j}))-f(u(x))}{u(x+he_{j})-u(x)}\ \ \ \ \text{si}\ \ \ \ u(x+he_{j})\neq u(x).$$
    Puisque $f$ est Lipschitzienne, il existe une constante $K$ telle que $|Q(.,h)|\leq K$, donc $$\|Q(.,h)\|_{\infty}\leq K.$$
   Alors $Q(.,h)$ converge vers $  f^{'}(u(x))$.
   D'autre part, on a $u\in W^{1,1}(supp(\phi))$ donc
   $$\ds\lim_{h\rightarrow0}\frac{u(x+he_{j})-u(x)}{h}\phi(x)= \frac{\partial u}{\partial x_{j}}(x).\phi(x)\in L^{1}(supp(\phi)).$$ D'où
   $$ -\int_{\Omega}f(u(x))\frac{\partial \phi}{\partial x_{j}}(x)dx =\int_{\Omega} f^{'}(u(x))\frac{\partial u}{\partial x_{j}}(x)\phi(x) dx.$$
\end{proof}

\begin{proof}[\textbf{Démonstration du théorème \ref{thm9}}]
La fonction $s=G_{*}(t)$ satisfait l'equation différentielle suivante \begin{equation}\label{92}
                                                                            G^{-1}(s)\frac{ds}{dt}=s^{N+1/N},
                                                                          \end{equation}
 comme $s\leq G^{-1}(s)(G^{*})^{-1}(s)$, on obtient $$ \frac{ds}{dt}\leq s^{\frac{1}{N}} (G^{*})^{-1}(s).$$
    Soit $\sigma(t)=(G_{*}(t))^{N-1/N}$, $\sigma$ satisfait l'inéquation différentielle suivante   \begin{equation}\label{93}
                                                                                 \frac{d\sigma}{dt}\leq \frac{N-1}{N}(G^{*})^{-1}\big{(}(\sigma(t))^{N/N-1}\big{)}.
                                                                               \end{equation}
   \textbf{\'{E}tape 1}\\ \\
     Soit $u\in W^{1,G}(\Omega)$.
 On suppose que $u$ est bornée sur $\Omega$ et non identiquement nulle
 alors l'application $\lambda\mapsto \displaystyle\int_{\Omega}G_{*}\bigg{(}\frac{|u(x)|}{\lambda}\bigg{)}dx$
 est continue et décroissante sur $]0,+\infty[$, de plus,
 $$\ds\lim_{\lambda\rightarrow+\infty}\int_{\Omega}G_{*}\bigg{(}\frac{|u(x)|}{\lambda}\bigg{)}dx =0,\ \ds\lim_{\lambda\rightarrow0}\int_{\Omega}G_{*}\bigg{(}\frac{|u(x)|}{\lambda}\bigg{)}dx =+\infty.$$
  Il existe une constante $K$ telle que
 \begin{equation}\label{94}
   \int_{\Omega}G_{*}\bigg{(}\frac{|u(x)|}{K}\bigg{)}dx=1,\ K=\|u\|_{(G_{*})}.
 \end{equation}
 Soit $f(x)=\displaystyle\sigma\bigg{(}\frac{|u(x)|}{K}\bigg{)}$, comme
 $u\in W^{1,1}(\Omega)$ et $\sigma$ satisfait la condition de lipschitz, alors d'après le lemme \ref{lem8}, $f\in W^{1,1}(\Omega)$. d'autre part, on a
  $$ W^{1,1}(\Omega)\hookrightarrow L^{N/N-1}(\Omega),$$
 et \begin{equation}\label{95}
         \begin{aligned}
          \|f\|_{N/N-1}&\leq K_{1}\bigg{(}\sum_{j=1}^{N}\bigg{\|}\frac{\partial f}{\partial x_{j}}\bigg{\|}_{1}+\|f\|_{1}\bigg{)}\\
          &= K_{1}\bigg{[}\sum_{j=1}^{N}\frac{1}{K}\int_{\Omega}\sigma^{'}\bigg{(}\frac{|u(x)|}{K}\bigg{)}\bigg{|}\frac{\partial u}{\partial x_{j}}\bigg{|}dx+\int_{\Omega}\sigma
          \bigg{(}\frac{|u(x)|}{K}\bigg{)}dx\bigg{]}.
         \end{aligned}
       \end{equation}
   En vertu de \eqref{94} et l'inégalité de H\"{o}lder \eqref{30}, on a
   \begin{equation}\label{96}
     \begin{aligned}
      1&=\bigg{(}\int_{\Omega}G_{*}\bigg{(}\frac{|u(x)|}{K}\bigg{)}dx\bigg{)}^{N-1/N}=\|f\|_{N/N-1}\\
      &\leq \frac{2K_{1}}{K}\sum_{j=1}^{N}\bigg{\|}\sigma^{'}\bigg{(}\frac{|u|}{K}\bigg{)}\bigg{\|}_{(G^{*})}\bigg{\|}\frac{\partial u}{\partial x_{j}}\bigg{\|}_{(G)}+K_{1}\int_{\Omega}
      \sigma\bigg{(}\frac{|u(x)|}{K}\bigg{)}dx.
     \end{aligned}
   \end{equation}
   Nous utilisons \eqref{93} on obtient \begin{align*}
                                           \bigg{\|}\sigma^{'}\bigg{(}\frac{|u|}{K}\bigg{)}\bigg{\|}_{(G^{*})}&\leq\frac{N-1}{N}
                                           \bigg{\|}(G^{*})^{-1}\bigg{(}\big{(}\sigma\big{(}\frac{|u|}{K}\big{)}\big{)}^{N/N-1}\bigg{)}\bigg{\|}_{(G^{*})}\\
                                           &=\frac{N-1}{N}\inf\bigg{\{}\lambda>0:\ \int_{\Omega}G^{*}\bigg{(}\frac{(G^{*})^{-1}(G_{*}(|u(x)|/K))}{\lambda}\bigg{)}dx\leq1\bigg{\}}.
                                        \end{align*}
    Si $\lambda>1$, alors $$\int_{\Omega}G^{*}\bigg{(}\frac{(G^{*})^{-1}(G_{*}(|u(x)|/K))}{\lambda}\bigg{)}dx\leq\frac{1}{\lambda}
    \int_{\Omega}G_{*}\bigg{(}\frac{|u(x)|}{K}\bigg{)}dx=\frac{1}{\lambda}<1.$$
    Ainsi, \begin{equation}\label{97}
             \bigg{\|}\sigma^{'}\bigg{(}\frac{|u|}{K}\bigg{)}\bigg{\|}_{(G^{*})}\leq\frac{N-1}{N}.
           \end{equation}
  Soit $\beta(t)=\displaystyle\frac{G_{*}(t)}{t}$ et $\alpha(t)=\displaystyle\frac{\sigma(t)}{t}$, $\alpha$ est bornée sur tout intervalle compact privé de zéro et
  $\ds\lim_{t\rightarrow+\infty}\frac{\beta(t)}{\alpha(t)}=+\infty$. Il existe alors une constante $t_{0}$ telle que $\forall\ t\geq t_{0}$, $$\alpha(t)\leq\frac{\beta(t)}{2K_{1}}.$$ On pose $K_{2}=K_{1}\ds\sup_{0\leq t\leq t_{0}} \alpha(t)$, on a, $\forall t\geq0$,
  $$\sigma(t)\leq\frac{1}{2K_{1}}G_{*}(t)+\frac{K_{2}}{K_{1}}t.$$ Ainsi
  \begin{equation}\label{98}
    \begin{aligned}
     K_{1}\int_{\Omega}\sigma\bigg{(}\frac{|u|}{K}\bigg{)}dx&\leq\frac{1}{2}\int_{\Omega}G_{*}\bigg{(}\frac{|u|}{K}\bigg{)}dx+\frac{K_{2}}{K}\int_{\Omega}
     |u(x)|dx\\&\leq\frac{1}{2}+\frac{K_{3}}{K}\|u\|_{(G)}.
    \end{aligned}
  \end{equation}  où $K_{3}=2K_{2}\|1\|_{(G^{*})}<\infty$ puisque $mes(\Omega)<+\infty$.\\
   combinant \eqref{96} et \eqref{98} on obtient $$1\leq\frac{2K_{1}}{K}(N-1)\|u\|_{W^{1,G}(\Omega)}+\frac{1}{2}+\frac{K_{3}}{K}\|u\|_{(G)},$$
   Donc \begin{equation}\label{99}
          \|u\|_{(G_{*})}=K\leq K_{4}\|u\|_{W^{1,G}(\Omega)},\  K_{4}\ \text{dépend de}\ N,\ G\ \text{et}\ mes(\Omega).
        \end{equation}
  \textbf{\'{E}tape 2}\\
  Soit $u\in W^{1,G}(\Omega)$. On considère la suite de fonction $$u_{n}(x)=\begin{cases}
                                                                              u(x) & \mbox{si }\ |u(x)|\leq n \\
                                                                              n\ sgn(u(x)) & \mbox{si}\ |u(x)|\geq n.
                                                                            \end{cases} $$
   Il est clair que $(u_{n})_{n\in\mathbb{N}}$ est bornée et appartient à $W^{1,G}(\Omega)$. D'après \eqref{99} on a
   $$ \|u_{n}\|_{(G_{*})}\leq K_{4}\|u_{n}\|_{W^{1,G}(\Omega)}\leq K_{4}\|u\|_{W^{1,G}(\Omega)}$$ la suite $(\|u_{n}\|_{(G_{*})})_{n}$ est croissante et bornée, $K=\ds\lim_{n\rightarrow+\infty}\|u_{n}\|_{(G_{*})}$ existe. En utilise le théorème de Fatou on a $$\int_{\Omega}G_{*}\bigg{(}\frac{|u(x)|}{K}\bigg{)}dx\leq\ds\liminf_{n\rightarrow+\infty}
   \int_{\Omega}G_{*}\bigg{(}\frac{|u_{n}(x)|}{\|u_{n}\|_{(G_{*})}}\bigg{)}dx\leq1.$$  $u\in L^{G_{*}}(\Omega)$ et
   $$\|u\|_{(G_{*})}\leq K=\ds\lim_{n\rightarrow+\infty}\|u_{n}\|_{(G_{*})}\leq K_{4}\|u\|_{W^{1,G}(\Omega)},$$ d'où le résultat.
\end{proof}

Considérons maintenant le:

\subsection{Cas où $\Omega$ n'est pas borné}

\begin{thm}\label{thm24}
  Soit $\Omega$ un ouvert de $\mathbb{R}^{N}$ ayant la propriété du c\^{o}ne. On suppose que $$\int_{1}^{+\infty}\frac{G^{-1}(s)}{s^{1+\frac{1}{N}}}ds=\infty,$$ alors
  $$W^{1,G}(\Omega)\hookrightarrow L^{G_{*}}(\Omega).$$
\end{thm}

La démonstration sera effectuée dans plusieurs lemmes. Le premier lemme établit une estimation de la conjuguée de Sobolev $G_{*}$ définie dans \eqref{43}.

\begin{lem}\label{lem11}
Soit $G$ une $N$-Fonction vérifiant \eqref{41} et \eqref{42}. Soit $1\leq p<N$ tel que la fonction $B$ définie par $B(t)=G(t^{\frac{1}{p}})$ est une
$N$-Fonction. Soient $q=\frac{Np}{N-p}$ et $G_{*}$ la fonction définie dans \eqref{43}. Alors on a les assertions suivantes:
\begin{enumerate}
  \item [(a)] Pour tout $\lambda\geq1$ la fonction $G_{*}^{\frac{\lambda p}{q}}$ est une $N$-Fonction,
  \item [(b)] $[G_{*}(t)]^{\frac{p}{q}}\leq q^{-p}G(t)$ pour tout $t\leq G_{*}^{-1}(1)$,
  \item [(c)] Pour tout $\epsilon>0$ il existe une constante $K_{\epsilon}$ telle que pour tout $t$
  \begin{equation}\label{111}
    [G_{*}(t)]^{\frac{p}{q}}\leq\epsilon G_{*}(t)+K_{\epsilon}G(t).
  \end{equation}
\end{enumerate}
\end{lem}

\begin{proof}
  $\textbf{(a)}$ On pose $Q(t)=[G_{*}(t)]^{\frac{\lambda p}{q}}$. On a $B^{-1}(t)=[G^{-1}(t)]^{p}$, d'après \eqref{43}
  \begin{align*}
     (Q^{-1})^{'}(t)&=\frac{d}{dt}G_{*}^{-1}(t^{\frac{q}{\lambda p}})=\frac{q}{\lambda p}\frac{{G}^{-1}(t^{\frac{q}{\lambda p}})}{t^{1+\frac{q}{\lambda N p}}}\\&=\frac{q}{\lambda p}\bigg{[}\frac{B^{-1}(t^{\frac{q}{\lambda p}})}{t^{\frac{q}{\lambda p}}}\bigg{]}^{\frac{1}{p}}t^{-\mu},
  \end{align*}
  avec $\mu=1+\frac{q}{\lambda N p}-\frac{q}{\lambda p^{2}}\geq0$. On a $$\ds\lim_{t\rightarrow0^{+}}\frac{B^{-1}(t)}{t}=+\infty\ \text{et}\ \ds\lim_{t\rightarrow+\infty}\frac{B^{-1}(t)}{t}=0$$ de plus $B^{-1}$ est concave, donc pour $0<\tau<\iota$, $\displaystyle\frac{B^{-1}(\tau)}{B^{-1}(\iota)}\geq\frac{\tau}{\iota}$. Ainsi, si $0<t<s$,
  $$\frac{(Q^{-1})^{'}(t)}{(Q^{-1})^{'}(s)}\geq\bigg{(}\frac{s}{t}\bigg{)}^{\mu}\geq1.$$ Alors $(Q^{-1})^{'}$ est positive, décroissante et tend vers zéro à l'infini, donc $Q$ est une $N$-Fonction.\\ \\
  $\textbf{(b)}$ Soit $t\leq1$, donc $t\leq q^{p}t^{\frac{p}{q}}$ et on a
  \begin{align*}
    G_{*}^{-1}(t)&=\int_{0}^{t}\frac{[B^{-1}(\tau)]^{\frac{1}{p}}}{\tau^{1+\frac{1}{N}}}d\tau\\
     &\geq\bigg{[}\frac{B^{-1}( q^{p}t^{\frac{p}{q}})}{ q^{p}t^{\frac{p}{q}}}\bigg{]}^{\frac{1}{p}}\int_{0}^{t}\tau^{\frac{1}{p}-\frac{1}{N}-1}d\tau\\
     &=G^{-1}(q^{p}t^{\frac{p}{q}}).
  \end{align*}
  ainsi $t^{\frac{p}{q}}\leq q^{-p}G(G_{*}^{-1}(t))$, $(b)$ se déduit en remplassant $t$ par $G_{*}(t)$.\\ \\
  $\textbf{(c)}$ Soit $\alpha(t)=\displaystyle\frac{G_{*}(t)}{G(t)}$ et $\beta(t)=\displaystyle\frac{[G_{*}(t)]^{\frac{p}{q}}}{G(t)}$. Il est clair que $$\frac{\alpha(t)}{\beta(t)}\rightarrow\infty,\ t\rightarrow\infty.$$ Alors $\forall\ \epsilon>0$, il existe $t_{0}>0$ tel que $\forall\ t\geq t_{0}$, $$\beta(t)\leq\epsilon\alpha(t),$$ c.à.d $$[G_{*}(t)]^{\frac{p}{q}}\leq\epsilon G_{*}(t),\ \forall\ t\geq t_{0}.$$
  et comme $\beta$ est bornée sur $[0,t_{0}]$ disons par $K_{\epsilon}$,
  $$[G_{*}(t)]^{\frac{p}{q}}\leq\epsilon G_{*}(t)+K_{\epsilon}G(t).$$
\end{proof}

\begin{lem}\label{lem14}
  Soient $\Omega$ un ouvert de $\mathbb{R}^{N}$ vérifiant la propriété du c\^{o}ne, $G$ une $N$-Fonction vérifiant \eqref{41}
et \eqref{42} et $G_{*}$ la conjuguée de Sobolev définie dans \eqref{43}. Alors $$W^{1,G}(\Omega)\hookrightarrow L^{G_{*}}(\Omega).$$
\end{lem}

\begin{proof}
D'aprés le lemme \ref{lem11}, $(c)$ avec $p=1$ et $\epsilon=\frac{1}{2K_{1}}$, et en utilisant l'inégalitée \eqref{98} de la démonstration du théorème \ref{thm9} on a
\begin{align*}
     \int_{\Omega}\sigma\bigg{(}\frac{|u|}{K}\bigg{)}dx&\leq\frac{1}{2K_{1}}\int_{\Omega}G_{*}\bigg{(}\frac{|u|}{K}\bigg{)}dx+K_{2}\int_{\Omega}
     |u(x)|dx\\&\leq\frac{1}{2K_{1}}+\frac{K_{2}}{K}\|u\|_{(G)}.
    \end{align*}
    et on termine la preuve avec le m\^{e}me démarche de la démonstration du théorème \ref{thm24}.
\end{proof}

\begin{lem}\label{lem12}
  Soit $G$ une $N$-Fonction telle que \begin{equation}\label{112}
                                        \int_{0}^{1}\frac{G^{-1}(\tau)}{\tau^{1+\frac{1}{N}}}d\tau=\infty,
                                      \end{equation}
                                      \begin{equation}\label{113}
                                        \int_{1}^{+\infty}\frac{G^{-1}(\tau)}{\tau^{1+\frac{1}{N}}}d\tau=\infty
                                      \end{equation}
    Alors il existe une $N$-Fonction $\tilde{G}$ (non nécessairement unique) telle que
    \begin{enumerate}
      \item [(a)] $\tilde{G}(t)=G(t)$ si $t\geq G^{-1}\bigg{(}\displaystyle\frac{G(1)}{2}\bigg{)}$,
      \item [(b)] $\displaystyle\int_{0}^{1}\frac{\tilde{G}^{-1}(\tau)}{\tau^{1+\frac{1}{N}}}d\tau<\infty$,
      \item [(c)] $\displaystyle\int_{1}^{+\infty}\frac{\tilde{G}^{-1}(\tau)}{\tau^{1+\frac{1}{N}}}d\tau=\infty$.
    \end{enumerate}
\end{lem}

\begin{proof}
Soit $1<r<N$, il existe $t_{0}< G^{-1}\bigg{(}\displaystyle\frac{G(1)}{2}\bigg{)}$ tel que
$$G(t_{0})<K t_{0}^{r}\ \ \text{et}\ \ K\bigg{[} G^{-1}\bigg{(}\displaystyle\frac{G(1)}{2}\bigg{)}\bigg{]}^{r}=\frac{G(1)}{2}.$$
Si ceci n'est pas vérifié, c.à.d, $$G(t)\geq K t^{r},\ \forall\ t<G^{-1}\bigg{(}\displaystyle\frac{G(1)}{2}\bigg{)},$$  $$G^{-1}(t)\leq\bigg{(}\frac{t}{K}\bigg{)}^{\frac{1}{r}},\ \forall\ t<\frac{G(1)}{2}$$  ce qui contredit \eqref{112}. Soit $t_{1}=\inf\{t\geq t_{0},\ G(t)=Kt^{r}\}$, on remarque que

$t_{0}<t_{1}\leq  G^{-1}\bigg{(}\displaystyle\frac{G(1)}{2}\bigg{)}.$ Soit $$\tilde{G}(t)=\begin{cases}
                                                                                  K t^{r} & \mbox{si}\ t<t_{1} \\
                                                                                  G(t) & \mbox{si}\ t\geq t_{1}.
                                                                                \end{cases} $$
    $\tilde{G}$ satisfait les conditions $(a)$,$(b)$ et $(c)$.
\end{proof}

\begin{lem}\label{lem13}
  Soient $A$, $B$ deux $N$-Fonctions et $a>0$. Alors il existe $0<b<a$ et $C$ une  $N$-Fonction telles que $$C(t)=A(t),\ \text{si}\ t\geq a,$$ et
  $$C(t)=B(t),\ \text{si}\ t\leq b.$$
\end{lem}

\begin{dfn}
  Soit $G$ une $N$-Fonction satisfait \eqref{112} et \eqref{113} et $\tilde{G_{*}}$ définie par
  $$ \tilde{G_{*}}(t)=\int_{0}^{t}\frac{\tilde{G}^{-1}(\tau)}{\tau^{1+\frac{1}{N}}}d\tau,$$ d'après le lemme \ref{lem12}, $\tilde{G_{*}}$ est une $N$-Fonction. On considère la fonction notée $C$ définie par $$C(t)=\begin{cases}
                                                                               \tilde{G_{*}}(t) & \mbox{si }\ t\geq1 \\
                                                                                G(t) & \mbox{si}\ t\leq t^{*},
                                                                           \end{cases} $$
  pour un certain $0<t^{*}<1$. $C$ est une $N$-Fonction (d'après le lemme \ref{lem13}). Il existe $K^{*}\geq1$ tel que \begin{equation}\label{114}
                                                                             C(t)\leq K^{*}G(t),\ \forall t\leq1.
                                                                            \end{equation}
\end{dfn}

\begin{lem}\label{lem15}
Le lemme \ref{lem14} reste valide même si \eqref{41} n'est pas satisfait. Dans ce cas $C$ remplace $G_{*}$.
\end{lem}

\begin{proof}
  Soit $u\in W^{1,G}(\Omega)$, telle que $\|u\|_{ W^{1,G}}=1$. On pose $$u_{1}(x)=\begin{cases}
                                                                                    1 & \mbox{si }\ u(x)>1 \\
                                                                                    u(x) & \mbox{si }\ |u(x)|\leq 1 \\
                                                                                    -1 & \mbox{si}\ u(x)<-1.
                                                                                  \end{cases}$$
   Soit $u_{2}=u-u_{1}$. D'après le lemme \ref{lem8}, $u_{1}\in W^{1,G}(\Omega)$, $\|u_{1}\|_{ W^{1,G}}\leq1$ et $\|u_{2}\|_{ W^{1,G}}\leq1$.
 Comme \eqref{114} est vérifiée pour tout $t\leq1$, alors $$\int_{\Omega}C\bigg{(}\frac{|u_{1}(x)|}{K^{*}}\bigg{)}dx\leq\int_{\Omega}G(|u_{1}(x)|)dx\leq1,$$ et $\|u_{1}\|_{(C)}\leq K^{*}$. De m\^{e}me $\|u_{2}\|_{(C)}\leq K^{*}$. Soit $$\Omega_{+}=\{x\in\Omega:\ |u(x)|>1\}=\{x\in\Omega:\ |u_{2}(x)|>0\}.$$ Alors
 $$1\geq\int_{\Omega}G(|u(x)|)dx\geq\int_{\Omega_{+}}G(|u(x)|)dx>G(1)mes(\Omega^{+}).$$
 Soit $$\Omega_{+}^{'}=\{x\in\Omega_{+}:\ |u_{2}(x)|/2\leq G^{-1}(G(1)/2)\}\ \text{et}\ \Omega_{+}^{''}=\Omega_{+}\setminus \Omega_{+}^{'},$$ on a
 $$\int_{\Omega_{+}^{'}}\tilde{G}\bigg{(}\frac{|u_{2}(x)|}{2}\bigg{)}dx\leq\int_{\Omega_{+}^{'}}G\bigg{(}G^{-1}\bigg{(}\frac{G(1)}{2}\bigg{)}\bigg{)}\leq
 \frac{G(1)}{2}mes(\Omega_{+})<\frac{1}{2},$$ et aussi $$\int_{\Omega_{+}^{''}}\tilde{G}\bigg{(}\frac{|u_{2}(x)|}{2}\bigg{)}dx\leq\frac{1}{2}\int_{\Omega_{+}^{''}}
 G(|u_{2}(x)|)dx\leq\frac{1}{2}.$$
 Ainsi \begin{align*}
          \int_{\Omega}\tilde{G}\bigg{(}\frac{|u_{2}(x)|}{2}\bigg{)}dx&=\int_{\Omega\setminus\Omega_{+}}\tilde{G}\bigg{(}\frac{|u_{2}(x)|}{2}\bigg{)}dx
         +\int_{\Omega_{+}^{'}}\tilde{G}\bigg{(}\frac{|u_{2}(x)|}{2}\bigg{)}dx\\&+
         \int_{\Omega_{+}^{''}}\tilde{G}\bigg{(}\frac{|u_{2}(x)|}{2}\bigg{)}dx\leq
        \frac{1}{2}+\frac{1}{2}=1,
       \end{align*}
   donc $\|u_{2}\|_{(\tilde{G})}\leq2$. Comme les dérivées de $u_{2}$ s'annulent en dehors de $\Omega_{+}$ en déduit que $\|u_{2}\|_{W^{1,\tilde{G}}(\Omega)}\leq2$. D'après le lemme \ref{lem14} il existe $K_{1}\geq0$ indépendent de $u$ telle que $\|u_{2}\|_{(\tilde{G})}\leq K_{1}$. D'autre part on a
 $$\|u_{2}\|_{(C)}\leq K_{2}=K_{1}\bigg{(}1+\frac{C(1)}{G(1)}\bigg{)},$$ en effet, soit
 $$\Omega_{+}^{1}=\{x\in\Omega_{+}:\ |u_{2}(x)|/K_{1}\geq1\}\ \text{et}\ \Omega_{+}^{2}=\Omega_{+}\setminus\Omega_{+}^{1},$$ rappelons que $C(t)=\tilde{G_{*}}(t)$ si $t\geq1,$ alors \begin{align*}
          \int_{\Omega}C\bigg{(}\frac{|u_{2}(x)|}{K_{2}}\bigg{)}dx&\leq\bigg{(}1+\frac{C(1)}{G(1)}\bigg{)}^{-1}\bigg{(}\int_{\Omega_{+}^{1}}
          +\int_{\Omega_{+}^{2}}\bigg{)}C\bigg{(}\frac{|u_{2}(x)|}{K_{2}}\bigg{)}dx\\
          &\leq \bigg{(}1+\frac{C(1)}{G(1)}\bigg{)}^{-1} \bigg{(}\int_{\Omega}\tilde{G_{*}}\bigg{(}\frac{|u_{2}(x)|}{K_{1}}\bigg{)}dx+
          C(1)mes(\Omega_{+})\bigg{)}\\&\leq1.
       \end{align*}
       Donc $\|u\|_{(C)}\leq K^{*}+K_{2}=K_{3}$. 
       et par suite
       $$\|u\|_{(C)}=\bigg{\|}\frac{u}{\|u\|_{W^{1,G}(\Omega)}}\bigg{\|}_{(C)}\|u\|_{W^{1,G}(\Omega)}\leq K_{3}\|u\|_{W^{1,G}(\Omega)}.$$
\end{proof}

\begin{proof}[\textbf{Démonstration du théorème \ref{thm24}}]
 La démonstration découle directement des lemmes \ref{lem14} et \ref{lem15}.
\end{proof}

\section{Injections des espace d'Orlicz-Sobolev dans les espaces $C_{B}(\Omega)$}

Dans cette section nous étudions les injections des espace d'Orlicz-Sobolev dans le cas où $\displaystyle\int_{1}^{+\infty}\frac{G^{-1}(t)}{t^{1+\frac{1}{N}}}dt<\infty.$

\begin{rem}\label{rem9}
Si $G$ est une $N$-Fonction vérifiant $$M=\int_{1}^{+\infty}\frac{G^{-1}(t)}{t^{1+\frac{1}{N}}}dt<\infty,$$ on peut travailler avec la conjuguée de Sobolev  $G_{*}$ dans un sens étendu (fonction de Young), où la fonction $G_{*}(t)$ est convexe, positive sur $]0,M[$, $G_{*}(0)=0$, $G_{*}(M)=+\infty$ et  $G_{*}(t)=+\infty$, $\forall t\geq M$.
\end{rem}

\begin{pro}
 Sous les hypothèses de la remarque \ref{rem9}, on a
  $$L^{\infty}(\Omega)=L^{G_{*}}(\Omega).$$
\end{pro}

\begin{proof}
on sait que $$L^{\infty}(\Omega)=\{f\ :\ \Omega\rightarrow\mathbb{R};\ \exists\alpha>0,\ |f(t)|\leq\alpha\ p.p\ t\in\Omega\}$$
par définition \begin{align*}
                  L^{G_{*}}(\Omega)&=\bigg{\{}f\ :\ \Omega\rightarrow\mathbb{R};\ \exists\lambda>0,\ \int_{\Omega}G_{*}(\lambda f(t))dt<\infty\bigg{\}}\\
                  &=\bigg{\{}f\ :\ \Omega\rightarrow\mathbb{R};\ \exists\lambda>0,\ \int_{\{t\in\Omega:\ |\lambda f(t)|\leq M\}}G_{*}(\lambda f(t))dt+\int_{\{t\in\Omega:\ |\lambda f(t)|> M\}}G_{*}(\lambda f(t))dt<\infty\bigg{\}}.
               \end{align*}
   Comme $$\int_{\{t\in\Omega:\ |\lambda f(t)|> M\}}G_{*}(\lambda f(t))dt<\infty$$
   alors, on a nécessairement $$mes(\{t\in\Omega:\ |\lambda f(t)|> M\})=0.$$ Ce qui donne,
   \begin{align*}
       L^{G_{*}}(\Omega)&=\{f\ :\ \Omega\rightarrow\mathbb{R};\  \exists\lambda>0,\  |\lambda f(t)|\leq M\ p.p\ t\in\Omega\}\\
       &=\{f\ :\ \Omega\rightarrow\mathbb{R};\  \exists\lambda>0,\  | f(t)|\leq \frac{M}{\lambda}\ p.p\ t\in\Omega\}\\
       &=L^{\infty}(\Omega).
   \end{align*}
\end{proof}

\begin{lem}\label{lem10}
  Sous les hypothèses de la remarque \ref{rem9}, on a \begin{equation}\label{101}
                                                      \frac{1}{M}\|u\|_{\infty}\leq\|u\|_{(G_{*})}\leq \frac{1}{c}\|u\|_{\infty},\
                                                      \forall u\in L^{G_{*}}(\Omega),
                                                     \end{equation}
                                                     où $c=\displaystyle G_{*}^{-1}\bigg{(}\frac{1}{mes(\Omega)}\bigg{)}.$
\end{lem}

\begin{proof}
Comme $G_{*}(t)=+\infty$, $\forall t\geq M$ et $\|u\|_{(G_{*})}<\infty$, alors nécessairement $$\frac{|u(x)|}{\|u\|_{(G_{*})}}\leq M,\ \text{c.à.d},\ |u(x)|\leq M\|u\|_{(G_{*})},\ p.p\ x\in\Omega$$ d'où la première inégalitée.\\ Soit $s=\displaystyle\frac{\|u\|_{\infty}}{c}$ alors $$\frac{|u(x)|}{s}=c\frac{|u(x)|}{\|u\|_{\infty}}\leq c,\ \text{c.à.d},\ G_{*}\bigg{(}\frac{1}{s}|u(x)|\bigg{)}\leq \frac{1}{mes(\Omega)}.$$  En intègre sur $\Omega$, on obtient
$$\int_{\Omega}G_{*}\bigg{(}\frac{1}{s}|u(x)|\bigg{)}dx\leq1,\ \text{d'où},\ \|u\|_{(G_{*})}\leq s.$$
La deuxième inégalité est vérifiée.
\end{proof}

Nous sommes maintenant en mesure de montrer le théorème suivant.

\begin{thm}
  Soit $\Omega$ un ouvert  ayant la propriété du c\^{o}ne. On suppose que \begin{equation}\label{100}
                                                                                      \int_{1}^{+\infty}\frac{G^{-1}(t)}{t^{1+\frac{1}{N}}}dt<\infty.
                                                                                    \end{equation}
 Alors \begin{equation}\label{102}W^{1,G}(\Omega)\hookrightarrow C_{B}(\Omega)=C(\Omega)\cap L^{\infty}(\Omega).\end{equation}
\end{thm}

\begin{proof}
  On remarque que la fonction $s=G_{*}(t)$ satisfait l'equation différentielle \eqref{92}. Nous pouvons répéter le démarche de la démonstration du théorème \ref{thm9} et conclure l'existence d'une constante $c_{1}>0$ telle que pour tout $u\in W^{1,G}(\Omega)$ on a $$\|u\|_{(G_{*})}\leq c_{1}
  \|u\|_{W^{1,G}(\Omega)},$$ d'après le lemme \ref{lem10},  \begin{equation}\label{103}\|u\|_{\infty}\leq M\|u\|_{(G_{*})}\leq c_{1}\int_{1}^{+\infty}
  \frac{G^{-1}(t)}{t^{1+\frac{1}{N}}}dt.\|u\|_{W^{1,G}(\Omega)}.\end{equation} Ainsi $W^{1,G}(\Omega)\hookrightarrow L^{\infty}(\Omega)$.\\ \\
   Il reste à montrer que $u$ est continue sur $\Omega$. Soient $y\in\Omega$ et $\delta>0$ tels que $B_{2\delta}(\Omega)\subset\Omega$. On pose $$\nu_{h}(x)=u(x+h)-u(x)$$ pour $x\in B_{\delta}(y)$ et $h\in\mathbb{R}^{N}$ telle que $|h|<\delta$. Remplaçant $\nu_{h}$ dans \eqref{103} on obtient $$|\nu_{h}(y)|\leq c_{2}\|\nu_{h}\|_{W^{1,G}(\Omega)}.$$  $\ds\lim_{|h|\rightarrow0}|\nu_{h}(y)|=0$. Alors $u\in C(\Omega)$.
\end{proof}

\section{Injection compacte}

\begin{dfn}[Convergence en mesure]
  Soit $(u_{n})$ une suite de fonctions, on dit que $(u_{n})$ converge en mesure sur $\Omega$ vers $u$ si pour tout $\epsilon,\delta>0$  il existe un entier $n_{0}$ tel que pour tout $j\geq n_{0}$, on a $$mes(\{x\in\Omega:\ |u_{j}(x)-u(x)|\geq\epsilon\})\leq\delta.$$
\end{dfn}

\begin{lem}\label{lem17}
  Soient $\Omega$ un ouvert de mesure fini et $B$ une $N$-Fonction telle que $B\prec\prec G$.  Soit $(u_{n})$ une suite bornée dans  $L^{G}(\Omega)$ et converge en mesure sur $\Omega$ alors $(u_{n})$ converge en norme dans $L^{B}(\Omega)$.
\end{lem}

\begin{proof}
Soient $\epsilon>0$ et $v_{j,m}(x)=\displaystyle\frac{u_{j}(x)-u_{m}(x)}{\epsilon}$. La suite $(v_{j,m})$ est bornée dans $L^{G}(\Omega)$ donc il existe $K>0$ tel que $\|v_{j,m}\|_{G}\leq K$. On a $B\prec\prec G$ donc il existe $t_{0}>0$ tel que $\forall t\geq t_{0}$, $$B(t)\leq\frac{1}{4}G\bigg{(}\frac{t}{K}\bigg{)}.$$
Soient $\delta=\frac{1}{4B(t_{0})}$ et $$\Omega_{j,m}=\bigg{\{}x\in\Omega:\ |v_{j,m}(x)|\geq B^{-1}\bigg{(}\frac{1}{2mes(\Omega)}\bigg{)}\bigg{\}}.$$
Comme $(u_{j})$ converge en mesure, alors il existe $n\in\mathbb{N}$ tel que $\forall j,m\geq n$, $mes(\Omega_{j,m})\leq\delta$. Soit
$$\Omega_{j,m}^{'}=\{x\in\Omega_{j,m}:\ |v_{j,m}(x)|\geq t_{0}\},\ \ \Omega_{j,m}^{''}=\Omega_{j,m}\setminus\Omega_{j,m}^{'}.$$
Pour $j,m\geq n$ on a \begin{align*}
                         \int_{\Omega}B(|v_{j,m}(x)|)dx&= \int_{\Omega\setminus\Omega_{j,m}}B(|v_{j,m}(x)|)dx+ \int_{\Omega_{j,m}^{'}}B(|v_{j,m}(x)|)dx+ \int_{\Omega_{j,m}^{''}}B(|v_{j,m}(x)|)dx\\
                         &\leq\frac{mes(\Omega)}{2mes(\Omega)}+\frac{1}{4}\int_{\Omega_{j,m}^{'}}G\bigg{(}\frac{|v_{j,m}(x)|}{K}\bigg{)}dx+\delta B(t_{0})\leq1.
                      \end{align*}
                      Alors $\|u_{j}-u_{m}\|_{B}\leq\epsilon$, d'où $(u_{j})$ converge dans $L^{B}(\Omega)$.
\end{proof}

\begin{lem}\label{lem16}
  Soient $\Omega$ un ouvert de mesure fini et $B$ une $N$-Fonction telle que $B\prec\prec G$. Soit $S$ un sous ensemble de $L^{G}(\Omega)$, tel que
  \begin{enumerate}
    \item $S$ est borné dans $L^{G}(\Omega)$,
    \item $S$ est précompact dans $L^{1}(\Omega)$.
  \end{enumerate}
  Alors $S$ est précompact dans $L^{B}(\Omega)$.
\end{lem}

\begin{proof}
  Puisque $mes(\Omega)<\infty$ donc  $L^{G}(\Omega)\hookrightarrow  L^{1}(\Omega)$. Soit $(u_{n})$ est une suite de $S$, alors elle admet une sous-suite notée aussi $(u_{n})$ qui converge vers $u$ dans $L^{1}(\Omega)$. Soit $\epsilon,\delta>0$, il existe un entier $n_{0}$ tel que pour tout $n\geq n_{0}$, $$\|u_{n}-u\|_{1}\leq\epsilon\delta.$$ Par suite, $$mes(\{x\in\Omega:\ |u_{n}(x)-u(x)|\geq\epsilon\})\leq\delta.$$
  Ainsi $(u_{n})$ converge en mesure vers $u$ sur $\Omega$ et donc $(u_{n})$ converge vers $u$ dans $L^{B}(\Omega)$(d'après le lemme \ref{lem17}).
\end{proof}

\begin{thm}[de Rellich-Kondrachov]\label{thm30}
Soit $\Omega$ un ouvert borné vérifiant la propritée du c\^{o}ne. Sous les m\^{e}mes hypothèses que le théorème \ref{thm9} et si $B$ est une $N$-Fonction telle que $B\prec\prec G$ alors l'injection $$W^{1,G}(\Omega)\hookrightarrow L^{B}(\Omega)$$ est compact.
\end{thm}

\begin{proof}
  Comme $mes(\Omega)<\infty$; $$W^{1,G}(\Omega)\hookrightarrow W^{1,1}(\Omega)\hookrightarrow L^{1}(\Omega),$$ l'injection de $W^{1,1}(\Omega)$ dans $L^{1}(\Omega)$ est compact.
  D'après le théorème \ref{thm9}, un sous-ensemble $A$ borné de $W^{1,G}(\Omega)$ est borné dans $L^{G_{*}}(\Omega)$ et précompact dans $L^{1}(\Omega)$. D'après le lemme \ref{lem16} $A$ est précompact dans $L^{B}(\Omega)$.
\end{proof}

\chapter{Multiplicité de solution pour une classe d'équations non-linéaire et non-homogène dans l'espace d'Orlicz-Sobolev}

\section{Hypothèses et R\'esultat principal}

Soient $\Omega$ un ouvert borné de $\mathbb{R}^{N}$ à bord régulier et $a:]0,+\infty[\rightarrow\mathbb{R}$ une fonction telle que l'application $\varphi:\mathbb{R}\rightarrow\mathbb{R}$ définie par $$\varphi(t)=\begin{cases}
                                                               a(|t|)t & \mbox{si }\ t\neq0 \\
                                                               0 & \mbox{si}\ t=0
                                                             \end{cases}$$
    est un homéomorphisme impaire et croissant de $\mathbb{R}$ dans $\mathbb{R}$. D'après la proposition \ref{pro11}, les fonctions $\Phi$ et $\Phi^{*}$ définis par
    $$\Phi(t)=\int_{0}^{t}\varphi(s)ds,\ \ \Phi^{*}(t)=\int_{0}^{t}\varphi^{-1}(s)ds,\ \ \forall\ t\in\mathbb{R},$$ sont des $N$-Fonctions.
     On note par $$\|u\|=\||\nabla u|\|_{(\Phi)},\ \forall u\in W_{0}^{1,\Phi}(\Omega),$$
     d'après l'inégalité de Poincaré, on a, $\||\nabla u|\|_{(\Phi)}$ est une norme équivalente à la norme $\|u\|_{ W^{m,G}(\Omega)}$ de $W^{1,\Phi}(\Omega)$.\\
Dans ce chapitre on considère le système suivant $$(P)\quad \begin{cases}
                                                            -div(a(|\nabla u|)\nabla u)=\lambda(u^{p-1}-u^{q-1}) & \mbox{si}\ x\in\Omega  \\
                                                            u=0 & \mbox{si}\ x\in\partial\Omega \\
                                                            u\geq0 & \mbox{si}\ x\in\Omega.
                                                          \end{cases}$$
       On note par $$\varphi_{0}:=\ds\inf_{t>0}\frac{t\varphi(t)}{\Phi(t)}\ \ \text{et}\ \  \varphi^{0}:=\ds\sup_{t>0}\frac{t\varphi(t)}{\Phi(t)}.$$

    Dans le reste de ce chapitre, on suppose que
     \begin{equation}\label{122}
       1<\varphi_{0}\leq\frac{t\varphi(t)}{\Phi(t)}\leq \varphi^{0}<\infty,\ \forall t\geq0.
     \end{equation}

      \begin{equation}\label{126}
       1<q<p<\varphi_{0}
     \end{equation}
     et la fonction
     \begin{equation}\label{125}
       t\rightarrow \Phi(\sqrt{t}),\  [0,+\infty[
     \end{equation}
     est convexe.

\begin{rem}
 D'après le théorème \ref{thm7} et l'hypothèse \eqref{122}, $\Phi$ et $\Phi^{*}$ satisfaient la condition $\triangle_{2}$.
\end{rem}

\begin{pro}
   \begin{equation}\label{123}
       \|u\|^{\varphi^{0}}\leq \int_{\Omega}\Phi(|\nabla u|)dx\leq \|u\|^{\varphi_{0}},\ \forall u\in W_{0}^{1,\Phi}(\Omega),\ \|u\|<1,
     \end{equation}
     \begin{equation}\label{124}
       \|u\|^{\varphi_{0}}\leq \int_{\Omega}\Phi(|\nabla u|)dx\leq \|u\|^{\varphi^{0}},\ \forall u\in W_{0}^{1,\Phi}(\Omega),\ \|u\|>1.
     \end{equation}
\end{pro}

\begin{proof}
  La preuve des deux inégalités
  $$\int_{\Omega}\Phi(|\nabla u|)dx\leq \|u\|^{\varphi_{0}},\ \forall u\in W_{0}^{1,\Phi}(\Omega),\ \|u\|<1;$$ et
  $$\|u\|^{\varphi_{0}}\leq \int_{\Omega}\Phi(|\nabla u|)dx,\ \forall u\in W_{0}^{1,\Phi}(\Omega),\ \|u\|>1,$$
   peut être réalisée comme dans [\cite{28}, lemme C.7].\\
   Ensuite, supposons $\|u\|<1$. Soit $\xi\in ]0,\|u\|[$. Par définition de $\varphi^{0}$, on a
   $$\Phi(t)\geq \tau^{\varphi^{0}}\Phi(\frac{t}{\tau}),\ \forall t>0,\ \tau\in]0,1[.$$
   En utilisant la relation ci-dessus, on a
   \begin{equation}\label{137}
     \int_{\Omega}\Phi(|\nabla u(x)|)dx\geq \xi^{\varphi^{0}}\int_{\Omega}\Phi(\frac{|\nabla u(x)|}{\xi})dx.
   \end{equation}
   Soit $v(x)=\displaystyle\frac{u(x)}{\xi}$, $\forall x\in\Omega$, on a $\|v\|=\frac{\|u\|}{\xi}>1.$ En utilisant la première inégalité de cette proposition, on obtient \begin{equation}\label{138}
             \int_{\Omega}\Phi(|\nabla v(x)|)dx\geq \|v\|^{\varphi_{0}}>1.
           \end{equation}
    Les relations \eqref{137} et \eqref{138} montrent que  $$\int_{\Omega}\Phi(|\nabla u(x)|)dx\geq\xi^{\varphi^{0}}.$$
    Si $\xi\nearrow\|u\|$ dans l'inégalité si-dessus, on obtient
    $$\int_{\Omega}\Phi(|\nabla u(x)|)dx\geq \|u\|^{\varphi^{0}},\ \forall u\in W_{0}^{1,\Phi}(\Omega),\ \|u\|<1.$$
    Enfin, nous montrons la dernière inégalité dans la proposition. Un calcul simple montre que
    \begin{equation}\label{139}
      \frac{\Phi(\sigma t)}{\Phi(t)}\leq \sigma^{\varphi^{0}},\ \forall t>0,\ \text{et}\ \sigma>1.
    \end{equation}
    Alors, $\forall u\in W_{0}^{1,\Phi}(\Omega)$ avec $\|u\|>1$, l'inégalité \eqref{139} donne
    \begin{align*}
       \int_{\Omega}\Phi(|\nabla u(x)|)dx&=\int_{\Omega}\Phi(\|u\|\frac{|\nabla u(x)|}{\|u\|})dx\\
                                         &\leq\|u\|^{\varphi^{0}}\int_{\Omega}\Phi(\frac{|\nabla u(x)|}{\|u\|})dx\\
                                         &\leq \|u\|^{\varphi^{0}}.
    \end{align*}
\end{proof}

\begin{dfn}
   Une fonction $u\in W_{0}^{1,\Phi}(\Omega)$ est dite solution faible de $(P)$ si $u\geq0$, $p.p$ sur $\Omega$ et vérifie
   $$\int_{\Omega}a(|\nabla u|)\nabla u.\nabla v dx-\lambda\int_{\Omega}u^{p-1}vdx+\lambda\int_{\Omega}u^{q-1}vdx=0,\ \forall v\in W_{0}^{1,\Phi}(\Omega).$$
\end{dfn}

 Le résultat principal est donné par le théorème suivant:

\begin{thm}\label{thm32}
    Supposons que \begin{equation}\label{127}
                    \varphi^{0}<min\bigg{\{}N,\frac{N\varphi_{0}}{N-\varphi_{0}}\bigg{\}}.
                  \end{equation}
   Alors il existe $\lambda^{*}>0$ telle que pour tout $\lambda>\lambda^{*}$ le problème $(P)$ admet au moins deux solutions faibles positives non trivials et distincts.
\end{thm}

\begin{rem}
   Nous soulignons que cet résultat a été signalé par [\cite{17}, théorème 1.2], dans le cas de l'opérateur $p$-Laplacian.
\end{rem}

\section{Démonstration du résultat principal}

La d\'emonstration du Th\'eor\`eme \ref{thm32} s'effectue en deux \'etapes. En premier lieu, nous d\'emontrons l'existence d'une solution d'\'energie n\'{e}gative en utilisant la méthode de minimum global. Ensuite, moyennant le Th\'{e}or\`{e}me du Mountain pass sans la condition de Palais-Smale, nous \'etablissons l'existence d'une seconde solution d'\'energie positive.\\
On note par $E=W_{0}^{1,\Phi}(\Omega)$ et
on définit la fonctionnelle  d'\'energie $I$ associ\'ee \`a $(P)$ par:
$$I(u)=\int_{\Omega}\Phi(|\nabla u|)dx-\frac{\lambda}{p}\int_{\Omega}u_{+}^{p}dx+\frac{\lambda}{q}\int_{\Omega}u_{+}^{q}dx,\ u\in E,$$
où $u_{\pm}(x)=\max(\pm u(x),0).$

\begin{lem}\label{lem19}
  Soit la fonctionnelle $I_{0}:E\rightarrow\mathbb{R}$ définie par $$I_{0}(u)=\int_{\Omega}\Phi(|\nabla u|)dx.$$ Alors $I_{0}$ est de classe $C^{1}$.
\end{lem}

\begin{proof}
Soient $u,h\in E$ et $0<t<1$ tel que $u+th\in\Omega$. On a
   $$\frac{1}{t}(I_{0}(u+th)-I_{0}(u))=\int_{\Omega}\frac{1}{t}\int_{|\nabla u|}^{|\nabla u+t\nabla h|}a(|s|)sdsdx.$$
D'autre part, $$|\nabla u+t\nabla h|\rightarrow|\nabla u|,\ \text{dans}\ L^{\Phi}(\Omega),\ t\rightarrow0,$$
et donc dans $L^{1}(\Omega)$. On a $$\bigg{|}\frac{1}{t}\int_{|\nabla u|}^{|\nabla u+t\nabla h|}a(|s|)sds\bigg{|}\leq \varphi(|\nabla u|+|\nabla h|)|\nabla h|,$$ comme $ \varphi(|\nabla u|+|\nabla h|)\in L^{\Phi^{*}}(\Omega)$ et $|\nabla h|\in L^{\Phi}(\Omega)$, d'après l'inégalité de H\"{o}lder, $$\int_{\Omega} |\varphi(|\nabla u|+t|\nabla h|)||\nabla h|dx\leq\| \varphi(|\nabla u|+t|\nabla h|)\|_{\Phi^{*}}
\||\nabla h|\|_{\Phi}<\infty,$$ donc $$ \varphi(|\nabla u|+|\nabla h|)|\nabla h|\in L^{1}(\Omega).$$
Soit $$\Psi(t)=|\nabla u+t\nabla h|\ \text{et}\ H_{x}(t)=\int_{|\nabla u|}^{\Psi(t)}a(|s|)sds$$
D'après le théorème de convergence dominé on a
\begin{align*}
   \ds\lim_{t\rightarrow0}\frac{1}{t}(I_{0}(u+th)-I_{0}(u))&=\ds\lim_{t\rightarrow0}\int_{\Omega}\frac{1}{t}H_{x}(t)dx\\
                                                           &=\int_{\Omega}\ds\lim_{t\rightarrow0}\frac{H_{x}(t)-H_{x}(0)}{t}dx
                                                           =\int_{\Omega}H_{x}^{'}(0)dx\\
                                                           &=\int_{\Omega}a(|\nabla u|)\nabla u.\nabla h dx\\&=I_{0}^{'}(u,h),
\end{align*}
donc $I_{0}$ est g\^{a}teaux différentiable.
De plus, $I_{0}^{'}$ est continue, en effet, soit $(u_{n})$ une suite de $E$ qui converge vers $u$. On a
$$|\langle I_{0}^{'}(u_{n})-I_{0}^{'}(u),h\rangle|=\int_{\Omega}[a(|\nabla u_{n}|)\nabla u_{n}-a(|\nabla u|)\nabla u].\nabla h dx,$$
d'après l'inégalité de H\"{o}lder, on a
 $$ \ds\sup_{\|h\|_{\Phi}\leq1}\bigg{|}\int_{\Omega}[a(|\nabla u_{n}|)\nabla u_{n}-a(|\nabla u|)\nabla u].\nabla hdx\bigg{|}\\
  \leq\|\varphi(|\nabla u_{n}|)-\varphi(|\nabla u|)\|_{\Phi^{*}}\ds\sup_{\|h\|_{\Phi}\leq 1}\|\nabla h\|_{\Phi}.$$
D'autre part, \begin{align*}
       \bigg{|}\rho(\varphi(|\nabla u_{n}|);\Phi^{*})-\rho(\varphi(|\nabla u|);\Phi^{*}) \bigg{|}&=\bigg{|}\int_{\Omega}\int_{\varphi(|\nabla u|)}^{\varphi(|\nabla u_{n}|)}\varphi^{-1}(s)dsdx \bigg{|}\\
               &\leq\int_{\Omega}|\nabla u_{n}|[\varphi(|\nabla u_{n}|)-\varphi(|\nabla u|)]dx\rightarrow0,\ n\rightarrow+\infty
     \end{align*}
     Puisque $\Phi^{*}$ satisfait la condition $\triangle_{2}$, alors
     $$\|\varphi(|\nabla u_{n}|)-\varphi(|\nabla u|)\|_{\Phi^{*}}\rightarrow0,\ n\rightarrow+\infty,$$

Ainsi $$\|I_{0}^{'}(u_{n})-I_{0}^{'}(u)\|_{E^{*}}\rightarrow0.$$ D'où le résultat.
\end{proof}

\begin{pro}\label{lem23}
  La fonctionnelle $I$ est de classe $C^{1}$ et $$\langle I^{'}(u),v\rangle=\int_{\Omega}a(|\nabla u|)\nabla u.\nabla v dx-\lambda\int_{\Omega}u_{+}^{p-1}vdx+\lambda\int_{\Omega}u_{+}^{q-1}vdx,\ \forall u,v\in E.$$
\end{pro}

\begin{proof}
On a $$I(u)=I_{0}(u)-\frac{\lambda}{p}\int_{\Omega}u_{+}^{p}dx+\frac{\lambda}{q}\int_{\Omega}u_{+}^{q}dx,$$ d'après le lemme \ref{lem19}, $I_{0}$ est de classe $C^{1}$, d'autre part l'application $$u\rightarrow-\frac{\lambda}{p}\int_{\Omega}u_{+}^{p}dx+\frac{\lambda}{q}\int_{\Omega}u_{+}^{q}dx$$ est de classe $C^{1}$, d'où $I$ de classe $C^{1}$ et $$\langle I^{'}(u),v\rangle=\int_{\Omega}a(|\nabla u|)\nabla u.\nabla v dx-\lambda\int_{\Omega}u_{+}^{p-1}vdx+\lambda\int_{\Omega}u_{+}^{q-1}vdx,\ \forall u,v\in E.$$
\end{proof}

\begin{lem}\label{rem10}
Les points critiques de $I$ sont positives.
\end{lem}
\begin{proof}
 Soit $u\in E$ tel que $I^{'}(u)=0$ sur $E$,
  \begin{align*}
    0&=\langle I^{'}(u),u_{-}\rangle=\int_{\Omega}a(|\nabla u|)\nabla u.\nabla u_{-}dx-\lambda\int_{\Omega}(u_{+})^{p-1}u_{-}dx+\lambda
    \int_{\Omega}(u_{+})^{q-1}u_{-}dx\\&=\int_{\Omega}a(|\nabla u|)\nabla u.\nabla u_{-}dx=\int_{\Omega}a(|\nabla u_{-}|)|\nabla u_{-}|^{2}dx\geq \varphi_{0}\int_{\Omega}\Phi(|\nabla u_{-}|)dx\\&\geq \varphi_{0}\|u_{-}\|^{\phi_{0}} (\text{D'après} \eqref{123})
  \end{align*}
 on déduit que $u\geq0$.
\end{proof}

\subsection{Existence d'une premi\`{e}re solution du problème $ (P) $}

On applique la méthode du minimum global pour montrer l'existence d'une solution faible du problème $ (P) $ d'\'energie n\'egative.

\begin{lem}\label{lem26}
  $E\hookrightarrow W_{0}^{1,\varphi_{0}}(\Omega)\hookrightarrow L^{\varphi_{0}}(\Omega)$.
\end{lem}

\begin{proof}
Soit $B(t)=t^{\varphi_{0}}$. D'après \eqref{122}, on a $$\bigg{(}\displaystyle\frac{\Phi(t)}{B(t)}\bigg{)}^{'}>0,$$ donc la fonction
 $\displaystyle\frac{\Phi(t)}{B(t)}$ est croissante pour tout $t\in]0,+\infty[$.
  Alors $$B(t)\leq \frac{1}{\Phi(1)}\Phi(t),\ \forall t>1.$$
  D'où $B\prec \Phi$ et $E\hookrightarrow W_{0}^{1,\varphi_{0}}(\Omega)$.

\end{proof}

\begin{lem}\label{lem18}
  Il existe $\lambda_{1}>0$ telle que $$\lambda_{1}=\ds\inf_{u\in E,\ \|u\|>1}\frac{\displaystyle\int_{\Omega}\Phi(|\nabla u|)dx}{\displaystyle\int_{\Omega}|u|^{\varphi_{0}}dx}.$$
\end{lem}

\begin{proof}
D'après le lemme \ref{lem26}, $E$ s'injecte continument dans $L^{\varphi_{0}}$,
il existe alors $C>0$ telle que $$\|u\|\geq C\|u\|_{L^{\varphi_{0}}(\Omega)},\ \forall u\in E.$$ D'autre part, on a $$\int_{\Omega}\Phi(|\nabla u|)dx\geq \|u\|^{\varphi_{0}},\ \forall u\in E\ \text{avec}\ \|u\|>1.$$ Combinant les deux inégalitées ci-dessus on obtient
$$\int_{\Omega}\Phi(|\nabla u|)dx\geq C^{\varphi_{0}}\int_{\Omega}|u|^{\varphi_{0}}dx,\ \forall u\in E\ \text{avec}\ \|u\|>1.$$ D'où le résultat.
\end{proof}

\begin{lem}\label{lem25}
  $I_{0}$ est faiblement semi-continue inférieurement sur $E$.
\end{lem}

\begin{proof}
Soit $(u_{n})$ une suite qui converge faiblement vers $u$ dans $E$, donc $$\int_{\Omega}u_{n}\omega dx\rightarrow\int_{\Omega}u\omega dx,\ \int_{\Omega}\frac{\partial u_{n}}{\partial x_{i}}\omega dx\rightarrow\int_{\Omega}\frac{\partial u}{\partial x_{i}}\omega dx,\ \forall\ \omega\in E^{\Phi^{*}}(\Omega),$$ en particulier $\forall\ \omega\in L^{\infty}(\Omega)$. Alors
$$ \frac{\partial u_{n}}{\partial x_{i}}\rightarrow\frac{\partial u}{\partial x_{i}}\ \text{dans}\ L^{1}(\Omega)\ \text{faiblement}.$$
D'autre part, $(u_{n})$ est bornée dans $E$ et l'injection $E\hookrightarrow L^{\Phi_{*}}(\Omega)$ est compacte. Alors à une sous-suite près $(u_{n})\rightarrow v$ dans $L^{\Phi_{*}}(\Omega)$. En particulier $(u_{n})\rightarrow v$ dans $L^{1}(\Omega)$ et donc dans $\mathfrak{D}^{'}(\Omega)$. Alors $u=v$ dans
$\mathfrak{D}^{'}(\Omega)$ ceci donne $u\in L^{\Phi_{*}}(\Omega)$ et $u_{n}\rightarrow u$ dans $L^{\Phi_{*}}(\Omega)$, $u_{n}\rightarrow u$, $p.p$ sur $\Omega$ et $$I_{0}(u)=\int_{\Omega}\Phi(|\nabla u|)dx\leq\liminf\int_{\Omega}\Phi(|\nabla u_{n}|)dx=\liminf I_{0}(u_{n}).$$
D'où  $I_{0}$ est semi-continue inférieurement sur $E$.
\end{proof}

\begin{pro}\label{pro9}
  \begin{enumerate}
    \item [(i)] La fonctionnelle $I$ est bornée inférieurement et coercive.
    \item [(ii)] La fonctionnelle $I$ est faiblement semi-continue inférieurement.
  \end{enumerate}
\end{pro}

\begin{proof}
  $\mathbf{(i)}$ Comme $1<q<p<\varphi_{0}$ alors $$\ds\lim_{t\rightarrow+\infty}\frac{\frac{1}{p}t^{p}-\frac{1}{q}t^{q}}{t^{\varphi_{0}}}=0,$$ Donc, il existe $t_{0}>0$, tel que $\forall\ t\geq t_{0}$, $$\lambda\bigg{(}\frac{1}{p}t^{p}-\frac{1}{q}t^{q}\bigg{)}\leq\frac{\lambda_{1}}{2}t^{\varphi_{0}}.$$
  D'autre part, il existe $0\leq\mu<t_{0}$ tel que  $\forall\ t\leq \mu$,
                \begin{align*}
                   \lambda\bigg{(}\frac{1}{p}t^{p}-\frac{1}{q}t^{q}\bigg{)}-\frac{\lambda_{1}}{2}t^{\varphi_{0}}&=
                   t^{q}\bigg{[}-\frac{\lambda}{p}+\frac{\lambda}{p}t^{p-q}-\frac{\lambda}{2}t^{\varphi_{0}-q}\bigg{]}\\
                   &=t^{q}\bigg{[}-\frac{\lambda}{p}+t^{p-q}\bigg{(}\frac{\lambda}{p}-\frac{\lambda}{2}t^{\varphi_{0}-p}\bigg{)}\bigg{]}\leq0.
                \end{align*}
  Soit  $C_{\lambda}=\ds\sup_{t\in[\mu,t_{0}]}  \lambda\bigg{(}\frac{1}{p}t^{p}-\frac{1}{q}t^{q}\bigg{)} $, alors $\forall\ \lambda>0$,
 $$\lambda\bigg{(}\frac{1}{p}t^{p}-\frac{1}{q}t^{q}\bigg{)}\leq\frac{\lambda_{1}}{2}t^{\varphi_{0}}+C_{\lambda},\ \forall t\geq0,$$ $\lambda_{1}$ est définie dans le lemme \ref{lem18}.
  D'après l'inégalité ci-desssus et la condition \eqref{124}, $\forall u\in E$, $\|u\|>1$
  \begin{align*} I(u)&\geq\int_{\Omega}\Phi(|\nabla u|)dx-\frac{\lambda_{1}}{2}\int_{\Omega}|u|^{\varphi_{0}}dx-C_{\lambda}mes(\Omega)\\&\geq\frac{1}{2}\int_{\Omega}\Phi(|\nabla u|)dx-C_{\lambda}mes(\Omega)\\&\geq\frac{1}{2}\|u\|^{\varphi_{0}}-C_{\lambda}mes(\Omega).\end{align*} D'où  $I$ est bornée inférieurement et coercive.\\
   $\mathbf{(ii)}$ Soit $(u_{n})\subset E$ une suite qui converge faiblement vers $u$ dans $E$. D'après le lemme \ref{lem25}, $I_{0}$ est semi-continue inférieurement, donc
   $$I_{0}(u)\leq\ds\liminf_{n\rightarrow\infty}I_{0}(u_{n}).$$
Comme l'injection de $E$ dans $L^{p}(\Omega)$ et dans $L^{q}(\Omega)$ est compacte alors $(u_{n^{+}})$ converge fortement vers $u_{+}$ dans $L^{p}(\Omega)$ et $L^{q}(\Omega)$. Ainsi $$I(u)\leq \ds\liminf_{n\rightarrow\infty}I(u_{n}).$$
\end{proof}

\begin{rem}
  D'après la proposition \ref{pro9} et le théorème \ref{thm34}, il existe $u_{1}\in E$ un minimum global de $I$, donc $u_{1}$ est une solution du problème $(P)$.
\end{rem}

\begin{pro}\label{pro10}
  Il existe $\lambda^{*}>0$ telle que $\ds\inf_{E}I<0$, $\forall \lambda>\lambda^{*}$.
\end{pro}

\begin{proof}
  Soient $\Omega_{1}\subset\Omega$ un sous-ensemble compact, $u_{0}\in E$ telle que $u_{0}(x)=t_{0}$ sur $\Omega_{1}$ et $0\leq u_{0}(x)\leq t_{0}$ sur $\Omega\backslash\Omega_{1}$ où $t_{0}>1$ telle que $$\frac{1}{p}t_{0}^{p}-\frac{1}{q}t_{0}^{q}>0.$$
  On a \begin{align*}\frac{1}{p}\int_{\Omega}u_{0}^{p}dx-\frac{1}{q}\int_{\Omega}u_{0}^{q}dx&\geq \frac{1}{p}\int_{\Omega_{1}}u_{0}^{p}dx-\frac{1}{q}\int_{\Omega_{1}}u_{0}^{q}dx-\frac{1}{q}\int_{\Omega\backslash\Omega_{1}}u_{0}^{q}dx\\&
  \geq \frac{1}{p}\int_{\Omega_{1}}u_{0}^{p}dx-\frac{1}{q}\int_{\Omega_{1}}u_{0}^{q}dx-\frac{1}{q}t_{0}^{q}mes(\Omega\backslash\Omega_{1}),\end{align*}
  on choisit $\Omega_{1}$ tel que $\frac{1}{p}\int_{\Omega_{1}}u_{0}^{p}dx-\frac{1}{q}\int_{\Omega_{1}}u_{0}^{q}dx-\frac{1}{q}t_{0}^{q}mes(\Omega\backslash\Omega_{1})>0$ donc pour $\lambda$ assez grand on obtient $I(u_{0})<0$, d'où le résultat.
\end{proof}

\begin{rem}
  D'après la proposition \ref{pro10}, $u_{1}$ est une solution non triviale du problème $(P)$  d'\'energie n\'egative et d'après le lemme \ref{rem10}, $u_{1}>0$.
\end{rem}

\subsection{Existence d'une seconde solution du problème $ (P) $}

Soit $\lambda\geq \lambda^{*}$. On pose $$f(x,t)=\begin{cases}
                                             0 & \mbox{si }\ t<0, \\
                                             t^{p-1}-t^{q-1} & \mbox{if }\ 0\leq t\leq u_{1}(x), \\
                                             u_{1}(x)^{p-1}-u_{1}(x)^{q-1} & \mbox{si}\ t>u_{1}(x).
                                           \end{cases}$$
                                           et $$F(x,t)=\int_{0}^{t}f(x,s)ds.$$
On définit la fonctionnelle $J:E\rightarrow\mathbb{R}$ par $$J(u)=\int_{\Omega}\Phi(|\nabla u|)dx-\lambda\int_{\Omega}F(x,u)dx.$$
Les mêmes arguments utilisés pour la fonctionnelle $I$ s'appliquent à la fonctionnelle $J$ et $J\in C^{1}(E,\mathbb{R})$,
$$\langle J^{'}(u),v\rangle=\int_{\Omega}a(|\nabla u|)\nabla u.\nabla v dx-\lambda\int_{\Omega}f(x,u)vdx,\ \ \forall\ u,v\in E.$$

\begin{rem}
Soit $u$ un point critique de $J$, alors
 \begin{align*}
    0=\langle J^{'}(u),u_{-}\rangle&=\int_{\Omega}a(|\nabla u|)\nabla u.\nabla u_{-} dx-\lambda\int_{\Omega}f(x,u)u_{-}dx\\
                                   &=\int_{\Omega}a(|\nabla u|)\nabla u.\nabla u_{-} =\int_{\Omega}a(|\nabla u_{-}|)|\nabla u_{-}|^{2}dx\\
                                    &\geq \varphi_{0}\int_{\Omega}\Phi(|\nabla u_{-}|)dx\geq\varphi_{0}\|u_{-}\|^{\Phi_{0}} (\text{D'après} \eqref{123})
 \end{align*}
  donc $u_{-}\equiv0$ et par suite $u\geq0$.
\end{rem}

\begin{lem}\label{lem20}
  Si $u$ est un point critique de $J$ alors $u\leq u_{1}$.
\end{lem}

\begin{proof}
  On a
       \begin{equation}\label{135}
        \begin{aligned}
          0&=\langle J^{'}(u)-I^{'}(u_{1}),(u-u_{1})_{+}\rangle=\int_{\Omega}\big{(}a(|\nabla u|)\nabla u-a(|\nabla u_{1}|)\nabla u_{1}\big{)}.\nabla(u-u_{1})_{+}dx\\&-\lambda\int_{\Omega}[f(x,u)-(u_{1}^{p-1}-u_{1}^{q-1})](u-u_{1})_{+}dx\\
          &=\int_{[u>u_{1}]}\big{(}a(|\nabla u|)\nabla u-a(|\nabla u_{1}|)\nabla u_{1}\big{)}.\nabla(u-u_{1})dx.
        \end{aligned}
       \end{equation}
       Comme $\varphi$ est croissante sur $\mathbb{R}$ alors pour tout $\xi$ et $\psi\in\mathbb{R}^{N}$
       $$(\varphi(|\xi|)-\varphi(|\psi|))(|\xi|-|\psi|)\geq0,$$ avec égalité si et seulement si $\xi=\psi$. Donc,
       $$(a(|\xi|)|\xi|-a(|\psi|)|\psi|)(|\xi|-|\psi|)\geq0,\ \forall \xi,\psi\in\mathbb{R}^{N},$$ on a  égalité si et seulement si $\xi=\psi$. D'autre part,
       $$(a(|\xi|)\xi-a(|\psi|)\psi).(\xi-\psi)\geq(a(|\xi|)|\xi|-a(|\psi|)|\psi|)(|\xi|-|\psi|)\geq0,\ \forall \xi,\psi\in\mathbb{R}^{N}.$$
       En effet, \begin{align*}
                              &(a(|\xi|)\xi-a(|\psi|)\psi).(\xi-\psi)-(a(|\xi|)|\xi|-a(|\psi|)|\psi|)(|\xi|-|\psi|)=\\
                              &a(|\xi|)\xi.\xi-a(|\xi|)\xi.\psi-a(|\psi|)\psi.\xi+a(|\psi|)\psi.\psi-a(|\xi|)|\xi|^{2}+a(|\xi|)|\xi||\psi|+\\
                              &a(|\psi|)|\psi||\xi|-a(|\psi|)|\psi|^{2}=
                              a(|\xi|)[|\xi||\psi|-\xi.\psi]+a(|\psi|))[|\psi||\xi|-\xi.\psi]\geq0.
                           \end{align*}

        En déduit que $$(\varphi(|\xi|)\xi-\varphi(|\psi|)\psi).(\xi-\psi)\geq0\ \forall \xi,\psi\in\mathbb{R}^{N},$$
        avec égalité si et seulement si $\xi=\psi$. Donc de \eqref{135} on obtient $\nabla u=\nabla u_{1}$ sur
         $\omega:=\{y\in\Omega;\ u(y)>u_{1}(y)\}$ et $$\int_{\omega}\Phi(|\nabla(u-u_{1})|)dx=0$$
         $$\int_{\Omega}\Phi(|\nabla(u-u_{1})_{+}|)dx=0.$$ Et vertu de \eqref{123} on obtient $$\|(u-u_{1})_{+}\|=0.$$ Alors $(u-u_{1})_{+}=0$ sur $\Omega$, et par suite $u\leq u_{1}$ $p.p\ x\in\Omega$.
\end{proof}

Le lemme suivant  montre que $J$ admet la g\'{e}om\'{e}trie de Th\'{e}or\`{e}me du col.

\begin{lem}\label{lem21}
  Il existe $\rho\in]0,\|u_{1}\|[$ et $\delta>0$ tels que $J(u)\geq \delta$, $\forall u\in E$ avec $\|u\|=\rho$.
\end{lem}

\begin{proof}
  Soit $u\in E$ telle que $\|u\|<1$. Il est clair que $$\frac{1}{p}t^{p}-\frac{1}{q}t^{q}\leq 0,\ \forall t\in[0,1].$$
  Soit $$\Omega_{u}:=\{x\in\Omega;\ u(x)>min(1,u_{1}(x))\}.$$
  Si $x\in\Omega\backslash\Omega_{u}$ alors $$F(x,u)=\frac{1}{p}u_{+}^{p}-\frac{1}{q}u_{+}^{q}\leq0.$$
  Si $x\in \Omega_{u}\cap\{x\in\Omega;\ u_{1}(x)<u(x)<1\}$ alors
  \begin{align*}
      F(x,u)&=\int_{0}^{u_{1}(x)}f(x,s)ds+\int_{u_{1}(x)}^{u}f(x,s)ds\\
            &=\int_{0}^{u_{1}(x)}s^{p-1}-s^{q-1}ds+\int_{u_{1}(x)}^{u}u_{1}(x)^{p-1}-u_{1}(x)^{q-1}ds\\
            &=\frac{1}{p}u_{1}^{p}-\frac{1}{q}u_{1}^{q}+\big{(}u_{1}^{p-1}-u_{1}^{q-1}\big{)}(u-u_{1})\leq0,
  \end{align*}
  Soit $$\Omega_{u,1}:=\Omega_{u}\setminus\{x\in\Omega,\ u_{1}(x)<u(x)<1\}.$$ D'après \eqref{123} on a
  \begin{equation}\label{128}
    J(u)\geq\int_{\Omega}\Phi(|\nabla u|)dx-\lambda\int_{\Omega_{u,1}}F(x,u)dx\geq\|u\|^{\varphi^{0}}-\lambda\int_{\Omega_{u,1}}F(x,u)dx.
  \end{equation}
  On a $\varphi^{0}<\varphi_{0}^{*}=\frac{N\varphi_{0}}{N-\varphi_{0}}$, donc $E\hookrightarrow W_{0}^{1,\varphi_{0}}(\Omega)$ et $\forall\  s\in]\varphi^{0},\frac{N\varphi_{0}}{N-\varphi_{0}}[$\\ $E\hookrightarrow L^{s}(\Omega)$ (injection continue).
  Il existe alors une constante $C>0$ telle que $$\|u\|_{L^{s}(\Omega)}\leq C\|u\|,\ \forall u\in E,\ \forall\  s\in]\varphi^{0},\frac{N\varphi_{0}}{N-\varphi_{0}}[.$$
  On obtient
  \begin{equation}\label{129}
    \begin{aligned}
     \lambda\int_{\Omega_{u,1}}F(x,u)dx&=\lambda\int_{\Omega_{u,1}\cap[u<u_{1}]}\bigg{(}\frac{1}{p}u_{+}^{p}-\frac{1}{q}u_{+}^{q}\bigg{)}dx+
     \lambda\int_{\Omega_{u,1}\cap[u>u_{1}]}\bigg{(}\frac{1}{p}u_{1}^{p}-\frac{1}{q}u_{1}^{q}\bigg{)}dx\\&+
     \lambda\int_{\Omega_{u,1}\cap[u>u_{1}]}\big{(}u_{1}^{p-1}-u_{1}^{q-1}\big{)}(u-u_{1})dx\\
     &\leq\frac{\lambda}{p}\int_{\Omega_{u,1}\cap[u<u_{1}]}u_{+}^{p}dx
     +\frac{\lambda}{p}\int_{\Omega_{u,1}\cap[u>u_{1}]}u_{1}^{p}dx+\lambda\int_{\Omega_{u,1}\cap[u>u_{1}]}u_{1}^{p-1}udx\\
     &+\lambda\int_{\Omega_{u,1}\cap[u>u_{1}]}u_{1}^{q}dx\\
     &\leq\lambda D\int_{\Omega_{u,1}}u_{+}^{p}dx\leq \lambda D\int_{\Omega_{u,1}}u_{+}^{s}dx\leq\lambda D_{1}\|u\|^{s},
    \end{aligned}
  \end{equation}
  où $D>0$, $D_{1}>0$. Combinant \eqref{128} et \eqref{129}, on obtient
  $$J(u)\geq(1-\lambda D_{1}\|u\|^{s-\varphi^{0}})\|u\|^{\varphi^{0}},$$
  soit $\rho\in]0,min(1,\|u_{1}\|)[$ tel que pour $\|u\|=\rho$, $1-\lambda D_{1}\|u\|^{s-\varphi^{0}}>0$ et prenons $\delta=(1-\lambda D_{1}\rho^{s-\varphi^{0}})\rho^{\varphi^{0}}>0$. D'où le résultat.
\end{proof}

\begin{lem}
  Il existe $e\in X$ tel que $\|e\|>\rho$, et $J(e)\leq0$.
\end{lem}

\begin{proof}
  On a $J(u_{1})=I(u_{1})<0$ et $\|u_{1}\|>\rho$. Alors on peut prendre $e=u_{1}$.
\end{proof}

\begin{lem}\label{lem27}
  La fonctionnelle $J$ est coercive.
\end{lem}

\begin{proof}
  Soit $u\in E$ telle que $\|u\|>1$. D'après \eqref{124} et l'inégalité de H\"{o}lder, on a
  \begin{align*}
     J(u)&\geq \int_{\Omega}\Phi(|\nabla u|)dx-\lambda\int_{[u>u_{1}]}F(x,u)dx-\lambda\int_{[u<u_{1}]}F(x,u)dx\\
     &\geq\|u\|^{\varphi_{0}}-\frac{\lambda}{p}\int_{[u>u_{1}]}u_{1}^{p}dx-\frac{\lambda}{p}\int_{[u>u_{1}]}u_{1}^{p-1}udx-\frac{\lambda}{p}\int_{[u<u_{1}]}u_{+}^{p}dx\\
     &\geq\|u\|^{\varphi_{0}}-\frac{\lambda}{p}\int_{\Omega}u_{1}^{p}dx-\frac{2\lambda}{p}\int_{\Omega}u_{+}^{p}dx\\
     &\geq \|u\|^{\varphi_{0}}-\frac{2\lambda}{p}C_{1}^{p}\|u\|^{p}-C_{2}\geq  \|u\|^{\varphi_{0}}-C_{3}\|u\|^{p}-C_{2},
  \end{align*}
  où $C_{1},C_{2}$ et $C_{3}$ sont des constantes positives. Comme $p<\varphi_{0}$, de l'inégalité ci-dessus en déduit $J(u)\rightarrow\infty$ si $\|u\|\rightarrow\infty$, d'où, $J$ est coercive.
\end{proof}

\begin{lem}\label{lem22}
  Soient $(u_{n})\in E$, $u_{n}\rightharpoonup u$ et
  \begin{equation}\label{136}
  \ds\limsup_{n\rightarrow\infty}\int_{\Omega}a(|\nabla u_{n}|)\nabla u_{n}.(\nabla u_{n}-\nabla u)dx\leq 0.
   \end{equation}
   Alors $(u_{n})$ converge fortement vers $u$ dans $E$.
\end{lem}

\begin{proof}
   On a $(\|u_{n}\|)$ est bornée, d'après \eqref{123} et \eqref{124}, $(I_{0}(u_{n}))$ est bornée. \`{A} une sous-suite près, $I_{0}(u_{n})\rightarrow c$, donc $$I_{0}(u)\leq\ds\liminf_{n\rightarrow\infty}I_{0}(u_{n})=c.$$ Comme $I_{0}$ est convexe on a $$I_{0}(u)\geq I_{0}(u_{n})+\langle I_{0}^{'}(u_{n}),u-u_{n}\rangle.$$ donc d'après \eqref{136}$I_{0}(u)=c$.\\
   $\frac{u_{n}+u}{2}$ converge faiblement vers $u$ dans $E$, alors \begin{equation}\label{130}
                                                                         c=I_{0}(u)\leq\ds\liminf_{n\rightarrow\infty}I_{0}\bigg{(}\frac{u_{n}+u}{2}\bigg{)}.
                                                                       \end{equation}
        Supposons que $(u_{n})$ ne converge pas vers $u$ dans $E$, d'après \eqref{123}, il existe $\epsilon>0$ et une sous-suite de $(u_{n})$ notée aussi $(u_{n})$ tel que  \begin{equation}\label{131}
                            I_{0}\bigg{(}\frac{u_{n}-u}{2}\bigg{)}\geq\epsilon,\ \forall n\in\mathbb{N}.
                          \end{equation}
                          D'autre part, d'après [\cite{23}, lemme 2.1],
                          \begin{equation}\label{132}
                            \frac{1}{2}I_{0}(u)+\frac{1}{2}I_{0}(u_{n})-I_{0}\bigg{(}\frac{u_{n}+u}{2}\bigg{)}\geq I_{0}\bigg{(}\frac{u_{n}-u}{2}\bigg{)}\geq\epsilon,\ \forall n\in\mathbb{N}.
                          \end{equation}
                         Alors $$c-\epsilon\geq\ds\limsup_{m\rightarrow\infty}I_{0}\bigg{(}\frac{u_{n_{m}}+u}{2}\bigg{)},$$
                          ce qui contredit \eqref{130}, donc $(u_{n})$ converge fortement vers $u$ dans $E$.
\end{proof}

\begin{proof}[\textbf{Démonstration du théorème \ref{thm32}}]
   D'après le lemme \ref{lem21} et le théorème du Mountain pass \ref{thm33}, il existe une suite $(v_{n})\subset E$ telle que
\begin{equation}\label{134}
   J(v_{n})\rightarrow c>0\ \ \text{et}\ \ J^{'}(v_{n})\rightarrow 0
\end{equation}
où  $$c=\ds\inf_{\gamma\in\Gamma}\ds\sup_{t\in[0,1]}J(\gamma(t))$$ et $$\Gamma=\{\gamma\in C([0,1],E);\ \gamma(0)=0,\ \gamma(1)=u_{1}\}.$$
En vertu de \eqref{134} et le lemme \ref{lem27}, la suite $(v_{n})$ est bornée, donc il existe une sous suite notée aussi $(v_{n})$ qui converge faiblement vers $u_{2}\in E$. Comme l'injection de $E$ dans $L^{i}(\Omega),\ \forall\ i\in[1,\varphi_{0}]$ est compact, alors $(v_{n})$ converge fortement vers $u_{2}$ dans $L^{i}(\Omega),\ \forall\ i\in[1,\varphi_{0}]$. Ainsi
$$\langle I_{0}^{'}(v_{n})- I_{0}^{'}(u_{2}),v_{n}-u_{2}\rangle=\langle J^{'}(v_{n})-J^{'}(u_{2}),v_{n}-u_{2}\rangle+\lambda\int_{\Omega}
[f(x,v_{n})-f(x,u_{2})](v_{n}-u_{2})dx=o(1),$$ quand $n\rightarrow+\infty$. D'après le lemme \ref{lem22} en déduit que $(v_{n})$ converge fortement vers $u_{2}$ dans $E$ et en utilisant la relation \eqref{134} on obtient
$$ J(u_{2})= c>0\ \ \text{et}\ \ J^{'}(u_{2})= 0.$$
D'après le lemme \ref{lem20}, $0\leq u_{2}\leq u_{1}$ sur $\Omega$. Donc
$$f(x,u_{2})=u_{2}^{p-1}-u_{2}^{q-1}\ \ \text{et}\ \ F(x,u_{2})=\frac{1}{p} u_{2}^{p}-\frac{1}{q} u_{2}^{q}$$
d'où $$J(u_{2})=I(u_{2})\ \text{et}\ J^{'}(u_{2})=I^{'}(u_{2}).$$
On conclut que $u_{2}$ est un point critique de $I$ et une solution du problème $(P)$. De plus, $I(u_{2})=c>0$ et $I(u_{2})>0>I(u_{1})$. Ainsi $u_{2}$ est une solution non trivial et $u_{2}\neq u_{1}$.
\end{proof}

\chapter{Appendice}

\section{Outils d'int\'{e}gration}
\begin{dfn}
  Soit $(X,\Sigma,\mu)$ un espace mesuré $\sigma$-fini. Une fonction $s\ :\ X\rightarrow [0,+\infty[$ est dite fonction simple, s'il existe $m\in\mathbb{N}$, $\{a_{1},...,a_{m}\}$ des réels et des sous-ensembles de $X$ mesurables disjoints $\{E_{1},...,E_{m}\}$ tels que
  $$s(x)=\begin{cases}
      a_{j}, & \mbox{si }\ x\in E_{j},\ j=1,...,m  \\
      0, & \mbox{si}\ x\in X\setminus\cup_{j=1}^{m}E_{j}.
    \end{cases}  $$
\end{dfn}
\begin{4}[\textbf{Propriétés de la mesure de Lebesgue}]
\begin{enumerate}
  \item La mesure de Lebesgue a la propri\'et\'e des sous-ensembles de mesures finies, c.à.d, pour tout ensemble $E$ mesurable tel que $mes(E)>0$, il existe $F$ mesurable, $F\subset E$ et $0<mes(F)<\infty$.
\item Un ensemble mesurable $A$ est dit atome pour la mesure de Lebesgue si $mes(A)>0$ et pour tout ensemble $B\subset A$, $mes(B)=0$ ou $mes(A\setminus B)=0$.
\item Un ensemble mesurable $D$ est dit diffuse pour la mesure de Lebesgue si $D$ ne contient aucune atome, ceci donne, pour $0\leq \lambda\leq mes(D)$, il existe un ensemble mesurable $D_{1}\subset D$ tel que $mes(D_{1})=\lambda$.
\end{enumerate}

\end{4}
\begin{thm}[Inégalité de Jensen]
Si $0<mes(D)<\infty$ et si $\Phi : \mathbb{R}\rightarrow\mathbb{R}$ est une fonction convexe, alors pour toute fonction $f\in L^{1}(D)$, on a
$$\Phi\bigg{(}\frac{1}{mes(D)}\int_{D}f(x)dx\bigg{)}\leq\frac{1}{mes(D)}\int_{D}\Phi\circ f(x)dx.$$
\end{thm}

\begin{thm}[Théorème de Radon-Nikodym \cite{6}]
Soit $\nu$ une mesure absolument continue par rapport à la mesure de Lebesgue. Alors il existe $f\in L^{1}(\Omega)$ tel que
$$\nu(E)=\int_{E}f(x)dx$$ pour tout sous ensemble mesurable $E$ de $\Omega$.
\end{thm}

\begin{thm}[Théorème de Fréchet \cite{5}] Soient $E$ un espace vectoriel et $A\subset E$. On suppose qu'il existe une famille de sous ensemble $(A_{n})_{n\in\mathbb{N}}$ compact dans $E$ et dense dans $A$. Alors $A$ est compact dans $E$.

\end{thm}
\section{Méthodes Variationnelles}

\subsection{Méthode du minimum global}

Si $K$ est un espace topologique, rappelons qu’une fonction $J$ de $K$ dans $\mathbb{R}$ est
dite semi-continue inf\'{e}rieurement (en abr\'{e}g\'{e} s.c.i.) si pour tout $\lambda\in\mathbb{R}$ l’ensemble $[J\leq\lambda]:=\{x\in K;\ J(x)\leq\lambda\}$ est fermé. Introduisons la définition suivante:

\begin{dfn}
  Soient $X$ un espace de Banach et $\omega$ une partie de $X$. Une
fonction $J : \omega\rightarrow \mathbb{R}$ est dite sequentiellement faiblement s.c.i. si pour toute
suite $(x_{n})$ de $\omega$ convergeant faiblement vers $x\in\omega$ on a $J(x)\leq\ds\liminf_{n\rightarrow\infty}J(x_{n})$.
\end{dfn}

On peut montrer que, dans les espaces de Banach r\'{e}flexifs, les fonctions faiblement
s.c.i. atteignent leur minimum, pourvu qu’elles sont coersives, c.à.d, tendent
vers $+\infty$ à l’infini.

\begin{thm}\label{thm34}[\cite{16}]
  Soient $X$ un espace de Banach réflexif, $K \subset X$ un convexe
fermé et $J : K \rightarrow\mathbb{R}$ une fonction faiblement s.c.i. De plus,
si $K$ est non borné, on suppose que pour toute suite $(x_{n})$ de $K$ telle que
$\|x_{n}\|\rightarrow+\infty,$ on a $J(x_{n})\rightarrow+\infty$. Alors $J$ atteint son minimum.
\end{thm}

\subsection{Théorème de Mountain pass}

\begin{dfn}
Soit $ X $ un espace de Banach et $ J: X \rightarrow \mathbb{R} $ une fonction diff\'erentiable.
\begin{enumerate}
   \item Une suite $ (u_k)_k \subset X $ telle que $ (J(u_k))_k $ est born\'ee et $ J'(u_k) \to 0,\ k \to \infty $ est dite suite de Palais-Smale pour $ J. $ S'il existe $ c \in \mathbb{R} $ tel que $ J(u_k) \to c $ et $ J'(u_k) \to 0, $ alors $ (u_k)_k $ est dite suite de Palais-Smale pour $ J $ au niveau $ c. $
   \item On dit que $ J $ satisfait la condition de Palais-Smale (on \'ecrit $ J $ satisfait (PS)) si toute suite de Palais-Smale pour $ J $ poss\`ede une sous-suite convergente. On dit que $ J $ satisfait la condition de Palais-Smale au niveau $ c \in \mathbb{R} $ (on \'ecrit $ J $ satisfait (PS)$_c$) si toute suite de Palais-Smale pour $ J $ au niveau $ c $ poss\`ede une sous-suite convergente.
\end{enumerate}
\end{dfn}

\begin{thm}\label{thm33}[\cite{32}, \cite{15}]
Soient $ X $ un espace de Banach, $ J: X \rightarrow \mathbb{R} $ une fonction de classe $ C^1 $.  On suppose que $J(0)=0$ et que:
\begin{enumerate}
  \item il existe $$\rho>0\ \text{et}\ r>0\ \text{tels que si}\ \|u\|=\rho,\ J(u)\geq r,$$
  \item il existe $$e\in X\ \text{tel que}\ \|e\|>\rho,\ \text{et}\ J(e)\leq0.$$
\end{enumerate}
Alors, $ J $ poss\`ede une suite de Palais-Smale au niveau $$ c = \inf_{\gamma \in \Gamma} \max_{0 \leq t \leq 1} J( \gamma(t)) \geq r, $$
où $$ \Gamma= \left\{ \gamma \in \mathcal{C}([0,1],X),\ \gamma(0) = 0,\ \gamma(1) = e\right\}. $$
\end{thm}

\end{document}